\documentclass{amsart}
\usepackage{amsmath,amsthm}
\usepackage{amsfonts,amssymb}
\usepackage{accents}
\usepackage{enumerate}
\usepackage{accents,color}
\usepackage{graphicx}
\usepackage{comment}
\usepackage{hyperref}

\newcommand{\lvt}{\left|\kern-1.35pt\left|\kern-1.3pt\left|}
\newcommand{\rvt}{\right|\kern-1.3pt\right|\kern-1.35pt\right|}

\newtheorem{thm}{Theorem}[section]
\newtheorem{cor}[thm]{Corollary}
\newtheorem{lem}[thm]{Lemma}
\newtheorem{prop}[thm]{Proposition}

\newtheorem{defn}[thm]{Definition}
\usepackage{bm}      
       
\theoremstyle{remark}
\newtheorem{rem}{Remark}[section]

 \def\la{{\langle}}
 \def\ra{{\rangle}}

 \def\d{\mathrm{d}}
 \def\e{\mathrm{e}}

 \def\sph{{\mathbb{S}^{d-1}}}
 
 \def\sb{{\mathsf b}}

  \def\sh{{\mathsf h}}
  \def\sw{{\mathsf w}}

 \def\sQ{{\mathsf Q}}

 \def\sW{{\mathsf W}}
 \def\sY{{\mathsf Y}}

 \def\fD{{\mathfrak D}}
 \def\fE{{\mathfrak E}}
 \def\fF{{\mathfrak F}}

 \def\fR{{\mathfrak R}}
  \def\fL{{\mathfrak L}}

 \def\a{{\alpha}}
 \def\b{{\beta}}
 \def\g{{\gamma}}
 \def\k{{\kappa}}
 \def\t{{\theta}}
 \def\l{{\lambda}}
 \def\o{{\omega}}
 \def\s{\sigma}
 \def\la{{\langle}}
 \def\ra{{\rangle}}

 \def\bb{{\bm b}}

 \def\jb{{\bm j}}
 \def\hb{{\bm h}}
 \def\kb{{\bm k}}

 \def\ub{{\bm u}}
 \def\vb{{\bm v}}
 \def\xb{{\bm x}}
 \def\yb{{\bm y}}
 
  \def\Bb{{\bm B}}

 \def\Hb{{\bm H}}

 \def\Lb{{\bm L}}
 \def\Pb{{\bm P}}
 \def\Qb{{\bm Q}}
 \def\Rb{{\bm R}}
 
 \def\Tb{{\bm T}}

 \def\Wb{{\bm W}}
 
 \def\Yb{{\bm Y}}

 \def\CH{{\mathcal H}}

 \def\CV{{\mathcal V}}

 \def\BB{{\mathbb B}}

 \def\NN{{\mathbb N}}

 \def\RR{{\mathbb R}}
 \def\SS{{\mathbb S}}
 
 \def\VV{{\mathbb V}}

      \def\proj{\operatorname{proj}}

\def\lla{\langle{\kern-2.5pt}\langle}      
\def\rra{\rangle{\kern-2.5pt}\rangle}

\def\bep{{\boldsymbol{\epsilon}}}
\def\bg{{\boldsymbol{\gamma}}}
\def\bk{{\boldsymbol{\kappa}}}

\def\f{\frac}

\graphicspath{{./}}
\begin{document}
 
\title[Orthogonal polynomials as eigenfunctions of a differential operator]
{Orthogonal polynomials which are eigenfunctions of a partial differential operator}
\author{Yuan~Xu}
\address{Department of Mathematics, University of Oregon, Eugene, OR 97403--1222, USA}
\email{yuan@uoregon.edu}
\date{\today}

\begin{abstract}
We study orthogonal polynomials of $d = d_1+d_2$ variables with respect to a wrapped 
product weight function $\Wb(\bm x,\bm y) = W_1(\bm x/\rho(\bm y)) W_2(\bm y)$ for $(\bm x, \bm y) \in 
\mathbb{R}^{d_1} \times \mathbb{R}^{d_2}$, where $\rho$ is 
either linear or the square root of a nonnegative quadratic polynomial, and identify all such polynomials that 
are eigenfunctions of a second-order linear differential operator. For $d =2$, it is known that there are primarily, 
up to affine transformations, five families of such polynomials, which are products or wrapped products of
classical orthogonal polynomials of one variable; all five families have their counterparts in higher dimensions, 
but no characterization is known in dimension three or higher. Our study explores viable wrapped product families,
finds explicit second-order differential operators for two new families of orthogonal polynomials in $d= d_1+d_2$ 
variables that have not been studied before when either $d_1>1$ or $d_2 > 1$, and provides, in particular, a complete 
list of such operators among all wrapped product orthogonal polynomials when $d = 3$. The list also includes 
four families that are eigenfunctions of a fourth-order differential operator, whereas no second-order operator is 
available. Moreover, orthogonal polynomials on the corresponding wrapped quadratic surfaces that are 
eigenfunctions of a second-order differential operator on the surface are also studied. 
\end{abstract}

\maketitle

\section{Introduction}
\setcounter{equation}{0}

Classical orthogonal polynomials on the real line consist of three families, Hermite, Laguerre, and Jacobi 
polynomials, and they are characterized as the only ones on the real line that are eigenfunctions of a 
second-order linear differential operator with polynomial coefficients \cite{B}. An extension of such a 
characterization is given in \cite{KrSh} for orthogonal polynomials on a domain of two variables, which 
shows that there are primarily, up to affine transformations, five regular families of orthogonal polynomials 
that are eigenfunctions of a second-order linear differential operator, called the {\it spectral operator}, with 
their eigenvalues depending only, apart from fixed parameters, on the total degree of orthogonal polynomials. 
The characterization has inspired further work on classification in two variables, some under weaker 
conditions, such as linear functionals instead of inner products, or orthogonality in a weaker sense; see, 
for example, \cite{AGRZ, FPP, KKL, KLL, Su} and references there. Of the primary five, three are from 
products of Hermite and/or Laguerre polynomials of one variable, the other two are orthogonal polynomials 
on the disk and on the triangle. All five primary families can be straightforwardly extended to more than two 
variables. There is, however, no classification for three or more variables. Moreover, besides products of 
Hermite and/or Laguerre polynomials, orthogonal polynomials on the unit ball and the simplex have been 
the sole examples for spectral operators in higher dimensions until two new families of orthogonal polynomials 
on rotational cones that possess spectral operators were revealed fairly recently in \cite{X21}.

Spectral operators reveal deep intrinsic characteristics of orthogonal structure, and they provide powerful 
tools, when they exist, such as defining fractional derivatives in Sobolev spaces or the $K$-functional for 
interpolation spaces that can be used to characterize the best approximation by polynomials (cf. \cite{LR, X21, X21b}). 
Orthogonal polynomials that possess a spectral operator enjoy rich structures that permit in-depth study of 
a variety of problems in high dimensions (cf. \cite{DaiX}). One naturally asks how many families of such 
polynomials there are, a question that we aim to provide a tentative answer to. Our study is motivated by 
the realization that, in all known cases when a spectral operator exists, the orthogonality is defined in 
terms of the weight function 
$$
   \Wb(\xb,\yb) = W_1\left(\frac{\xb}{\rho(\yb)}\right) W_2(\yb), \quad (\xb, \yb) \in \Omega_1 \rtimes \Omega_2 = \left \{(\xb,\yb): \frac{\xb}{\rho(\yb)} \in \Omega_1, \, \yb \in \Omega_2\right\},
$$
where $\rho$ is either a nonnegative linear polynomial or the square root of a nonnegative 
quadratic polynomial on $\Omega_2$, and $\Omega_1$ and $\Omega_2$ are domains in $\RR^{d_1}$
and $\RR^{d_2}$, respectively; moreover, one basis of orthogonal polynomials in such a setup consists of 
polynomials of the form 
$$
   Q_{\jb,\kb,m}^n(\xb,\yb) = R_{\jb, n-m}^{(m)} (\yb) [\rho(\yb)]^m P_{\kb, m}\bigg(\frac{\xb}{\rho(\yb)}\bigg), 
$$
where $P_{\kb, m}$ and $R_{\jb, n-m}^{(m)}$ are orthogonal polynomials with respect to $W_1$ on $\Omega_1$
and $\rho^{2m+d_1} W_2$, for each $m$, on $\Omega_2$, respectively. When $\rho(\yb) = 1$, these are 
tensor product orthogonal polynomials. For nonconstant $\rho$, we refer to them as {\it wrapped product}, and
they are orthogonal polynomials in $d_1 + d_2$ variables. 

The purpose of the present paper is to study wrapped product orthogonal polynomials and determine
those families that possess an explicitly given second-order partial differential operator that has orthogonal 
polynomials as eigenfunctions. For a family that does, it is necessary that both its components, orthogonal 
polynomials with respect to $W_1$ and $W_2$, have to possess their own spectral operator. This heritage makes it 
possible to determine all possible cases, at least in dimension three, by starting from the known classification
for $d=2$. Our result yields, besides known cases of product Hermite and/or Laguerre polynomials and those 
on the unit ball and the simplex, two more families of orthogonal polynomials that possess a second-order 
spectral operator, which are wrapped products of classical orthogonal polynomials on the unit ball 
with either classical Jacobi polynomials on the simplex or product Laguerre polynomials. For $d_2 =1$, 
these two cases become a variant of the families on the rotational cone studied in \cite{X20}, and they are 
new and have not been considered if either $d_1 > 1$ or $d_2  >1$. Among other viable wrapped product 
cases, some are affine equivalent to known cases, and others do not possess a spectral operator of order 
two, which is not surprising. However, four families of orthogonal polynomials are shown to possess 
a spectral operator of order four, which appears to be new even in two variables. 

For $d_1 > 1$, we can also consider a wrapped product of orthogonal polynomials on the quadratic surface
of $\Omega_1\!\rtimes \Omega_2$. The most well-known orthogonal polynomials on the surface are 
spherical harmonics on the unit sphere, which are eigenfunctions of the Laplace-Beltrami operator. 
Those on other rotational quadratic surfaces are studied more recently in \cite{X20, X21, X21b, X26}, 
including two families on rotational cones that are shown to be eigenfunctions of a second-order linear 
differential operator. The latter two can be extended to orthogonal polynomials on the wrapped product surfaces, 
which are wrapped products of spherical harmonics with either classical Jacobi polynomials on the simplex or
Laguerre polynomials. 

This study introduces new classes of multivariate orthogonal polynomials on domains that have not been 
previously studied, and uncovers new spectral operators. The latter will likely be useful for an eventual 
classification of spectral operators in higher dimensions. In fact, taking a cue from $d =2$, we speculate 
that our study has resulted in a complete list of spectral operators, up to affine transformation, at least in 
three variables. 
 
The paper is organized as follows. The next section reviews preliminaries on orthogonal polynomials of 
several variables, where we review necessary results on classical orthogonal polynomials needed for later 
sections. The wrapped products of weight functions and orthogonal polynomials are defined and studied in 
the third section. The two new families that possess a second-order spectral operator are studied in the fourth
and the fifth sections, respectively. The four families with fourth-order spectral operators are collected in the
sixth section. Finally, orthogonal polynomials on two wrapped product surfaces are studied in the 
seventh section. We also include an appendix that lists all spectral operators that are known in 
dimension three. 

\section{Orthogonal polynomials of several variables}
\setcounter{equation}{0}

The basics of classical orthogonal polynomials in one and two variables are recalled in the first subsection, 
followed by classical orthogonal polynomials of higher dimensions in the second subsection, which will serve 
as building blocks in our study of new orthogonal polynomials. The third subsection provides a prelude for 
the wrapped product to be studied in the next section. 

\subsection{Classical orthogonal polynomials} 
Let $\Omega$ be a set in $\RR^d$ with non-empty interior. Let $W$ be a non-negative weight function defined 
on $\Omega$, 
so that the bilinear form 
$$
    \la f,g\ra_W = \int_{\Omega} f(\bm x) g(\bm x) W(\bm x) \d \bm x
$$
is a well-defined inner product, which implies that orthogonal polynomials under the inner product consist
of an orthogonal basis in $L^2(\Omega, W)$. Let $\CV_n^d(\Omega, W)$ be the space of orthogonal 
polynomials of degree $n$ under this inner product. Then 
\begin{equation}\label{eq:dimVn}
  \dim \CV_n^d(\Omega, W) = \binom{n+d-1}{n}, \qquad n = 0,1,2, \ldots. 
\end{equation}
 
We are interested in the case when orthogonal polynomials are eigenfunctions of a linear differential operator
$\fD$ with polynomial coefficients; that is, 
$$
  \fD u =  - \l_n u, \qquad u \in \CV_n^d(\Omega, W), \quad n=0, 1, 2, \ldots,
$$
where $\l_n \in \RR$ are eigenvalues. We call such an operator a {\it spectral operator}. 

For $d =1$, there are exactly three families of such polynomials as classified in \cite{B}, which are 
called classical orthogonal polynomials on the real line,
\begin{itemize}
\item Hermite polynomials $u = H_n$ for $\sw(t) = e^{- t^2}$ defined on $\RR$, 
$$
      \fD^H u :=  u_t^2 - 2 t u_t = - 2 n u. 
$$
\item Laguerre polynomials $u= L_n^\a$ for $\sw_\a (t) = t^\a e^{- t}$, $\a > -1$, defined on $\RR+$, 
$$
       \fD^L_\a u := t u_t^2 + (\a + 1 - t) u_t = - n u.  
$$
\item Jacobi polynomials $u= P_n^{(\a,\b)}$ for $\sw_{\a,\b} (t) = (1-t)^\a(1+t)^\b$, $\a, \b  > -1$, defined on $[-1,1]$, 
$$
      \fD^J_{\a,\b} u := (1-t^2) u_t^2 - [\a - \b+(\a+\b+2) t] u_t = - n(n+\a + \b+1)u.
$$
\end{itemize}

For $d =2$, the simplest families of orthogonal polynomials of two variables are those for the product 
weight functions. Let $w_i$ be a weight function on the interval $I_i$ of the real line. Let $p_m(w_i)$ be 
the orthogonal polynomial of degree $m$ with respect to $w_i$ on $I_i$. Define 
$W(x,y) = w_1(x) w_2(y)$ on $I_1 \times I_2$. Then the product orthogonal polynomials defined by 
$$
    P_{m,n}(x,y) = p_m(w_1; x) p_{n-m}(w_2; y), \quad 0 \le m \le n,
$$ 
consist of an orthogonal basis for $\CV_n(I_1 \times I_2, W)$. It follows, in particular, that the product
Hermite and/or Laguerre polynomials possess spectral operators: 

\begin{itemize}
\item Hermite-Hermite $u = H_m(x) H_{n-m}(y)$ for $\Wb(x,y) = e^{- x^2-y^2}$, 
$$
    \left( \fD^{H, (x)}  + \fD^{H, (y)} \right) u  = - 2 n u, \quad u \in \CV_n\big(\RR^2, \Wb\big).
$$
\item Hermite-Laguerre $u= H_m(x) L_{n-m}^\a(y)$ for $\Wb_\a (x,y) = y^\a e^{- y} e^{-x^2}$, 
$$
       \left(  \fD^{H, (x)} +2\,  \fD_\a^{L, (y)}\right) u = - 2 n u, \quad u \in \CV_n\big(\RR \times \RR_+, \Wb_\a \big).
$$
\item Laguerre-Laguerre $u= L_m^\a(x) L_{n-m}^\b(y)$ for $\Wb_{\a,\b} (x,y) = x^\a y^\b \e^{-x-y}$, 
$$
        \left(  \fD_\a^{L, (x)}  +  \fD_\b^{L, (y)} \right) u = - n u, \quad u \in \CV_n\big(\RR_+ \times \RR_+, \Wb_{\a,\b} \big),
$$
\end{itemize}
where we include a superscript in the notation of the operator, such as $(x)$ in $\fD^{H, (x)}$ to indicate
the variables that the operator acts on. We note, however, that the products of Jacobi-Hermite or 
Jacobi-Laguerre polynomials do not possess a spectral operator, nor do products of Jacobi-Jacobi, 
since the eigenvalues for the two components that make up the products are not both linear to be
additive. 

The spectral operator also exists in two non-product cases. The first one is on the unit disk 
$\BB^2 =\{(x,y): x^2+y^2 \le 1\}$ with the Gegenbauer weight function
$$
  \Wb_\mu(x,y) = (1-x^2-y^2)^{\mu-\frac12}, \quad \mu > -\tfrac12. 
$$
An orthogonal basis for $\CV_n(\BB^2, \Wb_\mu)$ can be given via the Gegennbauer polynomial 
$C_n^\mu$, which is a constant multiple of the Jacobi polynomials $P_n^{(\mu-\f12, \mu-\f12)}$, 
\begin{equation} \label{eq:OP_BB2}
  \Bb_{m,n}^\mu(x,y) = C_{n-m}^{\mu+m +\f12}(y) (1-y^2)^{\f m 2} C_m^\mu\bigg(\frac{x}{\sqrt{1-y^2}}\bigg),
  \quad 0 \le m \le n.
\end{equation}
The eigenvalue equation of the spectral operator for these polynomials is given by 
\begin{itemize}
\item Gegenbauer polynomials $u = \Bb_{m,n}^\mu$ on the disk $\BB^2$ for $\Wb_\mu$, 
\begin{align*}
 \fD_\mu^\BB u := \,& (1-x^2) u_{x x} - 2xy u_{xy} + (1-y^2) u_{yy} - 2(\mu+1) (x u_x +  y u_ y) \\
    = \, & - n(n+2\mu +1)u, \qquad u \in \CV_n(\BB^2,  \Wb_\mu).
\end{align*}
\end{itemize}
The second one is on the triangle $\{(x,y): x \ge 0, \, y\ge 0, x+y \le 1\}$ with the Jacobi weight
$$
     \Wb_{\a,\b,\g}(x,y) = x^\a y ^\b(1-x-y)^\g, \qquad \a, \b, \g > -1.
$$
An orthogonal basis for $\CV_n(\triangle^2, \Wb_{\a,\b,\g})$ can be given via the 
Jacobi polynomials,
\begin{equation} \label{eq:OP_TT2}
  \Tb_{m,n}^{\a,\b,\g}(x,y) = P_{n-m}^{(\a+\g+2m+1,\b)}(2y-1) (1-y)^m
     P_m^{(\g, \a)}\bigg(\frac{2x}{1-y} -1\bigg),
  \quad 0 \le m \le n.
\end{equation}
The eigenvalue equation of the spectral operator for these polynomials is given by 
\begin{itemize}
\item Jacobi polynomials $u =  \Tb_{m,n}^{\a,\b,\g}$ on the triangle $\triangle^2$ for $\Wb_{\a,\b,\g}$,
\begin{align*}
 \fD_{\a,\b\g}^\triangle u := \,& x(1-x) u_{x x} - 2xy u_{xy} + y(1-y) u_{yy} \\
    &  +(\a+1 - (\a+\b+\g+3)x)u_x +  (\b+1 - (\a+\b+\g+3)y )u_ y \\
    =  \, &  - n(n+\a+\b+\g+2)u, \qquad u \in \CV_n(\triangle^2, \Wb_{\a,\b,\g}).
\end{align*}
\end{itemize}
Both these latter cases are classical, attributed to Hermite, and appeared already in the 
classical book \cite{AK}. 

We call these five families {\it classical orthogonal polynomials} of two variables. Up to affine transformations, 
they are the only ones that are eigenfunctions of a second-order differential operator for $d =2$ according to 
the classification in \cite{KrSh}. 

We note that the above classification holds under the assumption that inner products used to define orthogonality 
are positive definite, which holds under our setup for which the weight functions are assumed to be nonnegative
and the inner product is well-defined for all polynomials. If this assumption is relaxed to allow orthogonality under 
signed bilinear forms, then there are four more families of orthogonal polynomials, 
making it a total of nine families, in the classification of \cite{KrSh}. 

\subsection{Classical orthogonal polynomials of several variables}
For $d > 2$, no classification of spectral operators is known. All five cases for $d =2$ can be extended to 
$d >2$. They lead to four primary families of orthogonal polynomials that will serve as building blocks for 
the families of wrapped product orthogonal polynomials that we will study. We review these 
cases in this section for later reference. 

\subsubsection{Gegenbauer polynomials on the unit ball $\BB^d$}
These polynomials are an extension of orthogonal polynomials on the unit disk, and they are orthogonal 
with respect to the inner product 
\begin{equation}\label{eq:B-ipd}
  \la f, g\ra_\mu^\BB = \bm b_\mu^\BB \int_{\BB^d} f(\bm x) g(\bm x) \bm W_\mu^\BB (\bm x) \d  \bm x,
\end{equation}
where the weight function is defined by 
\begin{equation}\label{eq:W_ball}
  \bm W^\BB_\mu(\bm x) = (1-\|\bm x\|^2)^{\mu-\f12}, \qquad \mu > -\tfrac12,
\end{equation}
in which the parameter is chosen so that it resembles the Gegenbauer weight on $[-1,1]$ and the 
constant $\bm b_\mu^\BB$ is given by 
\begin{equation}\label{eq:bB}
\bm b_\mu^\BB = \frac{\Gamma(\mu+\f{d+1}{2})}{\pi^{\f d 2}\Gamma(\mu+\f12)},
\end{equation}
which is the normalization constant so that $\la 1, 1\ra_\mu^\BB =1$. 

Let $\CV_n\left(\BB^d, \bm W_\mu^\BB\right)$ be the space of orthogonal polynomials of degree $n$ with respect
to $\bm W^\BB_\mu$. An orthogonal basis for this space can be given explicitly in terms of the Gegenbauer polynomials
$C_n^\lambda$ in Cartesian coordinates, where $\lambda \ge 0$ with the case $\lambda = 0$ holds under limit. For 
$\bm x \in \RR^d$, define $\bm x_0 =0$ and $\bm x_j = (x_1,\ldots, x_j)$ for $1 \le j \le d$ and for $\kb \in \NN_0^d$, define 
$\kb^j = (k_j, \ldots, k_d)$ for $1 \le j \le d$ and $\kb^{d+1} =0$. For $\kb \in \NN_0^d$ with $|\kb| = n$, let 
\begin{equation}\label{eq:OP_BB}
 \bm B_{\kb, n}^\mu(\bm x) = \prod_{j=1}^d \left(1-\|\bm x_{j-1}\|^2\right)^{\frac{k_j}{2}} C_{k_j}^{ \mu+ |\kb^{j+1}| + \frac{d-j}{2}} \bigg(\frac{x_j}{\sqrt{1-\| \bm x_{j-1} \|^2}} \bigg). 
\end{equation}
Then $\{\bm B_{\kb,n}^\mu: |\kb|=n, \, \kb \in \NN_0^d\}$ is an orthogonal basis of 
$\CV_n\left(\BB^d, \bm W_\mu^\BB\right)$; see \cite[p. 143]{DX} for this basis as well as the formula 
for the $L^2$ norm of $\bm B_{\kb, n}^\mu$. Another explicit orthogonal basis can be given in terms of 
the Jacobi polynomials and spherical harmonics; see \cite[p. 142]{DX}. 
 
The spectral operator $\fD_\mu^\BB$ that has $\CV_n\left(\BB^d, \bm W_\mu^\BB\right)$ as its eigenspace is given by
\begin{equation} \label{eq:Diff-ball}
\fD_\mu^{\BB} = \fD_\mu^{\BB,d} :=  \Delta_{\bm x} - \la \bm x, \nabla_{\bm x}\ra ^2 - (2\mu + d-1) \la \bm x, \nabla_{\bm x} \ra  
\end{equation}
in a more succinct notation than the one stated for $d = 2$, and the eigenvalue equation with eigenvalue is given by 
\begin{equation} \label{eq:D-ball}
\fD_\mu^{\BB,d}\, u = -n(n+2\mu+d-1) u, \qquad \forall u \in \CV_n\! \left(\BB^d, \bm W_\mu^\BB\right).
\end{equation} 

The operator $\fD_\mu^\BB$ is self-adjoint in $L^2(\bm W_\mu^\BB, \BB^d)$ and it can be further written in a 
more structured format. Let $\partial_{x_i} = \frac{\partial}{\partial x_i}$ and let $D_{i,j}$ be the angular derivatives 
defined by 
\begin{equation} \label{eq:Dij}
  D_{i,j} = x_i \partial_{x_j} - x_j \partial_{x_i} = \frac{\partial} {\partial \t_{i,j}} \qquad 1 \le i, j \le d,
\end{equation}
where $\t_{i,j}$ is the angle of the polar coordinates $(x_i, x_j) = r_{i,j} (\cos \t_{i,j}, \sin \t_{i,j})$ of the $(x_i,x_j)$ plane. 
Then the operator $\fD_\mu^\BB$ can be written as 
\begin{equation} \label{eq:Diff-ball2}
\fD_\mu^\BB =  \frac1 { \bm W_\mu^\BB(\bm x) }  \sum_{i=1}^d \partial_i \left( \bm W_{\mu+1}^\BB (\bm x) \partial_i \right) 
+ \sum_{1 \le i< j \le d} D_{i,j}^2.
\end{equation}
The angular derivatives are skew-adjoint in $L^2(\sph)$, and hence, in $L^2(\BB^d, \Wb_\mu)$. Hence, integration 
by parts leads to 
\begin{align} \label{eq:Diff-ball3}
- \int_{\BB^d} \fD_\mu^\BB f(\bm x) \cdot g(\bm x) \bm W_\mu^\BB (\bm x) \d \bm x 
   & =  \sum_{i=1}^d \int_{\BB^d}\partial_i f(\bm x) \partial_i g(\bm x) \bm W_{\mu+1}^\BB (\bm x) \d \bm x\\
   &   + \sum_{1 \le i< j \le d}  \int_{\BB^d} D_{i,j} f(\bm x) D_{i,j} g(\bm x) \bm W_{\mu}^\BB (\bm x)\d \bm x, \notag
\end{align}
so that $\fD_\mu^\BB$ is self-adjoint in $L^2(\BB^d, \Wb_\mu^\BB)$. It is worthwhile to point out that the 
identity \eqref{eq:Diff-ball2} is not the only decomposition of $\fD_\mu^\BB$, as there are others as 
shown in \cite{BX}. 

\subsubsection{Jacobi polynomials on the simplex $\triangle^d$}
These polynomials are an extension of the Jacobi polynomials on the triangle. The $d$-dimensional simplex 
is defined by 
$$
  \triangle^d = \{\bm y \in \RR^d: y_1 \ge 0, \ldots, y_d \ge 0, |\bm y| <1\}, 
$$
where $|\bm y| = y_1+\cdots + y_d$, a notation that will be used throughout this work. The Jacobi polynomials on
$\triangle^d$ are orthogonal with respect to the inner product 
\begin{equation}\label{eq:T-ipd}
  \la f, g\ra_\bk^\triangle = \bm b_\bk^\triangle \int_{\triangle^d} f(\bm x) g(\bm x) \bm W_\bk^\triangle (\bm x) \d  \bm x,
\end{equation}
where the weight function is the Jacobi weight of several variables defined, for $\bk \in \RR^{d+1}$
with $\k_i > -1$, $1 \le i \le d+1$, by 
\begin{equation}\label{eq:W_simplex}
\bm W_\bk^\triangle(\bm y) = \prod_{i=1}^d y_i^{\k_i} (1-|\bm y|)^{\k_{d+1}}, \quad \bm y \in \triangle^{d},
\end{equation}
and the constant $\bm b_\bk^\triangle$ is given by
\begin{equation}\label{eq:bT}
  \bm b_\bk^\triangle =  \frac{\Gamma(|\bk| + d+1)}{\prod_{i=1}^{d+1} \Gamma (\k_i+1)},
\end{equation}
which is the normalization constant so that $\la 1,1 \ra_{\bk}^\triangle =1$. 

Let $\CV_n\left(\triangle^d, \bm W_\bk^\triangle\right)$ be the space of orthogonal polynomials of degree $n$ 
with respect to $\bm W^\triangle_\bk$. An orthogonal basis for this space can be given explicitly in terms of the 
Jacobi polynomials $P_n^{(\a,\b)}$. For $\bm y \in \RR^d$, we again define $\bm y_0 =0$ and 
$\bm y_j = (y_1,\ldots, y_j)$ for $1 \le j \le d$ and for $\bk \in \RR^{d+1}$, define $\bk^j = (\k_j, \ldots, \k_{d+1})$ 
for $1 \le j \le d$, and define likewise $\kb^j$ for $\kb \in \NN_0^d$ but with $\kb^{d+1} =0$. 
We now define 
\begin{equation}\label{eq:OP_TT}
 \bm T_{\kb, n}^\bk(\bm y) = \prod_{j=1}^d \left( 1-|\bm y_{j-1}| \right)^{k_j} 
     P_{k_j}^{(|\bk^{j+1}| +2 |\kb^{j+1}| + d - j, \k_j)} \bigg(\frac{2 y_j}{1-| \bm y_{j-1}|}-1\bigg).
\end{equation}
Then $\left \{ \bm T_{\kb, n}^\bk: |\kb| =n, \, \kb \in \NN_0^d\right\}$ is an orthogonal basis for 
$\CV_n^d(\triangle^d, \bm W_\bk^\triangle)$; see \cite[p. 150]{DX} for this basis as well as the 
formula for the $L^2$ norm of $\bm T_{\kb, n}^\bk$ in \cite[p. 151]{DX}. 
 
The spectral operator $\fD_\bk^\triangle$ that has $\CV_n\left(\triangle^d, \bm W_\bk^\triangle\right)$ as its eigenspace is given by
\begin{align}\label{eq:D-simplex}
  \fD_\bk^\triangle \, =  \sum_{i=1}^d y_i(1-y_i)  \partial_{y_i}^2 
    & - 2 \sum_{1 \le i< j \le d} y_i y_j  \partial_{y_i}\partial_{y_j} \\
    &  + \sum_{i=1}^d \big(\k_i +1 - (|\bk|+d+1) y_i \big)  \partial_{y_i},\notag
\end{align} 
and the eigenvalue equation is given by \cite[Section 5.3]{DX}, 
\begin{equation}\label{eq:Dk-eigen}
   \fD_\bk^\triangle\, u = - n (n+|\bk| + d) u, \qquad u \in \CV_n\left(\triangle^d, \bm W_\bk^\triangle \right).
\end{equation}
The operator $\fD_\mu^\triangle$ is self-adjoint in $L^2(\triangle^d, \bm W_\mu^\triangle)$ and it can be further written as
 \begin{align}\label{eq:DkSimplex2}
\fD_\bk^\triangle = \frac{1}{\bm W_\bk^\triangle(\bm y)} \left[
   \sum_{i=1}^d \partial_{y_i} \left( y_i (1-|\bm y|) \bm W_\bk^\triangle(\bm y) \partial_{y_i}\right)
   + \sum_{1\le i< j \le d} \partial_{i,j} \left( y_i y_j \bm W_{\bk}^\triangle(\bm y) \partial_{i,j} \right) 
\right] 
\end{align}
where $\partial_{i,j} = \partial_{y_i} - \partial_{y_j}$, which leads immediately to the identity
 \begin{align}\label{eq:DkSimplex3}
   - \int_{\triangle^d} \fD_\bk^\triangle f(\bm y) \cdot g(\bm y) \bm W_\bk^\triangle(\bm y) \d \bm y 
& =    \sum_{i=1}^d   \int_{\triangle^d} y_i (1-|\bm y|) \partial_{y_i}  f(\bm y) \partial_{y_i}  g(\bm y) 
  \bm W_\bk^\triangle(\bm y) \d \bm y \\
 &  + \sum_{1\le i< j \le d}  \int_{\triangle^d} y_i y_j  \partial_{i,j} f(\bm y) \partial_{i,j} g(\bm y) \bm W_\bk^\triangle(\bm y) \d \bm y, \notag 
\end{align}
which proves that $\fD_\bk^\triangle$ is self-adjoint in $L^2(\triangle^d, \Wb_\bk^\BB)$. 
It is worth mentioning that there are several other decompositions of $\fD_\bk^\triangle$, and each leads to an integral
identity showing that $\fD_\bk^\triangle$ is self-adjoint in a weighted $L^2$ space different from 
$L^2(\triangle^d, \bm W_\bk^\triangle)$ (\cite{GX, X23b}).  
 
 \subsubsection{Product Laguerre polynomials}
These polynomials are orthogonal with respect to the product Laguerre inner product defined by
$$
  \la f, g\ra_{\bk}^L = \bb_\k^L \int_{\RR_+^d} f(\xb) g(\xb) \Wb_\bk^L (\xb) \d \xb
$$
where the weight function $\Wb_\bk^L$ is defined, for $\bk \in \RR^d$ with $\k_i > -1$, $1 \le i \le d$, by 
\begin{equation}\label{eq:W_L}
    \Wb_\bk^L(\xb) = \prod_{i=1}^d x_i^{\k_i} \e^{- |\xb|}, \qquad \xb \in \RR_+^d,
\end{equation}
and $\bb_\bk^L$ is the normalization constant defined by $\la 1,1\ra_\bk^L = 1$, which satisfies
$$
   \bb_\bk^L = \frac{1}{ \Gamma(\k_1+1) \cdots \Gamma(\k_d+1)}. 
$$
An orthogonal basis for $\CV_n(\RR_+^d, \Wb_\bk^L)$ is given by the product Laguerre polynomials 
\begin{equation}\label{eq:prod_L}
    \Lb_{\jb}^\bk (\xb) = L_{j_1}^{\k_1}(x_1) \cdots  L_{j_d}^{\k_d}(x_d), \qquad |\jb|= n. 
\end{equation}
These polynomials are the eigenfunctions of the differential operator $\fD_\bk^L$,
\begin{equation}\label{eq:diff_prod_L}
   \fD_\bk^L u : = \sum_{i=1}^d \left[x_i \partial_i^2 + (\k_i + 1 -x_i) \partial_i \right] u   = - n u, 
    \quad u \in \CV_n\left(\RR_+^d, \bm W_\bk^L\right).  
\end{equation}
Moreover, the differential operator $\fD_\bk^L$ can be written in a more structured form,
\begin{equation}\label{eq:diff_prod_L2}
   \fD_\bk^L = \frac{1}{\Wb_\bk^L(\xb)}  \sum_{i=1}^d  \partial_i \big(x_i \Wb_\bk^L (\xb) \partial_i\big),
\end{equation}
which can be used to show that the operator is self-adjoint in $L^2(\RR_+^d, \Wb_\bk)$.

We shall need a variant of product Laguerre polynomials that are defined on the domain $\RR_\Sigma^d$ defined by
$$
  \RR_\Sigma^d := \left \{\yb \in \RR^d: y_1\ge 0, \ldots, y_{d-1} \ge 0, |\yb| \ge 0\right\}.
$$
Throughout this paper, we shall use the notation, for $ \yb \in \RR^d$, 
$$
   \yb = (\yb', y_d), \quad \hbox{with} \quad \yb' = (y_1,\ldots, y_{d-1}). 
$$
Then the condition $|\yb| \ge 0$ is equivalent to $y_d \ge - |\yb '|$, so that $y_d$ could be negative and $\RR_\Sigma^d$ is not 
$\RR_+^d$. For $\bk \in \RR^d$ with $\k_i > -1$, $1 \le i \le d-1$ and $\k_d > - 1$, let 
\begin{equation} \label{eq:LaguerreSig}
   \Wb_\bk^\Sigma(\yb) = \prod_{i=1}^{d-1} y_i^{\k_i} |\yb|^{\k_d} \e^{-|\yb|-|\yb'|}. 
\end{equation}

\begin{prop} \label{prop:fD-L}
An orthogonal basis of $\CV_n\big(\RR_\Sigma^d, \Wb_\bk^\Sigma\big)$ is given by 
\begin{equation} \label{eq:basis_S}
  \hat \Lb_{\jb,m}^\bk(\yb) := 
      \Lb_{\jb,m}^\bk (y_1,\ldots, y_{d-1}, |\yb|), \qquad |\jb|= m.
\end{equation}
Moreover, the polynomials in $\CV_n\big(\RR_\Sigma^d, \Wb_\bk^\Sigma\big)$ are eigenfunctions of $\fD_\bk^\Sigma$ defined by
\begin{align}\label{eq:fDSig}
\fD_\bk^\Sigma = \sum_{i=1}^d \left(y_i \partial_i^2 + (\k_i + 1 - y_i) \partial_i \right)  \,& + 2|\yb| \partial_{y_d}^2   - 2 \la \yb, \nabla_\yb \ra \partial_{y_d}\\
  & - (|\bk|-\k_d +d-1) \partial_{y_d},\notag
\end{align} 
and the eigenvalue equation is given by 
\begin{equation}\label{eq:fD-L}
   \fD_\bk^\Sigma \, u = - n u, \qquad \forall u \in \CV_n\big(\RR_\Sigma^d, \Wb_\bk^\Sigma\big). 
\end{equation}
\end{prop}

\begin{proof}
 We make a change of variables $\yb \in \RR_\Sigma^d \mapsto \xb \in \RR_+^d$ defined by 
 $$
        y_1 = x_1, \ldots, y_{d-1} = x_{d-1}, \quad y_d = x_d - |\xb'|,
 $$
which is invertible with $x_d = |\yb'| + y_d = |\yb|$, so that $|\xb| = |\yb'| + |\yb|$. It follows readily that 
$\Wb_\bk^\Sigma(\yb) =  \Wb_\bk^L(\xb)$ and $\hat \Lb_\jb^\bk(\yb) = L_{j_1}^{\k_1}(x_1) \cdots L_{j_d}^{\k_1}(x_d)
= \Lb_\jb^\bk(\xb)$, 
from which the orthogonality of $\hat \Lb_\jb^\bk$ follows from that of the product Laguerre polynomials. 

The change of variables $\yb \mapsto \xb$ leads to, with $y_d = x_d - |\xb'|$, 
$$
 \partial_{x_j} u(\yb) = \left(\partial_{y_j} - \partial_{y_d}\right) u(\yb), 1 \le j \le d-1, \quad \hbox{and} \quad 
  \partial_{x_d} u(\yb) = \partial_{y_d} u(\yb).
$$ 
Using these relations, a straightforward verification shows that $\fD_\kb^L$ in the $\xb$ variable becomes 
$\fD_\bk^{\Sigma}$ in $\yb$ variable under the change of variables, which verifies \eqref{eq:fD-L}. 
\end{proof}

As an analog of \eqref{eq:diff_prod_L2}, we can also write the operator $ \fD_\bk^\Sigma$ in a more structured form.
\begin{align}\label{eq:fD-L2}
   \fD_\bk^\Sigma  =\, &  \frac{1}{\Wb_\bk^\Sigma(\yb)}  \sum_{i=1}^{d-1} (\partial_{y_i} - \partial_{y_d}) y_i \Wb_\bk(\yb)  
   (\partial_{y_i} - \partial_{y_d})  
    + \frac{1}{\Wb_\bk^\Sigma(\yb)} \partial_{y_d} \big (|\yb| \Wb_\bk(\yb) \partial_{y_d}\big). 
\end{align}
 
\subsubsection{Product Hermite polynomials}
These polynomials are orthogonal in the inner product $L^2(\RR^d, \bm W^H)$, 
$$
    \la f, g \ra^H = \frac{1}{\pi^{\frac{d}2} } \int_{\RR^d} f(\xb) g(\xb)   \Wb^H(\xb) \d \xb,
$$  
where the weight function is the Hermite weight, which is the Gaussian function,  
\begin{equation}\label{eq:W_H}
         \Wb^H(\xb) =  \e^{- \|\xb\|^2}, \qquad \xb \in \RR^d,
\end{equation}
and an orthogonal basis for the space $\CV_n(\RR^d, \Wb^H)$ is given by the product Hermite polynomials 
\begin{equation}\label{eq:prod_H}
    \Hb_{\jb,n} (\xb) = H_{j_1}(x_1) \cdots  H_{j_d}(x_d), \qquad |\jb|= n. 
\end{equation}
These polynomials are the eigenfunctions of the differential operator $\fD^H$,
\begin{equation}\label{eq:diff-eqnH}
   \fD^H u : = ( \Delta - 2 \la \xb, \nabla \ra )u = - 2 n u, 
    \quad u \in \CV_n\left(\RR^d, \bm W^H\right).  
\end{equation}

We note that products of the Hermite polynomials and the Laguerre polynomials are eigenfunctions of the 
product differential operator. More precisely, for $\bk \in \RR^{d_2}$, $\k_i > -1$, $1 \le i \le d_2$, $\xb \in \RR^{d_1}$
and $\yb \in \RR^{d_2}$, the product polynomials
$$
  \Qb_{\jb, \kb, m}^n (\xb, \yb) = \Lb_{\jb,n-m}^\bk(\xb) \Hb_{\kb, m}(\yb),
$$
where $\jb \in \NN_0^{d_1}$ with $|\jb| = n-m$, $\kb \in \NN^{d_2}$ with $|\kb| = m$, and $0 \le m \le n$, are orthogonal 
polynomials of degree $n$ for the product weight function 
\begin{equation}\label{eq:W_LtimeH}
  \Wb_\bk^{L\times H}(\xb,\yb) = \Wb_\bk^L(\xb) \Wb^H(\yb), \qquad (\xb,\yb) \in \RR_+^{d_1}\times \RR^{d_2},
\end{equation}
and they consist of a basis for the space $\CV_n(\RR_+^{d_1} \times \RR^{d_2}, \Wb_\bk^{L\times H})$. Moreover,
let 
\begin{equation} \label{eq:fD-Hd}
  \fD_{\bk}^{L\times H} =  2 \fD_\bk^{L, (\xb)} + \fD^{H, (\yb)},
\end{equation}
where we include superscript $(\xb)$, or $(\yb)$, to indicate the variable that the operator acts on. Then
\begin{equation}\label{eq:eigen_LtimeH}
  \fD_{\bk}^{L \times H} u = - 2 n u, \qquad u \in \CV_n(\RR_+^{d_1} \times \RR^{d_2}, \Wb_\bk^{L\times H}),
\end{equation}
which follows easily as the eigenvalues for both $\fD_\bk^L$ and $\fD^H$ are linear, as in the case of two variables. 

\subsection{Wrapped product} 
Let $w_i$ be a weight function on the interval $I_i$ of the real line. Recall that $p_m(w_i)$ denotes the orthogonal 
polynomial of degree $m$ with respect to $w_i$. Beyond product weight functions, another family of weight
functions for which orthogonal polynomials of two variables can be constructed explicitly is defined by 
\begin{equation} \label{eq:wrap11}
       W(x,y) = w_1\bigg(\frac{x}{\rho(y)} \bigg) w_2(y),
\end{equation}
where $\rho$ is nonnegative, either a linear polynomial or the square root of a quadratic polynomial, for which 
we require that $w_1$ is an even weight function with $I_1 = [-c,c]$ being an interval symmetric about the origin. 
For such a weight function, the polynomials of degree $n$ defined by 
\begin{equation} \label{eq:wrapOP11}
  P_{n,m}(x,y) = p_{n-m} \big(\rho^{2m+1} w_2; y\big) [\rho(y)]^m p_m\bigg(w_1; \frac{x}{\rho(y)} \bigg), \quad 0 \le m \le n, 
\end{equation}
consist of an orthogonal basis for $\CV_n(\Omega, W)$, where $\Omega = \{(x,y): \frac{x}{\rho(y)} \in I_1, \, y \in I_2\}$.
This approach was used in \cite{Ag} and highlighted in \cite{K75}. 
We call such a weight function or orthogonal polynomials in \eqref{eq:wrapOP11} {\it wrapped product} ones. 

When $\rho(y) = 1$, the wrapped product coincides with the ordinary product. By \eqref{eq:OP_BB2}, 
the polynomials $\Bb_{m,n}^\mu$ on the unit disk are wrapped product orthogonal polynomials with 
$\rho(y) = \sqrt{1-y^2}$. Moreover, by \eqref{eq:OP_TT2}, the polynomials $\Tb_{m,n}^\mu$ 
on the triangle $\triangle^2$ are wrapped product orthogonal polynomials with $\rho(y) = 1-y$. 
In particular, it follows that all five families of classical orthogonal polynomials in two variables are 
wrapped products. 
\medskip

The concept of wrapped product is adopted and extended for studying orthogonal polynomials
on the quadratic domain 
$$
  \VV^{d+1} = \left\{(\xb, t): \|\xb\| \le \rho(t), \,\, x \in \RR^d, t \in I_\rho\right\},
$$
where $I_\rho$ is the interval on which $\rho$ is defined in \cite{OX, X20}, where the weight function is 
\begin{equation}\label{eq:wrap_d1}
    \Wb(\xb,t) =  w (t) W\bigg(\frac{\xb}{\rho(t)}\bigg), \qquad (\xb,t) \in \VV^{d+1}. 
\end{equation}
In particular, two families of orthogonal polynomials on the conic domains, $\VV^{d+1}$ with $\rho(t) = t$, 
were shown in \cite{X20} to possess a spectral operator; one is associated with the Jacobi weight 
$$
  \Wb_{\mu,\g}^J(\xb, t) = t^{2\mu-1} (1-t)^\g \left(1- \frac{\|x\|^2}{t^2}\right)^{\mu-\f12}= (t^2-\|x\|^2)^{\mu-\f12} (1-t)^\g,
   \quad I_\rho = [0,1], 
$$
and another is associated with the Laguerre weight 
$$
    \Wb_{\mu}^L(\xb, t) = t^{2\mu-1} e^{-t} \left(1- \frac{\|x\|^2}{t^2}\right)^{\mu-\f12}=
        (t^2-\|x\|^2)^{\mu-\f12} e^{-t}, \quad I_\rho = \RR_+. 
$$
These two will become special cases of our new construction in the next section. It is worth mentioning that 
before these two families on rotational cones, the unit ball $\BB^d$ and the simplex $\triangle^d$ 
were the only non-product domains on which orthogonal polynomials possess a spectral operator.
\medskip

The concept of wrapped product orthogonal polynomials will be further extended in the following section,
and used to discover new families of orthogonal polynomials that possess a spectral operator.

\section{Wrapped product orthogonal polynomials}
\setcounter{equation}{0}

In this section, we lay down the foundation for our study. In the first subsection, we define and discuss 
wrapped product orthogonal polynomials on domains. Potential families of wrapped product orthogonal 
polynomials that possess a spectral operator are analysed and listed in the second subsection. 

\subsection{Wrapped product and polynomials on domains} 
A domain $\Omega$ in $\RR^d$ is called centrally symmetric if $\xb \in \Omega$ implies $- \xb \in \Omega$. 
A weight function $\Wb$ defined on $\Omega$ is called centrally symmetric if $\Omega$ is centrally
symmetric and $\Wb(\xb) = \Wb(-\xb)$ for all $\xb \in \Omega$. For example, $\BB^d$ and $\Wb_\mu^\BB$
are centrally symmetric, so are $\RR^d$ and $\Wb^H$. 

\begin{defn} \label{defn:wrap}
Let $d_1$ and $d_2$ be positive integers and let $\Omega_1^{d_1} \subset \RR^{d_1}$ and 
$\Omega_2^{d_2} \subset \RR^{d_2}$ be two domains with non-empty interior. Let $\Wb_1$ be a weight 
function on $\RR^{d_1}$ and $\Wb_2$ be a weight function defined on $\Omega_2^{d_2}$. 
The wrapped product weight function $\Wb_1 \rtimes \Wb_2$ is defined by 
\begin{equation}\label{eq:wrapW}
\Wb^{\Omega_1\rtimes \Omega_2}(\xb,\yb) =\Wb_1\left( \frac{\xb} {\rho(\yb)} \right) \Wb_2(\yb) , 
   \qquad (\xb, \yb) \in \Omega_1^{d_1} \rtimes \Omega_2^{d_2}
\end{equation}
where the wrapped domain $\Omega_1^{d_1} \rtimes \Omega_2^{d_2}$ is defined by 
\begin{equation} \label{eq:wrapDomain}
\Omega_1^{d_1}  \rtimes \Omega_2^{d_2}:= \left\{(\xb,\yb): \frac{\xb}{\rho(\yb)} \in \Omega_1^{d_1} , 
\quad \yb \in \Omega_2^{d_2} \right\}, 
\end{equation}
where $\rho: \Omega_2^{d_2} \mapsto \RR_+$ is nonnegative and specified according to one of the two cases,
\begin{enumerate}[    ]
\item[] {\bf Case 1}. $\rho$ is a polynomial of degree $1$; \smallskip
\item[] {\bf Case 2}. $\rho$ is the square root of a polynomial of degree at most $2$, and $\Omega_1^{d_1} $ 
and $\Wb_1$ are both centrally symmetric. 
\end{enumerate}
\end{defn}

Let $\{\Pb_{\kb, m}: |\kb| = m, \kb \in \NN_0^{d_1}\}$ be an orthogonal basis of $\CV_m(\Omega_1^{d_1}, \Wb_1)$ and
denote the square of its norm by $\hb_{\kb,m}(\Wb_1)$. For each nonnegative integer $m$, let 
$\big\{\Rb^{(m)}_{\jb, n}: |\jb| = n, \jb \in \NN_0^{d_2}\big\}$ be an orthogonal basis for 
$\CV_n\left(\Omega_2^{d_2}, \rho^{2m + d_1} \Wb_2\right)$ and denote the square of its norm by 
$\hb_{\jb,m}(\rho^{2m+d_1}\Wb_2)$.

\begin{prop} \label{prop:wrapped_orth}
Let polynomials $\Qb_{\jb, \kb, m}^n$ of $d_1+d_2$ variables be defined by
\begin{equation}\label{eq:wrapOP}
   \Qb_{\jb, \kb, m}^n(\xb, \yb) = \Rb_{\jb,n-m}^{(m)}(\yb) [\rho(\yb)]^m \Pb_{\kb,m}\left(\frac{\xb}{\rho(\yb)}\right),  
\end{equation}
where $0 \le m \le n$, $|\kb| = m$ with $\kb \in \NN_0^{d_1}$, and $|\jb| = n-m$ with $\jb \in \NN_0^{d_2}$. 
Then the collection $\left\{  \Qb_{\jb, \kb, m}^n: |\kb| = m,\,  |\jb| = n-m, 0 \le m \le n\right\}$ is an orthogonal basis 
of $\CV_n(\Omega_1^{d_1}  \rtimes \Omega_2^{d_2}, \Wb^{\Omega_1\rtimes \Omega_2})$.
Moreover, the squared norm $\hb_{\jb,\kb,m}^{n, \Omega_1\rtimes \Omega_2}squared norm
 = \la  \Qb_{\jb, \kb, m}^n, \Qb_{\jb, \kb, m}^n\ra_\Wb$ is equal to 
$$
\hb_{\jb,\kb,m}^{n, \Omega_1\rtimes \Omega_2} = \hb_{\jb, n-m}\big(\rho^{2m+d_1}\Wb_2\big) \hb_{\kb,m}(\Wb_1).
$$
\end{prop}

\begin{proof}
First we verify that $\Qb_{\jb, \kb, m}^n$ is a polynomial of degree $n$ in $(\xb, \yb)$. For Case 1, this is trivial. 
For Case 2, the polynomial $\Pb_{\kb,m}$ consists of only monomials of even degrees if $m$ is even
and only monomials of odd degrees if $m$ is odd, by the central symmetry of $\sW_1$ and $\Omega_1$. 
It follows that $[\rho(\yb)]^m \Pb_{\kb, m}(\frac{\xb}{\rho(\yb)})$ is a polynomial of degree at most $m$ in 
$(\xb,\yb)$, so that $\Qb_{\jb,\kb,m}^n$ is of degree $n$ in $(\xb,\yb)$ variables.  

The integral on $\VV^{d_1,d_2}$ can be parametrized by setting $\xb = \rho(\yb) \ub$, so that the inner 
product $\la \cdot,\cdot \ra_{\Omega_1\rtimes \Omega_2}$ of $L^2( \Omega_1^{d_1}  \rtimes \Omega_2^{d_2}, \Wb^{\Omega_1\rtimes \Omega_2})$ can be written as 
\begin{align} \label{eq:integra;=}
\la f, g \ra_{\Omega_1\rtimes \Omega_2}
   \, &=   \int_{ \Omega_1^{d_1}  \rtimes \Omega_2^{d_2}} f(\xb,\yb) g(\xb,\yb) 
       \Wb^{\Omega_1\rtimes \Omega_2}(\xb, \yb) \d \xb \d \yb  \\
  & =  \int_{\Omega_2^{d_2}} \int_{\Omega_1^{d_1}} f\big(\rho(\yb) \ub, \yb\big)  g\big(\rho(\yb) \ub, \yb\big) \Wb_1(\ub) \d \ub 
  [\rho(\yb)]^{d_1} \Wb_2(\yb) \d \yb. \notag
\end{align}
It follows readily that 
\begin{align*}
  \la  \Qb_{\jb, \kb, m}^n,  \Qb_{\jb', \kb', m'}^{n'}\ra_{\Wb} & = 
   \left \langle \Rb_{\jb,n-m}^{(m)}, \Rb_{\jb',n'-m'}^{(m')} \right \rangle_{\rho^{d_1+m+m'} \Wb_2}  \la \Pb_{\kb,m}, \Pb_{\kb',m'}\ra_{\Wb_1}  \\
 & = \Big \langle  \Rb_{\jb,n-m}^{(m)}, \Rb_{\jb',n'-m}^{(m)} \Big \rangle_{\rho^{d_1+2m} \Wb_2} \hb_m(\Wb_1)
 \delta_{\kb, \kb'} \delta_{m,m'} \\
 & = \hb_{n-m}\big(\rho^{2m+d_1}\Wb_2\big) \hb_m(\Wb_1)
  \delta_{\jb, \jb'}\delta_{\kb, \kb'} \delta_{m,m'} \delta_{n,n'},
\end{align*}
which establishes the orthogonality. Moreover, since
$$
  \sum_{m=0}^n \binom{m + d_1 - 1}{m} \binom{n - m + d_2 - 1}{n - m} = \binom{n + d_1 + d_2 - 1}{n}, 
$$
the collection of $\bm Q_{\bm j, \bm k, m}^n$ forms an orthogonal basis of 
$\CV_n\big( \Omega_1^{d_1}\rtimes \Omega_2^{d_2}, \Wb^{\Omega_1\rtimes \Omega_2}\big)$. 
\end{proof}

We call the polynomial defined in \eqref{eq:wrapOP} {\it wrapped product polynomials} and, accordingly, the domain
$\Omega_1 \rtimes \Omega_2$ {\it wrapped product domain}.  

For $d_1 = d_2 = 1$, the definition coincides with \eqref{eq:wrap11} and \eqref{eq:wrapOP11} for wrapped product
and orthogonal polynomials of two variables. For $d_2 =1$, it coincides with the one defined in \eqref{eq:wrap_d1}. 
It has not been, however, systematically studied in the current generality as far as we are aware. 

In the case $\rho(\yb) = 1$, the wrapped product weight and polynomials are just the ordinary tensor product 
weight and product polynomials. We are naturally interested in the cases when $\rho$ is not a constant. 
Let us first consider the classical orthogonal polynomials on the simplex and on the unit ball as examples. 

\subsubsection{Jacobi polynomials on the Simplex} For $d = d_1+d_2$, we write the simplex $\triangle^{d}$ as a wrapped product 
domain with $\rho(\yb) = 1- |\yb|$ with $\yb \in \triangle^{d_2}$, 
$$
   \triangle^{d} = \left\{(\xb,\yb):   \frac{\xb}{1-|\yb|} \in \triangle^{d_1}, \, \yb \in \triangle^{d_2} \right\} 
      = \triangle^{d_1} \rtimes \triangle^{d_2}.
$$
and the Jacobi weight function $\Wb_{\bk, \bg}^{\triangle, d}$ for $\bk \in \RR^{d_1+1}$ with $\bk_i > -1$, $1 \le i \le d_1+1$
and $\bg \in \RR^{d_2}$ with $\g_i > -1$, $1 \le i \le d_2$ becomes 
\begin{align*}
  \Wb_{\bk, \bg}^{\triangle,d}(\xb, \yb) & =   \prod_{i=1}^{d_2} y_i^{\g_i} \prod_{i=1}^{d_1} x_i^{\k_i}(1-|\xb|-|\yb|)^{\k_{d_1+1}}  \\
    & = \Wb_\bk^{\triangle, d_1} \left(\frac{\xb}{1-|\yb|}\right) \Wb_{(\bg, |\bk|)}^{\triangle, d_2}(\yb),
\end{align*}
where we have added the dimension $d$ in the superindex of $\Wb_{\bk}^{\triangle, d}$ to distinguish the dimensions. 
In particular, the wrapped product polynomials $\Qb_{\jb,\kb,m}^n$ in \eqref{eq:wrapOP} are exactly the Jacobi polynomials
on $\triangle^d$ with the parameters $(\bk,\bg) \in \RR^{d+1}$, and 
$$
  \Tb_{(\kb,\jb),n}^{(\bk,\bg)} (\xb, \yb) = \Tb_{\jb, n-m}^{(\bg,  |\bk|+d_1+2m)}(\yb) (1-|\yb|)^m
     \Tb_{\kb,m}^{\bk}\left(\frac{\xb}{1-|\yb|}\right),
$$
where $|\jb| = n-m$ and $|\kb| = m$. In particular, if $d_1 = d-1$ and $d_2 = 1$, then 
$\Tb_{\jb, n-m}^{(\bg, |\bk|+d_1+2m)}(\yb)$ becomes the Jacobi polynomial 
$P_{n-m}^{(\g, |\bk|+d-1+2m)}(1- 2y)$ for $y \in  \triangle^1 = [0,1]$, and we end up with a recursive definition,
in dimension, of the Jacobi polynomials on $\triangle^d$. 
 
\subsubsection{Gegenbauer polynomials on the unit ball} For $d = d_1 + d_2$, we write the unit ball 
$\BB^d$ as a wrapped product domain with $\rho(x) = \sqrt{1-\|\yb\|^2}$ by 
$$
  \BB^d = \left \{\xb = (\xb,\yb): \frac{\xb}{\sqrt{1-\|\yb\|^2}} \in \BB^{d_1},\, \yb \in \BB^{d_2} \right\} = \BB^{d_1} \rtimes \BB^{d_2} 
$$
and the classical weight function $\Wb^\BB_\mu$ becomes accordingly
$$
   \Wb^{\BB,d}_\mu(\xb,\yb) = (1-\|\xb\|^2-\|\yb\|^2)^{\mu-\f12} 
        =  W_{\mu}^{\BB, d_1} \bigg(\frac{\xb}{\sqrt{1-\|\yb\|^2}} \bigg)\Wb^{\BB, d_2}_\mu (\yb),
$$
where we have again added the dimension $d$ in the superindex of $\Wb_{\mu}^{\BB, d}$. In particular, 
the wrapped product polynomials \eqref{eq:wrapOP} gives 
$$
  \Bb_{\jb, n}^{\mu,d} (\xb,\yb) 
    = \Bb_{\jb, n-m}^{\mu+m+\frac{d_1}{2}, d_2}(\yb) (1-\|\yb\|^2)^{\frac m 2}
     \Bb_{\kb,m}^{\mu,d_1} \bigg(\frac{\xb}{\sqrt{1-\|\yb\|^2}}\bigg), 
$$
where $|\jb|= n-m$ and $|\kb|= m$. In particular, 
this is a recursive definition of the Jacobi polynomials on $\BB^d$. If $d_1 = d-1$ and $d_2 = 1$, then 
$\Bb_{\jb, n-m}^{\mu}(\yb)$ becomes the Gegenbauer polynomial $C_{n-m}^{\mu+ m + \frac{d_1}{2}}(y)$ for 
$y \in \BB^1 = [-1,1]$, and we end up with a recursive definition, in dimension, of the classical 
Gegenbauer polynomials on $\BB^d$.

\medskip

Our next example explains why we need to use $\RR_\Sigma^d$ instead of $\RR_+^d$ in some cases. 

\subsubsection{Wrapped product of Simplex and Laguerre polynomials}\label{subsec:T-L}
 For $d_1 \ge 1$ and $d_2 \ge 1$, we consider
the wrapped product of orthogonal polynomials on the simplex and product Laguerre polynomials. We choose
$\rho (\yb) = |\yb|$ and work with
\begin{align*}
\triangle^{d_1}  \rtimes \RR_\Sigma^{d_2} \, & = \left\{(\xb,\yb): \frac{\xb}{|\yb|} \in \triangle^{d_1}, \, \yb \in \RR_\Sigma^{d_2}\right\} 
        = \left\{(\xb,\yb) \in \RR_+^{d_1} \times \RR_\Sigma^{d_2}:  |\xb| \le |\yb|\right\}.
\end{align*}
For $\bg \in \RR^{d_1}$ with $\g_i > -1$, $1 \le i \le d_1$, and $\bk \in \RR^{d_2}$ with $\k_i > -1$, $1\le i \le d_2$, 
the weight function $\Wb_{\bg, \bk}^{\triangle \rtimes \Sigma}$ is defined on $\triangle^{d_1}  \rtimes \RR_\Sigma^{d_2}$ 
by  
\begin{align} \label{eq:W-SigT}
\Wb_{\bg, \bk}^{\triangle \rtimes \Sigma}(\xb,\yb) \, & = \prod_{i=1}^{d_1} x_i^{\g_i}  \prod_{i=1}^{d_2-1} y_i^{\k_i}
     (|\bm y|- |\bm x | )^{\k_{d_2}} \e^{-|\yb'| -|\yb|}\\
  & = \Wb_{(\bg, \k_{d_2})}^{\triangle} \left( \frac{x}{|\bm y|} \right) \Wb_{\bk + |\bg| \bep}^\Sigma(\yb), \notag
\end{align}
where $\bep = (0,\ldots,0,1) \in \RR^{d_2}$. The wrapped orthogonal polynomials \eqref{eq:wrapOP} become in this setup
 \begin{equation} \label{eq:OP_SigT}
 \bm Q_{\bm j, \bm k, m}^n(\bm x,\bm y) = \hat \Lb_{\bm j, n-m}^{\bk+(\alpha+2m)\bep}(\bm y) |\bm y|^m 
      \bm T_{\bm k, m}^{(\bg, \k_{d_2}) } \left(\frac{\bm x}{|\bm y|}\right),
\end{equation}
where $\alpha = |\bg|+ d_1$, $|\bm k| = m$ for $\bm k \in \NN_0^{d_1}$, $|\bm j| = n-m$ for $\bm j \in \NN_0^{d_2}$, 
and $0 \le m  \le n$. 

This case, however, is affine equivalent to the product Laguerre polynomials. Indeed, making a linear change of 
variable $\yb \to \vb \in \RR_+^{d_2}$ with 
$$
v_1= y_1, \ldots, v_{d_2-1} = y_{d_2-1}, \quad v_{d_2} = |\yb| - |\xb|,
$$
it follows readily that $(\xb, \vb) \in \RR_+^{d_1+d_2}$ and, moreover, since 
$$
|\yb| + |\yb'| = |\xb| + v_{d_2} + |\vb'| = |\xb| + |\vb|
$$
the weight function $\Wb_{\bg, \bk}^{\triangle \rtimes \Sigma}(\xb,\yb)$ becomes the product Laguerre weight
function 
$$
  \Wb_{\bg,\bk}^L (\xb, \vb) = \prod_{i=1}^{d_1} x_i^{\g_i} \prod_{j=1}^{d_2} v_i^{\k_i} \e^{-|\xb| - |\vb|}
$$
with parameters $(\bg, \bk)$ on the domain $\RR_+^{d_1 + d_2}$. For $d =1$, $\RR_\Sigma^1 = \RR_+$, 
so that \eqref{eq:OP_SigT} with $d_1 = d-1$ and $d_2 = 1$ leads to another orthogonal basis for the space 
$\CV_n(\Wb_\bg^L, \RR_+^{d})$, which is given by, statede for $\xb \in \RR^d$ and $\bg \in \RR^d$ with $\g_i  > -1$, 
$$
    \bm Q_{\bm k, m}^n(\bm x) =  L_{n-m}^{|\bg|+2m+d-1}(|\xb|) |\xb|^m 
      \bm T_{\bm k, m}^{\bg} \left(\frac{\bm x}{|\bm x|}\right), 
$$
where $\kb \in \NN_0^{d}$ with $|\kb| = m$ and $0 \le m \le n$. 

If we were to use $\RR_+$ instead of $\RR_\Sigma$ and consider $\triangle^{d_1} \rtimes \RR_+^{d_2}$, then 
we would need, by \eqref{eq:wrapOP}, that $\rho^{d_1} \Wb_\bk^L$ is a product Laguerre weight, which 
means that we have to choose $\rho(\yb) = y_i$ for some $1\le i \le d_2$. This works if we replace, for example,
$\RR_+^d$ by the domain $\{\yb \in \RR^d: 0 \le y_1 \le y_2 \le ... \le y_d\}$ and choose $\rho(\yb) = y_d$. Such
an affine transformation of $\RR_+^d$ would change the product Laguerre weight to the weight function
$\prod_{i=1}^{d} (x_i-x_{i-1})^{\k_i} \e^{-x_d}$, where $x_0 =0$, which is somewhat more involved. Instead,
we choose $\RR_\Sigma^d$ and use $\rho(\yb) = |\yb|$.

\subsection{Spectral operator and wrapped product orthogonal polynomials} We are particularly interested in orthogonal
polynomials that are eigenfunctions of a linear differential operator that has eigenvalues depending only on the degree of 
polynomials. This means that we are looking for a linear differential operator  $\fD(\partial_{\xb}, \partial_{\yb})$ that has 
$\CV_n(\Omega_1^{d_1}  \rtimes \Omega_2^{d_2}, \Wb)$ as its eigenspaces, that is,
$$
    \fD\big(\partial_{\xb}, \partial_{\yb}\big) u = \l_n u, \qquad  u \in \CV_n\big( \Omega_1^{d_1}  \rtimes \Omega_2^{d_2}, \Wb\big),
$$
where $\l_n \in \RR$ depends only on $n$ and some fixed parameters used in the definition of $W$ and $\Omega_i^{d_i}$. 

\begin{prop}\label{prop:all_operator}
Let $\Wb$ be the wrapped weight in \eqref{eq:wrapW}. In order for the space
$\CV_n\big( \Omega_1^{d_1}  \rtimes \Omega_2^{d_2}, \Wb\big)$ to be an eigenspace of a linear 
partial differential operator for each $n$, it is necessary that 
$\CV_n\big(\Omega_1, \Wb_1\big)$ and $\CV_n\big(\Omega_2, \rho^{d_1} \Wb_2\big)$ are eigenspaces of their
respective spectral operators for each $n$. 
\end{prop}

\begin{proof}
In terms of the basis given in \eqref{eq:wrapOP}, the equation becomes 
\begin{equation} \label{eq:fD-general}
    \fD (\partial_\xb, \partial_\yb) \Qb_{\jb,\kb,m}^n = \l_n \Qb_{\jb,\kb,m}^n, \quad |\jb| = n-m, \quad |\kb| =m, \quad 0 \le m \le n,  
\end{equation}
where $\kb \in \NN^{d_1}$ and $\jb \in \NN^{d_2}$. Assume such an operator exists. Then, setting $m = 0$ in 
\eqref{eq:fD-general} and using $\Qb_{\jb,\kb, 0}^n = \Rb_{\jb,n}^{(0)}(\yb) = \Rb_{\jb,n}^{(\yb)}$, which consists of
an orthogonal basis for $\CV_n(\Omega_2, \rho^{d_1} \Wb_2)$, we obtain 
$$
  \fD (0, \partial_\yb)  \Rb_{\jb,n}(\yb)= \l_n \Rb_{\jb, n}(\yb), \qquad |\jb| = n, \quad \jb \in \NN_0^{d_2}, 
$$
where $\Rb_{\jb,n} =\Rb_{\jb,n}^{(0)}$ and we replace $\partial_\xb$ by 0 in $\fD (0, \partial_\yb)$ since $\Rb_{\jb,n}$ 
depends only on $\yb$ and we assume that the differential operator is linear. Since $\{\Rb_{\jb,n}: |\jb|=n\}$
consists of an orthogonal basis for $\CV_n(\Omega_2^{d_2},  \rho^{d_1} \Wb_2)$, it follows that 
$\CV_n(\Omega_2^{d_2},  \rho^{d_1} \Wb_2)$ is 
an eigenspace of the differential operator $\fD(0,\partial)$. Next, setting $m = n$ in \eqref{eq:wrapOP}, and 
applying $\fD (\partial_\xb, \partial_\yb)$ on $\Qb_{\jb,\kb, n}^n(\xb,\yb) = |\rho(\yb)|^n \Pb_{\kb,n}(\frac{\xb}{\rho(\yb)})$.
Since both partial derivatives with respect to $\xb$ and $\yb$ act on $\Pb_{\kb,n}$ and the $\fD (\partial_\xb, \partial_\yb)$
is linear, it follows readily that $\fD (\partial_\xb, \partial_\yb) \Qb_{\jb,\kb, n}^n$ becomes a second-order linear
differential operator $\fF$ acting on $\Pb_{\kb,n}(\cdot)$, whose coefficients depend on both $\xb$ and $\yb$ and
are polynomial in $\xb$ of degree no more than inherited from $\fD(\partial_\xb,\partial_\yb)$. For any fixed $\yb$ 
on the domain, $\fF$ is a spectral operator that has $\CV_n(\Omega_1^{d_1}, \Wb_1)$ as an eigenspace. In particular, 
$$
      \fF \Pb_{\kb,n} = \mu_n \Pb_{\kb, n}, \qquad |\kb| = n, \quad \kb \in \NN_0^{d_1},
$$
for some real number $\mu_n$. 
Thus, we conclude that both families of orthogonal polynomials on $\Omega_1^{d_1}$ and $\Omega_2^{d_2}$ are eigenfunctions of their own differential operator. 
\end{proof}

For $d =2$, all five primary spectral operators are established for wrapped product orthogonal polynomials, 
which arise from three classical orthogonal polynomials of one variable. However, not all wrapped products of
classical orthogonal polynomials are eigenfunctions of a spectral operator. For example, the tensor products of 
the Jacobi polynomials and the Laguerre polynomials are orthogonal polynomials in two variables, but they
are not eigenfunctions of a second-order linear differential operator. Starting from $d = 2$, 
Proposition \ref{prop:all_operator} makes it possible to work out all cases of wrapped product orthogonal 
polynomials that are eigenfunctions of a spectral operator. For $d =3$, this means that $\Wb_1$ and 
$\rho^{d_1} \Wb_2$ are among the five primary weight functions of two variables and the three primary 
weight functions of one variable, which can be consolidated as $\Wb_1$ and $\Wb_2$ belonging to the set 
of classical weight functions: (product) Hermite weight, (product) Laguerre weight, classical weight on the 
unit ball and on the simplex. 

We shall work with generic positive integers $d_1$ and $d_2$ and $\Wb_1$ and $\Wb_2$ belonging to 
the set of classical weight functions: product Hermite, product Laguerre, Gegenbauer polynomials on 
the ball, and Jacobi polynomials on the simplex, and identify all those cases when wrapped orthogonal 
polynomials are eigenfunctions of a differential operator. For $d = 3$, this will cover all possible cases. 
For $d > 3$, however, there could also be additional cases; for example, $\Wb_1 = 
\Wb_\mu^\BB \rtimes \Wb_\bk^\triangle$ and $\Wb_2 = \Wb_\bg^\Sigma$, which we shall not address here. 

There are 16 cases in total when we select $\Wb_1$ and $\Wb_2$ from the set of four classical weight functions. 
Instead of denoting the cases by their weight functions, some of which have their own parameters, we shall
distinguish the cases by their domains: the unit ball $\BB^d$, the simplex $\triangle^d$, $\RR_\Sigma^d$ 
(or $\RR_+^d$) for the Laguerre weight, and $\RR^d$ for the Hermite weight. Thus, the 16 possible cases are 
\begin{align}\label{eq:classes}
  \begin{matrix}  \BB^{d_1} \rtimes \BB^{d_2}, & \quad  \BB^{d_1} \rtimes \triangle^{d_2}, \quad &
    \BB^{d_1} \rtimes \RR_\Sigma^{d_2}, \quad & \quad \BB^{d_1} \rtimes \RR^{d_2},\\
     \triangle^{d_1} \rtimes \BB^{d_2}, & \quad  \triangle^{d_1} \rtimes \triangle^{d_2}, \quad &
    \triangle^{d_1} \rtimes \RR_\Sigma^{d_2}, \quad & \quad \triangle^{d_1} \rtimes \RR^{d_2}, \\
     \RR_+^{d_1} \rtimes \BB^{d_2}, & \quad  \RR_+^{d_1} \rtimes \triangle^{d_2}, \quad &
     \RR_+^{d_1} \rtimes \RR_+^{d_2}, \quad & \quad  \RR_+^{d_1} \rtimes \RR^{d_2},\\
     \RR^{d_1} \rtimes \BB^{d_2}, & \quad  \RR^{d_1}\rtimes \triangle^{d_2}, \quad &
      \RR^{d_1} \rtimes \RR_\Sigma^{d_2}, \quad & \quad   \RR^{d_1} \rtimes \RR^{d_2},
      \end{matrix}
 \end{align}

Some of these cases, however, can be quickly eliminated. As a starter, we recall that $\rho(\yb) = 1$ 
corresponds to tensor product cases, which work only for the tensor product of Hermite and/or Laguerre families 
because of the linearity of the eigenvalues. There are three such cases, $\RR^{d_1} \rtimes \RR^{d_2}$, 
$\RR_+^{d_1} \rtimes \RR^{d_2}$, and $ \RR_+^{d_1} \rtimes \RR_+^{d_2}$, each of which has the spectral 
operator given as the sum of the two spectral operators of its components. Assume next that $\rho$ is not a 
constant function. First, requiring $\rho^{d_1} \Wb_2$ to possess a spectral operator implies then that $\Wb_2$ 
cannot be the Hermite weight $\Wb_H$, which eliminates $\BB^{d_1} \rtimes \RR^{d_2}$ and 
$\triangle^{d_1} \rtimes \RR^{d_2}$. Second, if $\Wb_2$ is the classical weight function $\Wb_\mu^\BB$ on 
the unit ball, then $\rho^{d_1} \Wb_2$ being also such a function requires $\rho(\yb) = \sqrt{1- \|\yb\|^2}$. 
This $\rho$ is an example of Case 2, for which we require $\Wb_1$ to be centrally symmetric, as seen 
in \eqref{eq:wrapOP}. Consequently, this eliminates the cases $\triangle^{d_1} \rtimes \BB^{d_2}$ and 
$ \RR_\Sigma^{d_1} \rtimes \BB^{d_2}$. 

In the remaining 9 cases, two other diagonal cases, $\BB^{d_1} \rtimes \BB^{d_2}$ and 
$ \triangle^{d_1} \rtimes \triangle^{d_2}$ are $\BB^{d_1+d_2}$ and $\triangle^{d_1+d_2}$ discussed in 
the previous subsection, as is the case $\triangle^{d_1} \rtimes \RR_\Sigma^{d_2}$, which is an affine 
transformation of the product Laguerre setting on $\RR_+^{d_1+d_2}$. 

A comment on $\RR_+^d$ or $\RR_\Sigma^d$ is in order. As we explained at the end of the last subsection,
if $\rho(\yb)$ is not a constant, then requiring $\rho^{d_1} \Wb_2$ to be a Laguerre weight implies that 
$\Wb_2$ cannot be a product Laguerre weight unless $\rho(\yb) = y_i$ for some $1\le i \le d_2$, and
our choice is to choose $\rho(\yb) = |\yb|$ and work with the alternative Laguerre weight $\Wb_\k^\Sigma$ 
defined in \eqref{eq:LaguerreSig} as $\Wb_2$. This explains why, in \eqref{eq:classes}, $\RR_\Sigma^{d_2}$ 
appears whenever $\Wb_2$ is a Laguerre weight, whereas $\RR_+^{d_1}$ is used whenever $\Wb_1$ is a 
Laguerre weight. This leaves, however, an abnormality in the entry $\RR_+^{d_1} \rtimes \RR_+^{d_2}$ 
in \eqref{eq:classes}. Indeed, we could replace it by $\RR_+^{d_1} \rtimes \RR_\Sigma^{d_2}$ with 
$\rho(\yb) = |\yb|$ and the two cases are distinct and are not equivalent under affine transforms. Since 
the case is, however, somewhat artificial, we decide not to include it. It suffices to say that the orthogonal
polynomials for this case are not eigenfunctions of a second-order differential operator. 
 
Summing up, we are left with 6 cases to be determined, which are naturally divided into two groups. The 
first one consists of 
 \begin{equation} \label{eq:class1}
  \BB^{d_1} \rtimes \triangle^{d_2} \quad \hbox{and}\quad \BB^{d_1} \rtimes \RR_\Sigma^{d_2},  
 \end{equation}
and the second one consists of those with $\Wb_1 = \Wb^H$ or $\Wb_\bg^L$ being the Hermite or Laguerre weight, 
\begin{equation} \label{eq:class2}
  \RR^{d_1} \rtimes \BB^{d_2},  \quad  \RR^{d_1}\rtimes \triangle^{d_2}, \quad 
      \RR^{d_1} \rtimes \RR_\Sigma^{d_2}, \quad 
 \RR_+^{d_1} \rtimes \triangle^{d_2}. 
\end{equation}

These cases will be studied in several sections below. As we shall see, each of the two cases in 
\eqref{eq:class1} has a second-order differential operator as a spectral operator, whereas the cases 
in \eqref{eq:class2} do not have a second-order spectral operator, but have one of fourth order in each case. 
  
\begin{rem} \label{rem:d=3}
Once the two cases in \eqref{eq:class1} are affirmed, we will have found all wrapped orthogonal polynomials 
that possess a spectral operator in the list of \eqref{eq:classes}, which consists of, besides the product Hermite 
and/or Laguerre polynomials, orthogonal polynomials on the unit ball $\BB^d$ and on the simplex $\triangle^d$, 
as well as multiple families of the form 
$$
   \BB^{d_1} \rtimes \triangle^{d_2}, \quad  \BB^{d_1} \rtimes \RR_\Sigma^{d_2}, \qquad 
     (d_1, d_2) \in \NN^2 \quad \hbox{with} \quad d_1+d_2 = d.
$$ 
Based on the classification for $d =2$, this gives us a complete list of spectral operators for wrapped product 
orthogonal polynomials for $d =3$. Indeed, according to our discussion, the only possible case that is not 
covered in the list when $d = 3$ is the wrapped product of $\BB^1 \rtimes \triangle^1 \rtimes \RR_+$, which 
however is equivalent to $\BB^1 \rtimes \RR_\Sigma^2$ as shown in Subsection \ref{subsec:T-L}. 
\end{rem}

For the case $d > 3$, there could be new families that arise from 
wrapped products $\Omega^{d_1} \rtimes \Omega^{d_2}$ when either $\Omega_1^{d_1}$ 
and/or $\Omega_2^{d_2}$ are wrapped product domains themselves. 
 

We end this subsection with a technical lemma that will be useful for deriving spectral operators for
wrapped product orthogonal polynomials. 

\begin{lem} 
Let $\rho: \RR^{d_2} \mapsto \RR_+$ and $h: \RR^{d_1} \mapsto \RR$ be differentiable. Then the function
$$
   H(\xb, \yb) := [\rho(\yb)]^m h\left(\frac{\xb}{\rho(\yb)} \right)
$$
satisfies the following relations
\begin{equation} \label{eq:diffHa}
  \rho(\yb)  \partial_{y_i} H = \partial_{y_i} \rho(\yb) \left( m  - \la \xb,\nabla_\xb\ra \right) H, \quad 1 \le i\le d_2,
\end{equation}
and, in particular, 
 \begin{equation} \label{eq:diffHb}
  \rho(\yb)  \la \yb, \nabla_{\yb}\ra  H = \la \yb, \nabla_{\yb}\ra \rho (\yb) \big( m  - \la \xb,\nabla_\xb\ra H \big). 
\end{equation}
\end{lem} 

\begin{proof}
Taking derivatives shows 
\begin{align*}
  \rho(\yb)  \partial_{y_i} H(\xb,\yb) \,& = m \, \partial_{y_i} \rho(\yb) H(\xb,\yb) -  [\rho(\yb)]^m 
     \sum_{\ell=1}^{d_1} \frac{x_\ell}{\rho(\yb)}  \partial_\ell h\left(\frac{\xb}{\rho(\yb)} \right)\partial_{y_i} \rho(\yb) \\
    & =  \partial_{y_i}  \rho(\yb)\left( m H(\xb,\yb)  -  \la \xb,\nabla_{\xb}\ra H(\xb, \yb) \right),
 \end{align*}
which verifies \eqref{eq:diffHa}. Multiplying \eqref{eq:diffHa} by $y_i$ and summing up gives \eqref{eq:diffHb}. 
\end{proof} 

\setcounter{equation}{0}
\section{Ball-Simplex polynomials $\BB^{d_1}\!\rtimes \triangle^{d_2}$}

We consider the wrapped product $\BB^{d_1}\!\! \rtimes \triangle^{d_2}$ of the unit ball and the simplex,
for which we choose $\rho(y) = 1-|\yb|$, so that \eqref{eq:wrapW} becomes  
\begin{align*}
\BB^{d_1}\!\! \rtimes\triangle^{d_2} \, & = \left\{(\xb,\yb): \frac{\xb}{1-|\yb|} \in \BB^{d_1}, \, \yb \in \triangle^{d_2}\right\} \\
   & = \left\{(\xb,\yb) \in \BB^{d_1}\! \times \! \triangle^{d_2}: \|\xb\| \le 1-|\yb| \right\}.
\end{align*}
In terms of the classical weight functions \eqref{eq:W_ball} on the unit ball and \eqref{eq:W_simplex} on the simplex,
the weight function \eqref{eq:wrapW} becomes $\Wb_{\mu, \bk}^{\BB\rtimes \triangle}$ for $\mu > -\f12$ 
and $\bk \in \RR^{d_2}$ with $\k_i > -1$, $1\le i \le d_2$, given by
\begin{equation} \label{eq:Weight-BS} normalization constant
\Wb_{\mu, \bk}^{\BB\rtimes \triangle}(\xb,\yb) = \Wb_{\mu}^{\BB} \left( \frac{x}{1-|\bm y|} \right) \Wb_{(\bk, 2\mu-1)}^\triangle(\yb)
=  \prod_{i=1}^{d_2} y_i^{\k_i} \left[ (1-|\bm y|)^2- \|\bm x\|^2 \right]^{\mu - \f12}.
\end{equation}
We consider orthogonal polynomials with respect to the inner product defined by
\begin{equation} \label{eq:ipdOmega}
  \la f, g\ra_{\mu, \bk}^{\BB\rtimes \triangle}
    = \bm b_{\mu,\bk}^{\BB\rtimes \triangle}
        \int_{\BB^{d_1} \rtimes \triangle^{d_2}} f(\bm x,\bm y) g(\bm x,\bm y) \bm \Wb_{\mu, \bk}^{\BB\rtimes \triangle}(\bm x,\bm y) 
        \d \bm x \d \bm y, 
\end{equation}
where $\bb_{\mu,\bk}^{\BB\rtimes \triangle}$ is the  normalization constant so that $\la 1, 1\ra_{\mu, \bk}^{\BB\rtimes \triangle} =1$. 
Using \eqref{eq:integra;=} with $\bm x = (1- | \bm y|) \bm u$, it follows readily that the normalization constant 
$\bb_{\mu,\bk}^{\BB\rtimes \triangle}$ can be given in terms of the normalization constant $\bm b_\mu^\BB$ 
in \eqref{eq:bB} for $\bm W_\mu^\BB$ and the constant $\bm b_{\bm \gamma}^\triangle$ of 
$\bm W_{\bm \gamma}^\triangle$. More precisely, 
\begin{equation} \label{eq:bOmega}
  \bm b_{\mu,\bk}^{\BB\rtimes \triangle} = \bm b^\BB_\mu \times \bm b^\triangle_{\bm \gamma} 
       \quad\hbox{with}\quad \bm \gamma = (\bk, d_1+ 2\mu -1) \in \RR^{d_2+1}, 
\end{equation}

The wrapped orthogonal polynomials \eqref{eq:wrapOP} in this setup are given in terms of orthogonal polynomials 
on the simplex and those on the unit ball. 
Let $\{\bm B_{\bm k, m}^{\mu}: |\bm k|= m, \, \bm k \in \NN_0^{d_1}\}$ be an orthogonal basis for 
$\CV_m(\BB^{d_1},\bm W^\BB_\mu)$ such as the one given in \eqref{eq:OP_BB}, and let 
$\{\bm T_{\bm j, m}^{\bg}: |\bm j| = m, \, \bm j\in \NN_0^{d_2}\}$ be an orthogonal basis for
$\CV_{m}\left(\triangle^{d_2},\bm W^\triangle_{\bg}\right)$, such as the one given in \eqref{eq:OP_TT}. 
We denote by $\bm h_{\bm k,m}^{\mu,\BB}$ and $\bm h_{\bm j,m}^{\bg,\triangle}$ the squared norms
$$
  \bm h_{\bm k,m}^{\mu, \BB} =  \langle \bm B_{\bm k, m}^{\mu}, \bm B_{\bm k, m}^{\mu}\rangle_{\mu}^\BB
 \quad and \quad \bm h_{\bm j,m}^{\bg,\triangle} = \langle \bm T_{\bm j, m}^{\bg}, \bm T_{\bm j, m}^{\bg}  \rangle_{\bg}^\triangle.
$$

\begin{prop}
For $\mu > - \f12$, $\bk \in \RR^{d_2}$ with $k_j > -1$, $1 \le j \le d_2$, define 
\begin{equation} \label{eq:OP_BT}
 \bm Q_{\bm j, \bm k, m}^n(\bm x,\bm y) = \bm T_{\bm j, n-m}^{(\bk, \alpha+2m)}(\bm y) (1-|\bm y|)^m 
      \bm B_{\bm k, m}^\mu \left(\frac{\bm x}{1-|\bm y| }\right),
\end{equation}
where $\alpha = 2 \mu + d_1-1$. Then
$\{\bm Q_{\bm j, \bm k, m}^n: |\bm k| = m, \, |\bm j| = n-m, \, \bm k \in \NN_0^{d_1}, \, \bm j \in \NN_0^{d_2}, 0\le m \le n\}$ 
is an orthogonal basis of $\CV_n\big(\BB^{d_1}\!\! \rtimes\triangle^{d_2}, \bm W_{\mu, \bk}\big)$. 
Moreover, the squared norm of $\bm Q_{\bm j, \bm k, m}^n$ is given by 
\begin{equation} \label{eq:OP_BT_Norm}
   \bm h_{\bm j, \bm k, m}^{n; \mu,\bk}  
       = \frac{(2\mu+d_1)_{2m}} {(|\bk|+ 2\mu + d_1 + d_2)_{2m}}\bm h_{\bm k,m}^{\mu, \BB}  
       \bm h_{\bm j,n-m}^{(\bk, \a+2m), \triangle}. 
\end{equation}
\end{prop}
 
The orthogonality is a consequence of Proposition \ref{prop:wrapped_orth}. It also follows from 
\begin{align*}
 \left\langle \bm Q_{\bm j, \bm k, m}^n, \bm Q_{\bm j', \bm k', m'}^{n'} \right\rangle_{\mu, \bk} 
& =\bm b_{(\bk,\a)}^\triangle  \left\langle \bm T_{\bm j, n-m}^{(\bk, \a+2m)}, \bm T_{\bm j', n-m}^{(\bk, \a+2m)} 
    \right\rangle_{\bk}^\triangle   \bm h_{\bm k,m}^{\mu,\BB}  \delta_{\bm k, \bm k'} \delta_{m, m'}  \\
& =  \frac{\bm b^\triangle_{(\bk,\a)}}{\bm b^\triangle_{(\bk, \alpha+2m)}} 
  \bm h_{\bm j,n-m}^{(\bk, \a+2m), \triangle} \bm h_{\bm k,m}^{\mu,\BB} 
 \delta_{\bm j, \bm j'} \delta_{\bm k, \bm k'} \delta_{m, m'} 
  \delta_{n,n'}
\end{align*} 
by the orthogonality of $\bm B_{\bm k, m}^{\mu}$ and $\bm T_{\bm j, n-m}^{(\bk, \alpha +2m)}$, which also  
verifies \eqref{eq:OP_BT_Norm}.  

In the case $d_1 = d_2 =1$, the domain is the triangle given by 
$$
[-1,1]\rtimes [0,1]  = \{(x,y) \in  [-1,1] \times [0,1]: |x| \le 1-y\}
$$
and the weight function is 
$$
   \Wb^{[-1,1]\rtimes [0,1]}(x,y) =  y^\k [(1-y)^2 - x^2]^{\mu-\f12} = y^k (1-y-x)^{\mu-\f12} (1-y+x)^{\mu-\f12}.
$$
Under the affine change of variables $(x,y) = (u-v,1-u-v)$, it follows that
$$
\Wb(x,y) = 2^{2\mu-1} u^{\mu-\f12} v^{\mu-12}(1-u-v)^\k =  2^{2\mu-1}\Wb_{\mu-\f12,\mu-\f12,\k}^\triangle(u,v)
$$
becomes the classical Jacobi weight for $(u,v) \in \triangle^2$. As a result, the spectral operator on the triangle
\eqref{eq:D-simplex} becomes in $(x,y)$ variables, 
\begin{align*}
  \fD_{\mu,\k}^{[-1,1]\rtimes [0,1]} = \, & 
(1-y - x^2) \partial_x^2 - 2 x y \partial_x \partial_y +y(1-y) \partial_y^2 \\
 &   - (2\mu + \k +2) x \partial_x  + (k+1 - (2\mu+\k +2) y) \partial_y
\end{align*}
and the identity \eqref{eq:Dk-eigen} gives 
$$
  \fD_{\mu,\k}^{[-1,1]\rtimes [0,1]} u = - n (n + 2\mu + \k +1) u, \quad \forall u \in \CV_n\! \left([-1,1]\rtimes [0,1],  \Wb^{[-1,1]\rtimes [0,1]}\right).
$$ 

\begin{rem} \label{rem:B-T}
If $d_1=d$ and $d_2 =1$, then the domain becomes, using $\triangle^1 = [0,1]$, 
$\BB^{d}\! \rtimes [0,1]  = \left\{(\xb,y):  \|\xb\| \le 1- y \right\}$, which is the rotational cone with base on 
the ball in $\RR^d$ and the apex at $(0, \ldots, 0, 1)$, which was studied in \cite{X20} 
with $t= 1-y$ and the apex of the cone at the origin. If $d_1=1$ and $d_2 = d$, the domain becomes 
$$
[-1,1] \rtimes \triangle^d = \{ (x,\yb): |x| + |\yb| \le 1 \}. 
$$
For $d= 2$, this is the tetrahedron with vertices $(1,0,0)$, $(-1,0,0)$, $(0,1,0)$, and $(0,0,1)$. The domain 
$\BB^1 \rtimes \triangle^d = [-1,1] \rtimes \triangle^d$ is not the standard simplex $\triangle^{d+1}$ because of the
absolute value of the $x$ variable, but it is affine equivalent to $\triangle^{d+1}$.
\end{rem}

We now show that orthogonal polynomials in $\CV_n(\BB^{d_1}\! \rtimes \triangle^{d_2}, \bm W_{\mu, \bk}^{\BB\rtimes \triangle})$ 
are eigenfunctions of a second-order linear differential operator when $d_1, d_2$ are positive integers. Being the first of such
results, we provide a relatively detailed proof. 

\begin{thm}\label{thm:B-Tri}
For $\mu > -\f12$ and $\bk \in \RR^{d_2}$ with $\k_i > -1$, $1 \le i \le d_2$, define
\begin{align} \label{eq:diff_eqnBT}
  \fD_{\mu,\bk}^{\BB\rtimes \triangle} u : = \, &   (1-|\bm y|) \Delta_{\bm x} - \big(\la \bm x, \nabla_{\bm x} \ra + 
   \la \bm y, \nabla_{\bm y} \ra\big)^2  +  \sum_{i=1}^{d_2} \left( y_i \partial_{y_i}^2 + (\kappa_i +1) \partial_{y_i} \right) \\
    & - (|\bk| + 2\mu +d_1+d_2-1) \big(\la \bm x, \nabla_{\bm x} \ra +  \la \bm y, \nabla_{\bm y} \ra\big). \notag
\end{align}
Then $\CV_n(\BB^{d_1}\! \rtimes \triangle^{d_2}, \bm W_{\mu, \bk}^{\BB\rtimes \triangle})$ is an eigenspace of 
$\fD_{\mu,\bk}^{\BB\rtimes \triangle}$ for each $n \in \NN_0$. More precisely, 
\begin{align} \label{eq:eigenBT}
  \fD_{\mu,\bk}^{\BB\rtimes \triangle} u = - n \big( n+ | \bk| + 2 \mu + d_1 +d_2 -1 \big) u, \quad
      \forall u \in \CV_n\big(\BB^{d_1}\! \!\rtimes\!\triangle^{d_2}, \bm W_{\mu, \bk}^{\BB\rtimes \triangle}\big).
\end{align}
\end{thm}

\begin{proof}
It suffices to establish the result for polynomials in the basis \eqref{eq:OP_BT}. Let $u =\bm Q_{\bm j, \bm k, m}^n$. 
To simplify the notation, we further write $u(\bm x,\bm y) = g(\bm y) H(\bm x,\bm y)$ with 
$$
g(\bm y) = \bm T_{\bm j, n-m}^{(\bk, \alpha +2m)}(\bm y) \quad \hbox{and}\quad
  H(\bm x, \bm y) =   (1-|\bm y|)^m \bm B_{\bm k, m}^\mu \left(\frac{\bm x}{1-|\bm y|}\right),
$$
where we recall that $\a = 2\mu + d_1-1$. 

We first observe that, for $\rho(\yb) = 1-|\yb|$, the identity \eqref{eq:diffHa} becomes
\begin{equation} \label{eq:diffH}
     (1-|\bm y|) \partial_{y_i} H - \la \bm x, \nabla_{\bm x}\ra  H  = -m  H,\quad 1 \le i \le d_2.
\end{equation} 
Moreover, taking the derivative with $y_i$ one more time, \eqref{eq:diffH} leads to 
\begin{equation} \label{eq:diffH2}
     (1-|\bm y|) \partial_{y_i}^2 H - \la \bm x, \nabla_{\bm x}\ra \partial_{y_i} H  = -(m-1) \partial_{y_i} H ,
       \quad 1 \le i \le d_2. 
\end{equation}

By \eqref{eq:Dk-eigen}, $g$ is the eigenfunction of the differential operator $\fD_{(\bk, \a+2m)}^{\triangle, (\yb)}$, in which 
$(\bk, \a+2m) \in \RR^{d_2+1}$, and the superscript $\yb$ indicates it is for the $\yb$ variable. We work with the case 
$m =0$ and write the operator, by \eqref{eq:DkSimplex2}, as a sum of two terms, 
\begin{equation} \label{eq:fD+}
\fD_{(\bk, \a)}^{\triangle, (\yb)} = \fD_{\a} + \fR_{\a}, \qquad \fR_\a = \frac1{\bm W_{(\bk,0)}^\triangle(y)} \sum_{1\le i < j \le d_2} \partial_{y_i, y_j} 
  \left(y_i y_j \bm W_{(\bk,0)}^\triangle(\bm y)  \partial_{y_i, y_j} \right),
\end{equation}
where the first term, denoted by $\fD_\a$, is given more explicitly by carrying out the derivatives of the first term in the definition,   
\begin{align*}
 \fD_{\a} & = 
 \frac{1}{\bm W_{(\bk,\a)}^\triangle(\bm y)} 
   \sum_{i=1}^{d_2} \partial_{y_i} \left( y_i (1-|\bm y|) \bm W_{(\bk,\a)}^\triangle(\bm y) \partial_{y_i}\right) \\ 
 & = (1-|\bm y |)  \sum_{i=1}^{d_2} \left(y_i \partial_{y_i}^2 + (\k_i +1) \partial_{y_i} \right)  - (\a +1) 
     \la \bm y, \nabla_{\bm y} \ra,
\end{align*}
whereas in the second  term we have replaced $\bm W_{(\bk,\a)}^\triangle$ by $\bm W_{\bk,0}^\triangle$ since the
factor $(1-|\bm y|)^{\a}$ can be pulled out in front of the derivative $\partial_{y_i, y_j}$ upon using $\partial_{y_i, y_j} f(|\bm y|) = 0$
for all $f: \RR_+ \mapsto \RR$.  

Taking the derivatives of $u = g H$, we obtain 
\begin{align*}
 &  (1-|\bm y|) \sum_{i=1}^{d_2} y_i  \partial_{y_i}^2 u - 2 \la \bm x, \nabla_{\bm x} \ra  \la \bm y, \nabla_{\bm y} \ra u 
  = (1-|\bm y|) \sum_{i=1}^{d_2} y_i  \partial_{y_i}^2 g \cdot H \\
   & \quad + 2 \sum_{i=1}^{d_2} y_i \partial_{y_i} g  \left[ (1-|\bm y|) \partial_{y_i} 
    -  \la \bm x, \nabla_{\bm x} \ra\right] H 
      + g \sum_{i=1}^{d_2} y_i  \left[ (1-|\bm y|) \partial_{y_i}^2 - 2 \la \bm x, \nabla_{\bm x}\ra  \partial_{y_i}\right] H \\
     &  =(1-|\bm y|) \sum_{i=1}^{d_2} y_i  \partial_{y_i}^2 g \cdot H  - 2 m \la \bm y, \nabla_{\bm y} \ra g \cdot H
     - g \cdot (\la \bm x, \nabla_{\bm x}\ra +m-1) \la \bm y, \nabla_{\bm y}\ra  H,
\end{align*}
where we have used \eqref{eq:diffH} and \eqref{eq:diffH2} in the second step. We now add three more identities
according to \eqref{eq:fD+} and the definition of $\fD_\a$,  so that $\fD_{(\bk,\a)}^\triangle$ appears on the 
left-hand side. The first one is 
$$
(1-|\bm y |)  \sum_{i=1}^{d_2} (\k_i +1) \partial_{y_i} u = 
(1-|\bm y |)  \sum_{i=1}^{d_2} (\k_i +1) \partial_{y_i} g \cdot H + g (|\bk|+ d_2) (\la \bm x, \nabla_{\bm x} \ra -m)H,
$$
where we have used \eqref{eq:diffH} to derive the second term on the right-hand side, and the second one is 
$$
(\a +1) \la \bm y, \nabla_{\bm y} \ra u = (\a +1)  \la \bm y, \nabla_{\bm y} \ra g \cdot H + g (\a +1) \la \bm y, \nabla_{\bm y} \ra H,
$$
whereas the third one is the last term on the right-hand side of \eqref{eq:fD+}, which is zero when acting on $H$ as 
$(\partial_{y_i} -\partial_{y_j})F(|\bm y|) = 0$ for all $i< j$, and it follows immediately that 
$\fR_\a u = \fR_\a g\cdot H = \fR_{\a+2m} g \cdot H$, where the second step follows as we can pull the factor 
$(1-|\bm y|)^{2m}$ in front of $\partial_{y_i} -\partial_{y_j}$ when needed. Putting together, we obtain
\begin{align*}
\fD_{(\bk,\a)}^{\triangle, (\yb)} u  -  2 \la \bm x, \nabla_{\bm x} \ra  \la \bm y, \nabla_{\bm y} \ra u 
  & =  \fD_{(\bk, \a + 2m)}^{\triangle, (\yb)} \, g \cdot H  - g \cdot (\la \bm x, \nabla_{\bm x}\ra +m-1) \la \bm y, \nabla_{\bm y}\ra  H\notag \\
 & + g\cdot (|\bk|+ d_2) (\la \bm x, \nabla_{\bm x} \ra -m)H - g (\a +1) \la \bm y, \nabla_{\bm y} \ra H. 
\end{align*}
Since $g$ is the eigenfunction of $\fD_{(\bk, \a + 2m}^{\triangle, (\yb)}$, it follows by \eqref{eq:Dk-eigen} that 
\begin{align} \label{eq:diff-gH}
\fD_{(\bk,\a)}^{\triangle, (\yb)} u  -  2 \la \bm x, \nabla_{\bm x} \ra  & \la \bm y, \nabla_{\bm y} \ra u 
 =  - (n-m) (n+ m + |\bk|+ \a +d_2) u \\
  & + (|\bk|+ d_2)\left (\la \bm x, \nabla_{\bm x} \ra -m\right )u  - g  \fE(H), \notag
\end{align}
where $\fE(H)$ is given by 
$$
  \fE(H) =  (\la \bm x, \nabla_{\bm x} \ra +m + \a) \la \bm y, \nabla_{\bm y} \ra H.
$$
Now, taking the derivatives in $\bm x$, it follows from \eqref{eq:D-ball} that 
\begin{equation} \label{eq:diffH3}
  \left[ (1-|\bm y|)^2 \Delta_{\bm x} - \la \bm x, \nabla_{\bm x} \ra^2 - (2\mu + d_1 -1) \la \bm x, \nabla_{\bm x} \ra \right ]  H = 
    -m(m+2\mu+d_1-1) H. 
\end{equation}
Multiplying $\fE(H)$ by $1-|\bm y|$ and using \eqref{eq:diffH} and \eqref{eq:diffH3}, recall $\a = 2\mu + d_1 -1$, we deduce
\begin{align*}
 g (1-|\bm y|) \fE(H)
 \, & = g |\bm y| ( \la \bm x, \nabla_{\bm x} \ra +m + \a) (\la \bm x, \nabla_{\bm x} \ra -m) H \\
 & = g |\bm y| (\la \bm x, \nabla_{\bm x} \ra^2 + \a \la \bm x, \nabla_{\bm x} \ra - m(m+\a) )H 
 =  |\bm y| (1-|\bm y|)^2 \Delta_x u,
\end{align*}
which shows $g \fE(H) =  |\bm y| (1-|\bm y|) \Delta_x u$. Consequently, rearranging terms, we obtain
\begin{align*}
 \fD_{(\bk,\a)}^{\triangle, (\yb)} u  - 2 \la \bm x, \nabla_{\bm x} \ra &  \la \bm y, \nabla_{\bm y} \ra u 
+  |\bm y| (1-|\bm y|) \Delta_x u - (|\bk|+ d_2) \la \bm x, \nabla_{\bm x} \ra u  \\
& =  - (n-m) (n+ m + |\bk|+ \a +d_2) u  - m(|\bk|+ d_2) u.
\end{align*}
Adding \eqref{eq:diffH3} to this identity and observing that
$$
 (n-m) (n+ m + |\bk|+ \a +d_2)  + m(|\bk|+ d_2) + m(m + \a) =  n(n+|\bk| + \a + d_2),
$$
we then conclude 
\begin{align*}
    \fD_{(\bk, \a)}^{\triangle, (\yb)} u  & - 2 \la \bm x, \nabla_{\bm x} \ra  \la \bm y, \nabla_{\bm y} \ra u 
+ (1-|\bm y|) \Delta_x u - \la \bm x, \nabla_{\bm x} \ra^2 u\\
& - (|\bk|+ d_2+\a) \la \bm x, \nabla_{\bm x} \ra u   = - n(n+|\bk| + \a + d_2)u
\end{align*}
after simplification. The left-hand side of the above identity is  $\fD_{\mu,\bk}^{\BB\rtimes \triangle}$, which can be stated 
in the more explicit form given in \eqref{eq:diff_eqnBT} since, using the identity
\begin{equation}\label{eq:double_sum}
2 \sum_{i< j} x_i x_j \partial_{x_i} \partial_{x_j} = \la \xb, \nabla_\xb\ra^2 - \la \xb,\nabla_\xb\ra -\sum_{i=1}^{d_1} x_i^2
 \partial_{x_i}^2
 \end{equation}
to rewrite $\fD_{\bk}^{\triangle}$ in \eqref{eq:D-simplex} leads to 
\begin{align} \label{eq:Dsimplex2}
  \fD_{\bk}^{\triangle} =  \sum_{i=1}^{d} \left[ y_i \partial_{y_i}^2 + (\k_i+1) \partial_{y_i} \right] -  \la \bm y, \nabla_{\bm y} \ra^2
   - (|\bk|+d) \la \bm y, \nabla_{\bm y} \ra .
\end{align}
Using this expression with $\bk$ replaced by $(\bk, \a)$, we obtain \eqref{eq:eigenBT} after further combining terms. 
This completes the proof. 
\end{proof}

It is worth pointing out that the operator $\fD_{\mu, \bk}^{\BB\rtimes \triangle}$ has the operators  $\fD_\mu^\BB$ in
$\xb$ variables and $\fD_\bk^\Delta$ in $\bm y$ variables, which we denote as $\fD_\bk^{\BB, (\bm x)}$ 
and $\fD_\bk^{\Delta, (\bm y)}$, as summands. Indeed, using \eqref{eq:diff_eqnBT} 
via \eqref{eq:fD+} with $\a = 0$, it follows that 
\begin{align*}
   \fD_{\mu,\bk}^{\BB\rtimes \triangle} = \fD_\mu^{\BB, (\bm x)}+\fD_\bk^{\Delta, (\bm y)}
   \, & -|\bm y| \Delta_{\bm x}  - (|\bk| +d_2)  \la \bm x, \nabla_{\bm x} \ra   \\
   \,&  -2 \la \bm x, \nabla_{\bm x} \ra \la \bm y, \nabla_{\bm y} \ra 
        - (2\mu+d_1-1)\la \bm y, \nabla_{\bm y} \ra. \notag 
\end{align*}

The eigenvalue of \eqref{eq:diff_eqnBT} depends only on the degree of orthogonal polynomials, so that it holds for all 
elements of $\CV_n(\BB^{d_1} \rtimes \triangle^{d_2}, \bm W_{\mu,\bk}^{\BB\rtimes \triangle})$. This implies, in particular, that the operator 
$\fD_{\mu,\bk}^{\BB\rtimes \triangle}$ is self-adjoint in the space $L^2\left(\BB^{d_1} \rtimes \triangle^{d_2}, 
\bm W_{\mu,\bk}^{\BB\rtimes \triangle}\right)$. We can also write the operator in a more structural form, which makes 
the self-adjointness transparent. 

\begin{thm}\label{cor:diff-eqnV}
Let  $\mu > -\f12$, $\k_i > -1$ for $1 \le i \le d_2$. Let $\fD_i$ denote the operator 
$$
  \fD_i  := (1-|\bm y|) \partial_{y_i} - \la \bm x, \nabla_{\bm x} \ra, \qquad 1 \le i \le d_2.
$$
Then the spectral operator can be written as 
\begin{align} \label{eq:diff-eqn2}
  \fD_{\mu, \bk}^{\BB\rtimes \triangle} = \, & \frac{1}{1-|\bm y|} \left[   \frac{1}{\bm W_{\mu,\bk}^{\BB\rtimes \triangle}(\bm x, \bm y)}
  \sum_{i=1}^{d_1} \partial_{x_i} \bm W_{\mu+1,\bk}^{\BB\rtimes \triangle}(\bm x, \bm y) \partial_{x_i} 
  +  \sum_{1 \le i< j \le d_1} \left(D_{i,j}^{(\bm x)}\right)^2 \right] \notag \\
  & + \frac{1}{\bm W_{\mu,\bk}^{\BB\rtimes \triangle}(\bm x, \bm y)}  
  \sum_{1\le i< j \le d_2} (\partial_{y_i} - \partial_{y_j}) \left( y_i y_j \bm W_{\mu,\bk}^{\BB\rtimes \triangle}(\bm x, \bm y)
   (\partial_{y_i} - \partial_{y_j}) \right) \\
   & +  \frac{1}{(1-|\bm y|)^{d_1+1}} \frac{1}{\bm W_{\mu,\bk}^{\BB\rtimes \triangle}(\bm x, \bm y)}  \sum_{i=1}^{d_2}  
\fD_i  (1-|\bm y|)^{d_1} y_i \bm W_{\mu,\bk}^{\BB\rtimes \triangle}(\bm x, \bm y) \fD_i. \notag 
 \end{align}
 \end{thm}

\begin{proof}
The main task is to identify the correct factorization. This is, however, by no means straightforward, whereas the verification is 
straightforward once the form is identified. We record several intermediate identities. The first one follows from 
\eqref{eq:Diff-ball2} by a simple change of variables, as the part of the weight function that depends only on $\yb$ cancels out. 
\begin{align*}
 \frac{1}{1-|\bm y|} & \left [
  \frac{1}{\bm W_{\mu,\bk}^{\BB\rtimes \triangle}(\bm x, \bm y)} \right. 
   \left. \sum_{i=1}^{d_1} \partial_{x_i} \bm W_{\mu+1,\bk}^{\BB\rtimes \triangle}(\bm x, \bm y) \partial_{x_i} 
   +  \sum_{1 \le i< j \le d_1} \left(D_{i,j}^{(\bm x)}\right)^2 \right]\\
& =  \frac{-1}{1-|\bm y|} \left[ \la \bm x, \nabla_{\bm x} \ra^2+(2\mu + d_1 -1) \la \bm x, \nabla_{\bm x} \ra   \right]+  (1-|\bm y|) \Delta_{\bm x}.
\end{align*}
The second one can be verified by carrying out derivatives and simplifying, 
\begin{align*}
 \frac{1}{\bm W_{\mu,\bk}^{\BB\rtimes \triangle}(\bm x, \bm y)}&  \sum_{1\le i< j \le d_2} (\partial_{y_i} - \partial_{y_j}) 
 \left(y_i y_j \bm W_{\mu,\bk}^{\BB\rtimes \triangle}(\bm x, \bm y)   (\partial_{y_i} - \partial_{y_j}) \right) \\
& = |\bm y|  \sum_{i=1}^{d_2} \left(y_i \partial_{y_i}^2 + (\k_i+1) \partial_{y_i} \right) - (|\bk| + d_2) \la \bm y, \nabla_{\bm y} \ra 
- \la \bm y, \nabla_{\bm y} \ra^2.
\end{align*}
Moreover, using the observation that 
$$
   \fD_i \left((1-|\bm y|)^2 - \|x\|^2\right)^{\mu - \f12} = - (2\mu -1) \left((1-|\bm y|)^2 - \|x\|^2\right)^{\mu - \f12},
$$
we can further deduce, by carrying out derivatives and simplifying, that  
\begin{align*}
 & \frac{1}{(1-|\bm y|)^{d_1+1}} \frac{1}{\bm W_{\mu,\bk}^{\BB\rtimes \triangle}(\bm x, \bm y)}  \sum_{i=1}^{d_2}  
\fD_i  (1-|\bm y|)^{d_1} y_i \bm W_{\mu,\bk}^{\BB\rtimes \triangle}(\bm x, \bm y) \fD_i \\
&\qquad = (1-|\bm y|) \sum_{i=1}^{d_2} (\k_i+1) \partial_{y_i} - (|\bk|+d_2) \la \bm x, \nabla_{\bm x} \ra \\
  &\qquad \quad 
  - (2\mu + d_1-1) \left( \la \bm y, \nabla_{\bm y} \ra  - \frac{|\bm y|} {1-|\bm y|}  \la \bm x,  \nabla_{\bm x} \ra\right)  
  + \frac{1}{1-|\bm y|} \sum_{i=1}^{d_2}y_i \fD_i^2 \\
&\qquad = (1-|\bm y|) \sum_{i=1}^{d_2} \left(y_i \partial_{y_i}^2 + (\k_i+1) \partial_{y_i} \right) - (2\mu + d_1) \la \bm y, \nabla_{\bm y} \ra 
- 2 \la \bm x, \nabla_{\bm x} \ra \la \bm y, \nabla_{\bm y} \ra  \\
& \qquad \quad - (|\bk| + d_2) \la \bm x, \nabla_{\bm x}\ra + \frac{|\bm y|}{1-|\bm y|} \left(\la \bm x, \nabla_{\bm x} \ra^2+ (2\mu+d_1-1) \la \bm x, \nabla_{\bm x}\ra  \right). 
\end{align*}
Adding the three identities and comparing the results with \eqref{eq:diff_eqnBT} verifies the new expression of 
$\fD_{\mu,\bk}^{\BB\rtimes \triangle}$.  
\end{proof}

Comparing with \eqref{eq:Diff-ball2} and \eqref{eq:DkSimplex2}, we can see traces of the differential operators of 
$\fD_\mu^\BB$ and $\fD_\bk^\triangle$ in the above decomposition. As we mentioned in Section 2, the decompositions 
for $\fD_\mu^\BB$ and $\fD_\bk^\triangle$ are not unique. There are likely other decompositions for 
$\fD_{\mu,\bk}^{\BB\rtimes \triangle}$; we leave them to interested readers. 

\begin{cor}\label{cor:self-adjBT}
Let  $\mu > -\f12$, $\k_i > -1$ for $1 \le i \le d_2$.  Then
\begin{align} \label{eq:self-adjBT}
   & - \int_{\BB^{d_1} \rtimes \triangle^{d_2}}\fD_{\mu, \bk}^{\BB\rtimes \triangle}
        f (\bm x,\bm y) \cdot g(\bm x, \bm y) \bm W_{\mu, \bk}(\bm x, \bm y)^{\BB\rtimes \triangle} \d \bm x \d \bm y  \\
  &  = \int_{\BB^{d_1} \rtimes \triangle^{d_2}}  \frac{1}{1-|\bm y|} \sum_{i=1}^{d_1} \partial_{x_i} f(\bm x,\bm y)\cdot \partial_{x_i} g(\bm x,\bm y) 
  \bm W_{\mu+1,\bk}^{\BB\rtimes \triangle}(\bm x,\bm y)  \d \bm x \d \bm y  \notag \\ 
  & +  \int_{\BB^{d_1} \rtimes \triangle^{d_2}}  \frac{1}{1-|\bm y|} \sum_{1 \le i< j \le d_1} y_iy_j D_{i,j}^{(\bm x)} f(\bm x, \bm y)  D_{i,j}^{(\bm x)} g(\bm x, \bm y) 
  \bm W_{\mu,\bk}^{\BB\rtimes \triangle}(\bm x,\bm y)  \d \bm x \d \bm y  \notag \\ 
  & + \int_{\BB^{d_1} \rtimes \triangle^{d_2}}   \sum_{1\le i< j \le d_2} (\partial_{y_i} - \partial_{y_j}) f(\bm x,\bm y) \cdot 
    (\partial_{y_i} - \partial_{y_j})g(\bm x,\bm y)
  \bm W_{\mu,\bk}^{\BB\rtimes \triangle}(\bm x,\bm y) \d \bm x \d \bm y  \notag \\
  &  +  \int_{\BB^{d_1} \rtimes \triangle^{d_2}} \sum_{i=1}^{d_2} \frac{y_i}{(1-|\bm y|)}  \fD_i f(\bm x,\bm y) 
 \fD_i g(\bm x,\bm y) 
  \bm W_{\mu,\bk}^{\BB\rtimes \triangle}(\bm x,\bm y)  \d \bm x \d \bm y.  \notag
 \end{align}
\end{cor} 

\begin{proof}
The proof of \eqref{eq:self-adjBT} follows by applying integration by parts to \eqref{eq:diff-eqn2}. The first 
three terms on the right-hand side of \eqref{eq:self-adjBT} follow from straightforward integration by parts 
of the first two terms on the right-hand side of \eqref{eq:diff-eqn2}. 
The verification of the fourth term on the right-hand side of \eqref{eq:self-adjBT} requires the integration by 
parts of the operator $\fD_i$, given as follows, 
\begin{align*} 
 & \int_{\BB^{d_1} \rtimes \triangle^{d_2}}\fD_i F(\bm x,\bm y) \cdot \frac{g(\bm x,\bm y)}{(1-|\bm y|)^{d_1+1}} \d \bm x \d \bm y \\
  = & -   \int_{\BB^{d_1} \rtimes \triangle^{d_2}} F(\bm x,\bm y) \cdot \left (\frac{\partial_{y_i} g(\bm x,\bm y)}{(1-|\bm y|)^{d_1}} 
       + d_1 \frac{g(\bm x,\bm y)}{(1-|\bm y|)^{d_1+1}} \right)  \d \bm x \d \bm y  \\
 &  + \int_{\BB^{d_1} \rtimes \triangle^{d_2}} \frac{F(\bm x,\bm y)}{(1-|\bm y|)^{d_1+1}} \cdot \left (d_1 g(\bm x,\bm y) 
      + \la \bm x, \nabla_{\bm x}\ra g(\bm x,\bm y) \right)  \d \bm x \d \bm y  \\
  = & - \int_{\BB^{d_1} \rtimes \triangle^{d_2}} \frac{F(\bm x,\bm y)}{(1-|\bm y|)^{d_1+1}} \cdot \fD_i g(\bm x,\bm y) \d \bm x \d \bm y.
\end{align*}
Setting $F(\bm x, \bm y) = (1-|\bm y|)^{d_1} y_i \bm W_{\mu,\bk}^{\BB\rtimes \triangle}(\bm x, \bm y) \fD_i f(\bm x,\bm y)$ in the above 
identity gives the desired integration by parts. The proof is completed. 
\end{proof}

As a simple consequence of the above theorem, we can derive sharp Bernstein inequalities for polynomials in the 
$L^2$ norm. Let 
$$
  \l_n^{\mu,\bk} = n (n+|\bk| + 2 \mu + d_1 +d_2-1). 
$$
Denote by $\|\cdot\|_{\mu,\bk}$ the norm of $L^2\big(\BB^{d_1}\! \!\rtimes\!\triangle^{d_2}, \bm W_{\mu, \bk}^{\BB\rtimes \triangle}\big)$. 

\begin{thm} \label{thm:Bernstein}
Let  $\mu > -\f12$, $\k_1,\ldots, k_{d_2} > -1$. Then for every polynomial $f$ of degree at most $n$, 
\begin{align}\label{eq:B1}
  &  \sum_{i=1}^{d_1}  \left \| \frac{\sqrt{(1-|\bm y|)^2- \|\bm x\|^2}}{\sqrt{1-|\bm y|}} \partial_{x_i} f \right \|_{\mu,\bk}^2  
   + \sum_{1 \le i< j \le d_1}  \left\|   \frac{1}{\sqrt{1-|\bm y|}} D_{i,j}^{(\bm x)} f \right \|_{\mu,\bk}^2 \\
  &+  \sum_{1\le i< j \le d_2} \left \| \sqrt{y_i y_j} (\partial_{y_j} - \partial_{y_j}) f  \right \|_{\mu,\bk}^2 +
     \sum_{i=1}^{d_2} \left \|  \frac{\sqrt{y_i}}{\sqrt{1-|\bm y|}}  \fD_i f \right \|_{\mu,\bk}^2 \le \l_n^{\mu,\bk} \| f \|_{\mu,\bk}^2.  \notag
\end{align}
Moreover, the inequality is sharp in the sense that the equality is attained by some polynomial of degree $n$, and 
so is the inequality  
\begin{align}\label{eq:B2}
  \sum_{1\le i< j \le d_2} \left \|\sqrt{y_i y_j} (\partial_{y_i} - \partial_{y_j}) f  \right \|_{\mu,\bk}^2 +
     \sum_{i=1}^{d_2} \left \| \frac{\sqrt{y_i}}{\sqrt{1-|\bm y|}}  \fD_i f \right \|_{\mu,\bk}^2 \le \l_n^{\mu,\bk} \| f \|_{\mu,\bk}^2.
\end{align}
\end{thm}

\begin{proof}
Writing $f = \sum_{m=0}^n \proj_m^{\mu,\bk} f$ and using $\proj_m^{\mu,\bk} f \in
 \CV_m\big(\BB^{d_1}\! \!\rtimes\!\triangle^{d_2}, \bm W_{\mu, \bk}^{\BB\rtimes \triangle}\big)$,
the spectral equation \eqref{eq:eigenBT} leads to 
$$
 \fD_{\mu, \bk}^{\BB\rtimes \triangle} f =  \sum_{m=0}^n \fD_{\mu, \bk}^{\BB\rtimes \triangle} \proj_m^{\mu,\bk}f
     =  - \sum_{m=0}^n \l_m^{\mu,\bk} \proj_m^{\mu,\bk}f, 
$$
which leads immediately to 
$$
  \left \| \fD_{\mu, \bk}^{\BB\rtimes \triangle} f  \right \|_{\mu,\bk} \le
    \l_n^{\mu,\bk} \left( \sum_{m=0}^n   \left \| \proj_m^{\mu,\bk} f \right\|_{\mu,\bk}^2
  \right)^{\f12} 
     = \l_n^{\mu,\bk} \|f\|_{\mu,\bk},
$$
where the last step follows from the Parseval identity. Now, setting $g = f$, the right-hand side of \eqref{eq:self-adjBT} becomes 
a sum of four terms, each of which can be written as a sum of $L^2$ norms, which agrees with the left-hand sides of the
inequality \eqref{eq:B1}, whereas the integral on the left-hand side is bounded by
\begin{align*}
  \left |  \int_{\VV^{d_1,d_2}}\fD_{\mu, \bk}^{\BB\rtimes \triangle} f (\bm x,\bm y) \cdot g(\bm x, \bm y)
   \bm W_{\mu, \bk}^{\BB\rtimes \triangle}(\bm x, \bm y) \d \bm x \d \bm y 
      \right| 
    \le  \left \| \fD_{\mu, \bk}^{\BB\rtimes \triangle} f \right\|_{\mu,\bk} \|f\|_{\mu,\bk} 
    \le \l_n^{\mu,\bk} \|f\|_{\mu,\bk}^2.
\end{align*}
Putting these together establishes the inequality \eqref{eq:B1}, from which the second inequality \eqref{eq:B2} follows trivially. 
The inequality \eqref{eq:B1} is sharp for every polynomial in
 $\CV_n(\VV^{d_1,d_2}, \bm W_{\mu, \bk}^{\BB\rtimes \triangle})$, and the inequality 
\eqref{eq:B2} is sharp for $f(\bm x, \bm y) = \bm T_{\bm j, n}^{(\bk, \a)}(\bm y)$, which is $\bm Q_{\bm j, \bm 0, 0}^n$ in 
\eqref{eq:OP_BT}, as can be easily verified. 
\end{proof}

\section{Ball-Laguerre polynomials $\BB^{d_1} \rtimes \RR_\Sigma^{d_2}$}
\setcounter{equation}{0}

In this section we consider the wrapped product $\BB^{d_1} \rtimes \RR_\Sigma^{d_2}$ with $\rho(y) = |\yb|$, 
\begin{align*}
\BB^{d_1} \rtimes \RR_\Sigma^{d_2} \, & = \left\{(\xb,\yb): \frac{\xb}{|\yb|} \in \BB^{d_1}, \quad \yb \in \RR_\Sigma^{d_2}\right\} 
        = \left\{(\xb,\yb) \in \RR^{d_1} \times \RR_\Sigma^{d_2}:  \| \xb \| \le |\yb|\right\}
\end{align*}
and consider orthogonal polynomials with respect to the inner product defined by
\begin{equation} \label{eq:ipdB-Sig}
  \la f, g\ra_{\mu, \bk}^{\BB\rtimes \Sigma}  = \bm b_{\mu,\bk}^{\BB\rtimes \Sigma}
   \int_{\BB^{d_1} \rtimes \RR_\Sigma^{d_2}} f(\bm x,\bm y) g(\bm x,\bm y) 
  \bm W^{\BB\rtimes \Sigma}_{\mu,\bk}(\bm x,\bm y) \d \bm x \d \bm y, 
\end{equation}
where $ \bm b_{\mu,\bk}^{\BB\rtimes \Sigma}$ is the  normalization constant so that $ \la 1, 1\ra_{\mu, \bk} =1$ and the weight 
function $\Wb_{\mu, \bk}^{\BB\rtimes \Sigma}$ for $\mu > -\f12$ and $\bk \in \RR^{d_2-1}$ with $\k_i > -1$, $1\le i \le d_2-1$, 
is of the form \eqref{eq:wrapW} and defined by 
\begin{align*}
\Wb_{\mu, \kb}^{\BB\rtimes \Sigma}(\xb,\yb)  = \left[ |\bm y|^2- \|\bm x\|^2 \right]^{\mu - \f12}  \prod_{i=1}^{d_2-1} y_i^{\k_i} \e^{-|\yb'| -|\yb|}
  = \Wb_{\mu}^{\BB} \left( \frac{x}{|\bm y|} \right) \Wb_{(\kb, 2\mu-1)}^\Sigma(\yb).
\end{align*}

Using \eqref{eq:integra;=} with $\bm x = | \bm y| \bm u$, it follows readily that the normalization constant 
$\bb_{\mu,\bk}^{\BB\rtimes \Sigma}$ can be given in terms of the normalization constant $\bm b_\mu^\BB$ 
in \eqref{eq:bB} for $\bm W_\mu^\BB$ and the constant $\bm b_{\bm \gamma}^L$ for $\bm W_{\bm \gamma}^L$,
which remains the same for $\bm W_\bg^\Sigma$. More precisely, 
\begin{equation} \label{eq:b_B-Sig}
  \bm b_{\mu,\bk}^{\BB\rtimes \Sigma} = \bm b^\BB_\mu \times \bm b^L_{\bg} 
       \quad\hbox{with}\quad \bm \gamma = (\bk, d_1+ 2\mu -1) \in \RR^{d_2}. 
\end{equation}

The wrapped orthogonal polynomials \eqref{eq:wrapOP} are now given in terms of classical orthogonal polynomials 
on the unit ball and alternative Laguerre polynomials on $\RR_\Sigma^{d_2}$. Let 
$\{\bm B_{\bm k, m}^{\mu}: |\bm k|= m, \, \bm k \in \NN_0^{d_1}\}$ be an orthogonal basis for 
$\CV_m(\BB^{d_1},\bm W^\BB_\mu)$, and let $\{\hat \Lb_{\bm j, m}^{\bg}: |\bm j| = m, \, \bm j\in \NN_0^{d_2}\}$ 
be an orthogonal basis for $\CV_{m}\left(\RR_\Sigma^{d_2},\bm W^\Sigma_{\bg}\right)$. 
We denote by $\bm h_{\bm k,m}^{\mu,\BB}$ and $\bm h_{\bm j,m}^{\bg,\Sigma}$ the squared norms
$$
  \bm h_{\bm k,m}^{\mu, \BB} =  \langle \bm B_{\bm k, m}^{\mu}, \bm B_{\bm k, m}^{\mu}\rangle_{\mu}^\BB
 \quad and \quad \bm h_{\bm j,m}^{\bg,\Sigma} = \big \langle \hat \Lb_{\bm j, m}^{\bg}, \hat \Lb_{\bm j, m}^{\bg} \big \rangle_{\bg}^\Sigma.
$$
Then the orthogonality in Proposition \ref{prop:wrapped_orth} becomes the following:

\begin{prop}
For $\mu > - \f12$, $\bk \in \RR^{d_2-1}$ with $\k_j > -1$, $1 \le j \le d_2-1$, define 
\begin{equation} \label{eq:OP_BSig}
 \bm Q_{\bm j, \bm k, m}^n(\bm x,\bm y) = \hat \Lb_{\bm j, n-m}^{(\bk, \alpha+2m)}(\bm y) |\bm y|^m 
      \bm B_{\bm k, m}^\mu \left(\frac{\bm x}{|\bm y|}\right),
\end{equation}
where $\alpha = 2 \mu + d_1-1$. Then
$\{\bm Q_{\bm j, \bm k, m}^n: |\bm k| = m, \, |\bm j| = n-m, \, \bm k \in \NN_0^{d_1}, \, \bm j \in \NN_0^{d_2}, 0\le m \le n\}$ 
is an orthogonal basis of $\CV_n\big(\BB^{d_1}\!\! \rtimes \RR_\Sigma^{d_2}, \bm W_{\mu, \bk}^{\BB\rtimes \Sigma}\big)$. 
Moreover, the squared norm of $\bm Q_{\bm j, \bm k, m}^n$ is given by 
\begin{equation} \label{eq:OP_BSig_Norm}
   \bm h_{\bm j, \bm k, m}^{n; \mu,\bk}  
       =  (2\mu+d_1)_{2m} \bm h_{\bm k,m}^{\mu, \BB}  
       \bm h_{\bm j,n-m}^{(\bk, \a+2m), \Sigma}. 
\end{equation}
\end{prop}
  
The norm \eqref{eq:OP_BSig_Norm} is computed exactly as in the proof of \eqref{eq:OP_BT_Norm}.  

If $d_1 =1$ and $d_2 =1$, then $\BB^1 = [-1,1]$ and $\RR_\Sigma^1 = \RR_+$ so that 
$\BB^1 \rtimes \RR_\Sigma^1$ becomes
$$
       [-1,1] \rtimes \RR_+ = \{(x,y): |x| \le y < + \infty\}
$$
which is an unbounded domain affine equivalent to $\RR_+^2$. If $d_1 = d$ and $d_2 =1$, then the
domain is $\BB^{d} \rtimes \RR_+ = \{(\xb, y) \in \RR^d \times \RR_+: \|\xb\| \le y\}$, which is an infinite
rotational cone in $\RR^{d+1}$, and orthogonal polynomials on this domain were studied in \cite{X20}. 

\begin{rem} \label{rem:B-Sig}
If $d_1 =1$ and $d_2 = d \ge 2$, then the domain is 
$$
 [-1,1] \rtimes \RR_\Sigma^d = \left \{(x, \yb) \in \RR \times \RR_\Sigma^d: |x| \le |\yb| \right\}
$$
is affine equivalent to $\RR_+^{d+1}$ under $(x,\yb) \mapsto \vb = (y_1,\ldots, y_{d-1}, \frac12(|\yb|- x), \frac12 (|\yb|+x))$
as can be easily verified.  

We note that the domain $[-1,1] \rtimes \RR_+^d = \left \{(x, \yb) \in \RR \times \RR_+^d: |x| \le |\yb| \right\}$,
with $\RR_+^d$ instead of $\RR_\Sigma^d$, is a polyhedral domain but it is not affine equivalent to $\RR_+^{d+1}$. 
For $d = 2$, for example, it is a polyhedral (square) cone symmetric about the $y_1 y_2$ plane. 
Orthogonal polynomials on this domain are not wrapped products of classical orthogonal polynomials, 
nor do they possess a second-order spectral operator.
\end{rem}

Our main result shows that orthogonal polynomials for $\Wb_{\mu, \kb}^{\BB\rtimes \Sigma}$ are 
eigenfunctions of a second-order differential operator. 

\begin{thm}\label{thm:B-Sig}
For $\mu > -\f12$ and $\bk \in \RR^{d_2-1}$ with $\k_i > -1$, $1 \le i \le d_2-1$, define
\begin{align} \label{eq:B-Sig}
  \fD_{\mu,\bk}^{\BB\rtimes \Sigma} u : = \, &   |\bm y| \Delta_{\bm x} + 2 \la \bm x, \nabla_{\bm x} \ra \partial_{y_d} - 
   \la \bm x, \nabla_{\bm x} \ra 
      +  \sum_{i=1}^{d_2} \left( y_i \partial_{y_i}^2 + (\kappa_i +1 - y_i) \partial_{y_i} \right)  \\
   &  + 2 |\yb| \partial_{y_d}^2 - 2 \la \bm y, \nabla_{\bm y} \ra \partial_{y_d} + (2\mu + d_1 - |\bk| - d_2) \partial_{y_d}. \notag
\end{align}
Then $\CV_n\big(\BB^{d_1} \rtimes \RR_\Sigma^{d_2}, \bm W_{\mu, \bk}^{\BB\rtimes \Sigma}\big)$ is an eigenspace of 
$ \fD_{\mu,\bk}^{\BB\rtimes \Sigma}$ for each $n \in \NN_0$. More precisely, 
\begin{align} \label{eq:B-Sig-eigen}
  \fD_{\mu,\bk}^{\BB \rtimes \Sigma} u = - n u, \qquad \forall 
         u \in \CV_n\big (\BB^{d_1} \rtimes \RR_\Sigma^{d_2}, \bm W_{\mu, \bk}^{\BB\rtimes \Sigma}\big).
\end{align}
\end{thm}

\begin{proof}
It is sufficient to establish \eqref{eq:B-Sig-eigen} for $u = \Qb_{\jb,\kb,m}^n$ in \eqref{eq:OP_BSig}. For convenience, we denote 
$u = g H$ with 
$$
g = \hat \Lb_{\bm j, n-m}^{(\bk, \alpha+2m)}(\bm y) \quad \hbox{and}\quad 
H =  |\bm y|^m \bm B_{\bm k, m}^\mu \left(\frac{\bm x}{|\bm y|}\right).
$$ 
For $\rho(\yb) = |\yb|$, the identity \eqref{eq:diffHa} becomes $|\yb| \partial_{y_i} H = m H - \la \xb, \nabla_\xb \ra H$ for $1 \le i \le d_2$ and, moreover, \eqref{eq:diffHb} becomes 
\begin{equation} \label{eq:H-iden}
   \la \xb, \nabla_\xb \ra H +  \la \yb, \nabla_\yb \ra H = m H \quad \hbox{and}\quad  |\yb| \partial_{y_i} H = \la \yb, \nabla_{\yb} \ra H,
    \quad 1 \le i \le d_2, 
\end{equation}
where the second identity follows from combining the two previous identities. Taking one more derivative, we obtain
\begin{equation} \label{eq:H-iden2}
  | \yb| \partial_{y_i}^2 H + \la \xb, \nabla_\xb \ra \partial_{y_i} H = (m-1)\partial_{y_i} H. 
\end{equation}
Recall that $\hat \Lb_{\bm j, n-m}^{(\bk, \alpha)}$ are eigenfunctions of the operator $\fD_{(\bk,\alpha)}^\Sigma$ as shown 
in Proposition \ref{prop:fD-L}. We apply this operator on $u$ and add it with $2\la \xb,\nabla_\xb\ra \partial_{y_{d_2}}$ since,
by \eqref{eq:H-iden},
\begin{align*}
  2\la \xb,\nabla_\xb\ra \partial_{y_{d_2}} u \, & = 2 \partial_{y_{d_2}} g \cdot \la \xb,\nabla_\xb\ra H 
      + 2g \la \xb,\nabla_\xb\ra \partial_{y_{d_2}} H \\
& = 2m\, \partial_{y_{d_2}} g\cdot  H - 2 \partial_{y_{d_2}} g \la \yb,\nabla_\yb\ra  H + 2 g  \la \xb,\nabla_\xb\ra \partial_{y_{d_2}} H, 
\end{align*} 
which contains $2m \, \partial_{y_{d_2}} g = \fD_{(\bk, \alpha+2m)}^{\Sigma, (\yb)} g - \fD_{(\bk, \alpha)}^{\Sigma, (\yb)} g$, 
where we again included the superscript $(\yb)$ to indicate the variable the operator acts on, then a careful 
and lengthy computation leads to
\begin{align*}
  \fD_{(\bk,\alpha)}^{\Sigma, (\yb)} u \, & +  2\la \xb,\nabla_\xb\ra \partial_{y_{d_2}} u = 
     \fD_{(\bk,\alpha+2m)}^\Sigma g \cdot H  \\
   &  + g \left( \sum_{j=1}^{d_2} y_j \partial_{y_j}^2  - \la \yb, \nabla_{\yb}\ra  
      + 2 \la \xb, \nabla_\xb \ra \partial_{y_{d_2}}  + (\a + \k_{d_2}+1) \partial_{y_{d_2}}   \right) H 
\end{align*}
where we have used, for example, $|y| \partial_{y_{d_2}}^2 H - \la \yb, \nabla_{\yb} \ra \partial_{y_{d_2}} H = 0$. Consequently,
using $\fD_{(\bk,\alpha+2m)}^{\Sigma, (\yb)} g = -(n-m) g$ and the first identity in \eqref{eq:H-iden}, we further obtain
\begin{align} \label{eq:D+EH}
  \fD_{(\bk, \alpha)}^{\Sigma, (\yb)} u  \, & + 2 \la \xb,\nabla_\xb\ra \partial_{y_{d_2}} u = 
   -n u + \la \xb,\nabla_{\xb} \ra u + g \cdot \fE(H), 
\end{align}
where 
$$
  \fE(H) = \sum_{j=1}^{d_2} y_j \partial_{y_j}^2 H + 2 \la \xb, \nabla_\xb \ra \partial_{y_{d_2}} H + (\a +1) \partial_{y_{d_2}} H.
$$
Now, using the identities \eqref{eq:H-iden} and \eqref{eq:H-iden2}, we can replace all derivatives with
respect to $\yb$ in $|\yb| \fE(H)$ to those with respect to $\xb$, 
\begin{align*}
|\yb| \fE(H) & =  \sum_{j=1}^{d_2} y_j \left[(m-1) \partial_{y_j} - \la \xb, \nabla_{\xb} \ra \partial_{y_j} \right] H 
  + [2 \la \xb, \nabla_\xb\ra  + (\a +1) ] \la \yb, \nabla_{\yb} \ra H \\
& = (\a + m) \la \yb, \nabla_\yb\ra H  +  \la \xb, \nabla_{\xb} \ra  \la \yb, \nabla_{\yb} \ra H \\  
& = \left( \a + m + \la \xb, \nabla_{\xb}\ra  \right) (m- \la \xb, \nabla_{\xb} \ra)H   \\
& =  - \la \xb,\nabla_\xb\ra^2 H -  \a \la \xb,\nabla_\xb\ra H + m (m+ \a) H.
\end{align*}
Since $H(\xb,\yb) = |\yb|^m \bm B_{\bm k, m}^\mu \left(\frac{\bm x}{|\bm y|}\right)$, it follows from the spectral operator
of $\fD_\mu^\BB$ given in \eqref{eq:Diff-ball} and \eqref{eq:D-ball}, that $H$ satisfies 
$$
    |\yb|^2 \Delta_{\xb} H - \la \xb, \nabla_{\xb} \ra^2 H - (2\mu + d_1-1) \la \xb, \nabla_{\xb} \ra H = - m(m+2\mu + d_1 -1) H.
$$
Together, the last two identities lead to $\fE(H) = - |\yb| \Delta_{\xb} H$. Consequently, we obtain
\begin{align*}
  \fD_{(\kb, \alpha)}^{\Sigma, (\yb)} u  \,  + 2 \la \xb,\nabla_\xb\ra \partial_{y_{d_2}} u - \la \xb,\nabla_{\xb} \ra u + |\yb| \Delta_{\xb} u
  =  -n u. 
\end{align*}
This is \eqref{eq:B-Sig-eigen} and its left-hand side, using the expression of $\fD_{(\kb, \alpha)}^\Sigma$ in \eqref{eq:fDSig}
and rearranging, is \eqref{eq:B-Sig}. 
\end{proof}
 
 The operator $\fD_{\mu,\bk}^{\BB\rtimes \Sigma}$ is self-adjoint in 
 $L^2\big(\BB^{d_1} \rtimes \RR_\Sigma^{d_2}, \bm W_{\mu, \bk}^{\BB\rtimes \Sigma}\big)$ and it can also be written 
 in a more structural form that makes the self-adjointness transparent. 

\begin{thm} \label{thm:self-adjBS}
For $\mu > -\f12$ and $\bk \in \RR^{d_2-1}$ with $\k_i > -1$, $1 \le i \le d_2-1$,
\begin{align*} 
  \fD_{\mu,\bk}^{\BB\rtimes \Sigma} & = \frac{1}{|\yb|} \left[ \frac{1}{\Wb_{\mu,\bk}^{\BB\rtimes \Sigma} (\xb,\yb)}
    \sum_{i=1}^{d_1} \partial_{x_i} \left(\Wb_{\mu+1,\bk}^{\BB\rtimes \Sigma} (\xb,\yb) \partial_{x_i} \right)  
     + \sum_{1 \le i < j \le d_1} \big( D_{i,j}^{(\xb)} \big)^2 \right]\\
   & +  \frac{1}{\Wb_{\mu,\bk}^{\BB\rtimes \Sigma} (\xb,\yb)} \sum_{i=1}^{d_2-1} (\partial_{y_i} - \partial_{y_{d_2}})
       y_i  \Wb_{\mu,\bk}^{\BB\rtimes \Sigma} (\xb,\yb) (\partial_{y_i} - \partial_{y_{d_2}}) \notag \\
  & +  \frac{1}{|\yb|^{d_1+1} \Wb_{\mu,\bk}^{\BB\rtimes \Sigma} (\xb,\yb)} 
         \big(|\yb| \partial_{y_{d_2}} +\la \xb, \nabla_{\xb}\ra \big) |\yb|^{d_1}
              \Wb_{\mu,\bk}^{\BB\rtimes \Sigma} (\xb,\yb) \big(|\yb| \partial_{y_{d_2}} + \la \xb, \nabla_{\xb}\ra \big). \notag
\end{align*}
\end{thm}

\begin{proof}
Just like Theorem \ref{cor:diff-eqnV}, the main task is to identify the correct factorization. We record the intermediate step for
verification. The first one follows from \eqref{eq:Diff-ball2} by a simple change of variables,
\begin{align*}
  \frac{1}{\bm W_{\mu,\bk}^{\BB\rtimes\Sigma}(\bm x, \bm y)}   
&    \sum_{i=1}^{d_1} \partial_{x_i} \bm W_{\mu+1,\bk}^{\BB\rtimes \Sigma}(\bm x, \bm y) \partial_{x_i} 
   +  \sum_{1 \le i< j \le d_1} \left(D_{i,j}^{(\bm x)}\right)^2 \\
& =    |\bm y|^2 \Delta_{\bm x} -  \la \bm x, \nabla_{\bm x} \ra^2 - (2\mu + d_1 -1) \la \bm x, \nabla_{\bm x} \ra.
\end{align*}
The second one follows from a straightforward calculation 
\begin{align*}
 &  \frac{1}{\Wb_{\mu,\bk}^{\BB\rtimes \Sigma} (\xb,\yb)}  \sum_{i=1}^{d_2-1} (\partial_{y_i} - \partial_{y_d}) 
       y_i  \Wb_{\mu,\bk}^{\BB\rtimes \Sigma} (\xb,\yb)  (\partial_{y_i} - \partial_{y_d})   \\
  & \quad \sum_{i=1}^{d_2} \big( y_i \partial_{y_i}^2 + (\k_i+1-y_i) \partial_{y_i} \big) -( |\bk|+d_2) \partial_{y_{d_2}} +|\yb|\partial_{y_{d_2}}
      -2 \la \yb, \nabla_{\yb} \ra \partial_{y_{d_2}} + |\yb| \partial_{y_{d_2}}^2,
\end{align*}
where we assume $\k_d =0$ and we note that $y_i$ for $1 \le i \le d_2-1$ comes from the derivative of $\e^{-|\yb'|}$ in the 
weight function. Furthermore, using the observation that 
$$
 (|\yb| \partial_{y_{d_2}} + \la \xb,\nabla_\xb\ra ) (|\yb|^2 - \|\xb\|^2)^{\mu-\f12} = (2 \mu -1) (|\yb|^2 - \|\xb\|^2)^{\mu-\f12},
$$
we further deduce that 
\begin{align*}
  \frac{1}{|\yb|^{d_1+1} \Wb_{\mu,\bk}^{\BB\rtimes \Sigma} (\xb,\yb)}
   &  \big(|\yb| \partial_{y_{d_2}} +\la \xb, \nabla_{\xb}\ra \big) |\yb|^{d_1}
              \Wb_{\mu,\bk}^{\BB\rtimes \Sigma} (\xb,\yb) \big(|\yb| \partial_{y_{d_2}} + \la \xb, \nabla_{\xb}\ra \big) \\
     & = |\yb| \partial_{y_{d_2}}^2 + 2 \la \xb, \nabla_{\xb} \ra \partial_{y_{d_2}} - \la \xb, \nabla_{\xb} \ra 
       + (2\mu + d_1 -|\yb|) \partial_{y_{d_2}} \\
     & + |\yb|^{-1} \left(\la \bm x, \nabla_{\bm x} \ra^2+(2\mu + d_1 -1) \la \bm x, \nabla_{\bm x} \ra \right).
\end{align*}
Multiplying the first identity by $|\yb|^{-1}$, then adding the resulting identity with the second and the third identities, 
the right-hand side becomes $ \fD_{\mu,\bk}^{\BB\rtimes \Sigma}$ given in \eqref{eq:B-Sig}. 
\end{proof}

\begin{cor} \label{cor:B-Sig_surface}
For $\mu > -\f12$ and $\bk \in \RR^{d_2-1}$ with $\k_i > -1$, $1 \le i \le d_2-1$,
\begin{align*} 
   & - \int_{\BB^{d_1} \rtimes \Sigma^{d_2}}\fD_{\mu, \bk}^{\BB\rtimes\Sigma}
     f (\bm x,\bm y) \cdot g(\bm x, \bm y) \bm W_{\mu, \bk}(\bm x, \bm y)^{\BB\rtimes\Sigma} \d \bm x \d \bm y  \\
  &  = \int_{\BB^{d_1} \rtimes \Sigma^{d_2}}  \frac{1}{|\bm y|} \sum_{i=1}^{d_1} \partial_{x_i} f(\bm x,\bm y)\cdot \partial_{x_i} g(\bm x,\bm y) 
  \bm W_{\mu+1,\bk}^{\BB\rtimes \Sigma}(\bm x,\bm y)  \d \bm x \d \bm y  \notag \\ 
  & +  \int_{\BB^{d_1} \rtimes \Sigma^{d_2}}  \frac{1}{|\bm y|} \sum_{1 \le i< j \le d_1}  
       D_{i,j}^{(\bm x)} f(\bm x, \bm y) \cdot D_{i,j}^{(\bm x)} g(\bm x, \bm y) 
  \bm W_{\mu,\bk}^{\BB\rtimes \Sigma}(\bm x,\bm y)  \d \bm x \d \bm y  \notag \\ 
  & + \int_{\BB^{d_1} \rtimes \Sigma^{d_2}}   \sum_{ i=1}^{d_2-1} y_i (\partial_{y_i} - \partial_{y_{d_2}}) f(\bm x,\bm y) \cdot 
    (\partial_{y_i} - \partial_{y_{d_2}})g(\bm x,\bm y)
  \bm W_{\mu,\bk}^{\BB\rtimes \Sigma}(\bm x,\bm y) \d \bm x \d \bm y  \notag \\
  &  +  \int_{\BB^{d_1} \rtimes \Sigma^{d_2}} \frac{1}{|\bm y|} (|\yb| \partial_{d_2} +\la \xb,\nabla_\xb\ra) f(\bm x,\bm y) 
  \cdot (|\yb| \partial_{d_2} + \la \xb,\nabla_\xb\ra)g(\bm x,\bm y) 
  \bm W_{\mu,\bk}^{\BB\rtimes \Sigma}(\bm x,\bm y)  \d \bm x \d \bm y.  \notag
 \end{align*}
\end{cor} 

The proof follows from the expression in Theorem \ref{thm:self-adjBS} via integration by parts. Only the last term requires 
some work, which however can be easily taken care of as in the proof of Corollary \ref{cor:self-adjBT}. 

As in the case of $\BB^{d_1} \rtimes \triangle^{d_2}$, we could derive sharp Bernstein inequalities in the 
$L^2\big(\BB^{d_1}\rtimes\Sigma^{d_2},  \bm W_{\mu,\bk}^{\BB\rtimes \Sigma}\big)$ norm from the identity
in the above corollary. We leave the details to interested readers. 

Since the eigenvalue \eqref{eq:B-Sig-eigen} of $\fD_{\mu,\bk}^{\BB\rtimes \Sigma}$ is linear in $n$, we obtain two
more families of orthogonal polynomials that possess a spectral operator. Recall that $\fD^H$ denotes the spectral
operator \eqref{eq:diff-eqnH} for product Hermite polynomials and $\fD_\bg^L$ denotes the spectral operator  
\eqref{eq:diff_prod_L} for product Laguerre polynomials. 

\begin{thm} 
Let $d_1, d_2, d_3$ be positive integers. Let $\mu > -\f12$ and $\bk \in \RR^{d_2-1}$ with $\k_i > -1$, 
$1 \le i \le d_2-1$. Let also $\bg \in \RR^{d_3}$ with $\g_i > -1$, $1 \le i \le d_3$. Then orthogonal polynomials 
for the weight function 
$$
\Wb_{\mu,\bk,\bg}^{\BB\rtimes \Sigma \times\RR_+}(\xb,\yb,\bm z) = \Wb^{\BB\rtimes \Sigma}(\xb,\yb)\Wb_\bg^{L}(\bm z)
$$ 
on the product domain $(\BB^{d_1}\rtimes \RR_\Sigma^{d_2}) \times\RR_+^{d_3}$ satisfy 
$$
   \fD_{\mu,\kb}^{\BB\rtimes\Sigma, (\xb,\yb)} u  + \fD_\bg^{L, \bm z} u = -  n u,
    \quad \forall u \in \CV_n \big((\BB^{d_1}\rtimes \RR_\Sigma^{d_2}) \times\RR_+^{d_3}, \Wb_{\mu,\bk,\bg}^{\BB\rtimes \Sigma \times\RR_+}\big).
$$
Moreover, orthogonal polynomials for the weight function 
$$
\Wb_{\mu,\bk}^{\BB\rtimes \Sigma \times\RR}(\xb,\yb,\bm z) = \Wb^{\BB\rtimes \Sigma}(\xb,\yb)\Wb^H (\bm z)
$$
on the product domain $(\BB^{d_1}\rtimes \RR_\Sigma^{d_2}) \times\RR^{d_3}$ satisfy 
$$
  2 \fD_{\mu,\kb}^{\BB\rtimes\Sigma, (\xb,\yb)} u  + \fD^{H, \bm z} u = - 2 n u,
    \quad \forall u \in \CV_n\big((\BB^{d_1}\rtimes \RR_\Sigma^{d_2}) \times\RR^{d_3}, \Wb_{\mu,\bk}^{\BB\rtimes \Sigma \times\RR}\big).
$$
\end{thm}

This follows immediately using product orthogonal polynomials and the linearity of the eigenvalues.

\section{Spectral operator of fourth order} 
\setcounter{equation}{0}

The four families of wrapped product orthogonal polynomials in \eqref{eq:class2} are not eigenfunctions of 
a second-order differential operator, but they are eigenfunctions of a fourth-order differential operator as we
show in this section. Each of these families is a wrapped product, in which $\Omega_1^{d_1}$ is either
$\RR^{d_1}$ or $\RR_+^{d_1}$, so that $\xb /\rho(\yb) \in \Omega_1^{d_1}$ in \eqref{eq:wrapDomain} is
equivalent to $\xb \in \Omega_1^{d_1}$. Consequently, $\Omega_1^{d_1} \rtimes \Omega_2^{d_2}$ 
coincides with $\Omega_1^{d_1} \times \Omega_2^{d_2}$. In other words, the wrapped product domain
coincides with the tensor product domain. We concentrate on establishing the spectral operator of order four
in each case and will be brief on other aspects. 

\subsection{Hermite-Ball polynomials $\RR^{d_1}  \rtimes \BB^{d_2}$}
We consider the wrapped product  $\RR^{d_1}  \rtimes \BB^{d_2}$ with $\rho(y) = \sqrt{1-\|\yb\|^2}$, so that 
\begin{align*}
\RR^{d_1}  \rtimes \BB^{d_2}\, & = \left\{(\xb,\yb): \frac{\xb}{\sqrt{1-\|\yb\|^2}} \in  \RR^{d_1},  \, \yb \in  \BB^{d_2}\right\} 
 = \RR^{d_1} \times \BB^{d_2}
\end{align*}
and the weight function $\Wb_{\mu}^{\RR \rtimes \BB}$ for $\mu > -\f12$, defined on $\RR^{d_1} \times \BB^{d_2}$ by  
\begin{align} \label{eq:W-RB}
\Wb_{\mu}^{\RR \rtimes \BB}(\xb,\yb) \, & = (1-\|\yb\|^2)^{\mu-\f12} \e^{ - \frac{\|x\|^2}{1-\|\yb\|^2}}
   =  \Wb^H \bigg( \frac{x}{\sqrt{1-\|\yb\|^2}} \bigg)\Wb_{\mu}^\BB(\yb).
\end{align}
In this setting, the wrapped orthogonal polynomials \eqref{eq:wrapOP} are given in terms of classical orthogonal 
polynomials on the unit ball and product Hermite polynomials, 
 \begin{equation} \label{eq:OP_RB}
 \bm Q_{\bm j, \bm k, m}^n(\bm x,\bm y) = \BB_{\bm j, n-m}^{\mu+m+\frac{d_1}{2}}(\bm y) (1-\|\bm y\|^2)^{\frac{m}{2}}  
      \bm H_{\bm k, m}  \bigg(\frac{\bm x}{\sqrt{1-\|\bm y\|^2}} \bigg),
\end{equation}
where $|\bm k| = m$ for $\bm k \in \NN_0^{d_1}$, $|\bm j| = n-m$ for $\bm j \in \NN_0^{d_2}$, 
and $0 \le m  \le n$. 

We show that these polynomials are eigenfunctions of a fourth-order differential operator. 


\begin{thm}\label{thm:RB}
For $\mu > -\f12$, define the fourth-order differential operator 
\begin{align} \label{eq:RB}
  \fD_{\mu}^{\RR \rtimes \BB} u: = \,& -\frac14 (1-\|\yb\|^2) \Delta_{\xb}^2 + \la \xb,\nabla_\xb\ra \Delta_\xb - (\la \xb,\nabla_\xb\ra 
   + \la \yb,\nabla_\yb\ra) ^2  \\
   & + \Delta_\xb + \Delta_\yb 
   +  (2\mu + d_1+d_2 -1) \left( \f12 \Delta_\xb- \la \xb,\nabla_\xb\ra - \la \yb,\nabla_\yb\ra \right). \notag
\end{align}
Then $\CV_n(\RR^{d_1}  \rtimes \BB^{d_2}, \bm W_{\mu}^{\RR \rtimes \BB})$ is an eigenspace 
of $ \fD_{\mu}^{\RR \rtimes \BB}$ for each $n \in \NN_0$. More precisely, 
\begin{align} \label{eq:RB-eigen}
  \fD_{\mu}^{\RR \rtimes \BB} u =  - n (n+ 2\mu + d_1+d_2 - 1) u, \qquad 
         \forall u \in \CV_n\big(\RR^{d_1}  \rtimes \BB^{d_2}, \bm W_{\mu}^{\RR \rtimes \BB}\big).
\end{align}
\end{thm}

\begin{proof}
We work with $u =\bm Q_{\bm j, \bm k, m}^n$ given in \eqref{eq:OP_RB} and write 
$u(\bm x,\bm y) = g(\bm y) H(\bm x,\bm y)$ with 
$$
g(\bm y) = \BB_{\bm j, n-m}^{\a+m}(\bm y) \quad \hbox{and} \quad
  H(\bm x, \bm y) =  (1-\|\bm y\|^2)^{\frac{m}{2}}  
      \bm H_{\bm k, m} \bigg(\frac{\bm x}{\sqrt{1-\|\bm y\|^2}}\bigg).
$$
where $\a = \mu+ \frac{d_1}2$.
With $\rho(\yb) = \sqrt{1-\|\bm y\|^2}$, it follows from \eqref{eq:diffHa} and \eqref{eq:diffHb} that
\begin{align} \label{eq:RB-diffH}
 \begin{split}
    (1-\|\yb\|^2) \partial_{y_i} H \, & = y_i (\la \xb, \nabla_\xb \ra-m) H, \quad 1 \le i \le d_2, \\
    (1-\|\yb\|^2) \la \yb, \nabla_\yb\ra H \,& = \|\yb\|^2 (\la \xb, \nabla_\xb \ra -m) H.
\end{split}
\end{align}
As a consequence of the two identities in \eqref{eq:RB-diffH}, we obtain readily that 
\begin{align*}
   (1-\|\yb\|^2)  \la \nabla_\yb g, \nabla_\yb H \ra  = \sum_{i=1}^{d_2} \partial_{y_i} g \cdot y_i  \,  (\la \xb, \nabla_\xb \ra  -m) H
    =  \, & \la \yb, \nabla_\yb \ra g (\la \xb, \nabla_\xb \ra -m) H, \\
(1-\|\yb\|^2) \la \yb, \nabla_\yb\ra g \cdot \la \yb \nabla_\yb \ra H 
     =  \| \yb\|^2 \, & \la \yb, \nabla_\yb\ra g  (\la \xb, \nabla_\xb \ra -m) H,
\end{align*}
so that their difference implies the identity
\begin{align} \label{eq:RB-diffH2}
 \la \nabla_\yb g, \nabla_\yb H \ra -\la \yb, \nabla_\yb\ra g \cdot \la \yb \nabla_\yb \ra H 
      =  \la \yb, \nabla_\yb \ra g  \big (\la \xb, \nabla_\xb \ra -m \big) H.  
\end{align}
Let $\fD_\a^{\BB, (\yb)}$ be the spectral operator \eqref{eq:Diff-ball} for the classical orthogonal polynomials on 
the unit ball in the variable $\yb$. Then, $\fD_{\a+ m}^{\BB, (\yb)} g - \fD_\a^{\BB, (\yb)} g= - 2m  \la \yb, \nabla_\yb \ra g$,
so that a straightforward computation shows that the identity \eqref{eq:RB-diffH2} implies 
\begin{align*}
\fD_\a^{\BB, (\yb)} u -2 \la \yb,\nabla_{\yb} \ra \la \xb,\nabla_{\xb} \ra u 
     = \,& \fD_\a^{\BB, (\yb)} g \cdot H + 2  \la \nabla_\yb g, \nabla_\yb H \ra 
         - 2 \la \yb, \nabla_\yb\ra g \cdot \la \yb \nabla_\yb \ra H \\
    & -2 \la \yb,\nabla_{\yb} \ra g \la \xb,\nabla_{\xb} \ra H + g \cdot \fE(H) \\
    = \, & \fD_{\a+ m}^{\BB, (\yb)}g\cdot H  + g \cdot \fE(H) ,
\end{align*}
where the operator $\fE$ is given by 
$$
  \fE = \Delta_\yb - \la \yb,\nabla_{\yb} \ra^2 -2 \la \yb,\nabla_{\yb} \ra \la \xb,\nabla_{\xb} \ra
  - (2\a + d_2 -1)  \la \yb,\nabla_{\yb} \ra.  
$$
Using $\fD_{\a+ m}^{\BB, (\yb)}g = - (n-m) (n+m + 2\a + d_2 -1) g$, 
we then obtain
\begin{align}\label{eq:Dmu_RB}
\fD_\a^{\BB, (\yb)} u -2 \la \yb,\nabla_{\yb} \ra \la \xb,\nabla_{\xb} \ra u 
     =\,&   -n(n+2\a+d_2 -1) u  \\
     &  + m (m+2\a+d_2 -1) u + g \cdot \fE(H), \notag
\end{align}

We now use \eqref{eq:RB-diffH2} to replace the derivatives with respect to $\yb$ in $\fE(H)$ by derivatives in
$\xb$. Applying the second identity in \eqref{eq:RB-diffH} twice, a careful computation shows that 
\begin{align*}
 (1-\|\yb\|^2) \big(\Delta_\yb -   \la \yb,\nabla_{\yb} \ra^2\big)H =\, &
       (1-\|\yb\|^2) \la \yb,\nabla_{\yb} \ra ( \la \xb,\nabla_{\xb} \ra - m )H\\
  &   + 2 (1-\|\yb\|^2)  \la \yb,\nabla_{\yb} \ra H 
     + (d_2 - \|\yb\|^2) ( \la \xb,\nabla_{\xb} \ra -m )H \\
   = &\|\yb\|^2 
       ( \la \xb,\nabla_{\xb} \ra -m )^2 H 
     +  d_2 ( \la \xb,\nabla_{\xb} \ra -m )H.
\end{align*}
Substituting this identity into $(1-\|\yb\|^2) \fE(H)$ leads to, after further computation, that 
\begin{align} \label{eq:EH-R}
 (1-\|\yb\|^2) \fE(H) = \, & \|\yb\|^2 \left(-\la \xb,\nabla_{\xb} \ra^2 - (2\a -1) \la \xb,\nabla_{\xb} \ra+ m(m+2 \a -1) \right)H \\
     & + d_2 (1-\|\yb\|^2) (\la \xb, \nabla \xb\ra -m)H. \notag
\end{align}
By the definition of $H$ and the spectral equation \eqref{eq:diff-eqnH}, we obtain
\begin{equation} \label{eq:diffH-RB}
    (1-\|\yb\|^2) \Delta_\xb H - 2 \la \xb, \nabla_\xb \ra H = - 2 m H,
\end{equation}
which can be rewritten as 
$$
   m H = \la \xb, \nabla_\xb \ra H - \frac12 (1-\|\yb\|^2) \Delta_{\xb} H.
$$
Using this relation twice, we deduce 
\begin{align} \label{eq:mm-H}
   & m(m+ 2 \a -1) H  = \la \xb, \nabla_\xb \ra^2 + (2\a-1)  \la \xb, \nabla_\xb \ra \\
    &\,  - \frac12 (1-\|\yb\|^2) \left[  \la \xb, \nabla_\xb \ra \Delta_\xb + \Delta_\xb \la \xb, \nabla_\xb\ra 
      - \frac 12(1-\|\yb\|^2) \Delta_\xb^2 + (2\a-1) \Delta_\xb \right], \notag
\end{align} 
which implies an expression of $\fE(H)$ given by 
\begin{align*}
\fE(H) = - \frac{\|\yb\|^2}2  \left[  \la \xb, \nabla_\xb\ra \Delta_\xb + 
    \Delta_\xb  \la \xb, \nabla_\xb\ra - \frac 12(1-\|\yb\|^2) \Delta_\xb ^2 
       + (2\a-1) \Delta_\xb  \right] H  \\
        - \f{d_2}{2} (1-\|\yb\|^2) \Delta_x  H.
\end{align*}
Using this identity and \eqref{eq:mm-H} with $\a$ replaced by $\a+ \frac{d_2}2$, we obtain from \eqref{eq:Dmu_RB} that 
\begin{align*}
& \fD_\a^{\BB, (\yb)} u  -2 \la \yb,\nabla_{\yb} \ra \la \xb,\nabla_{\xb} \ra u - \la \xb,\nabla_\xb \ra^2  u -(2\a-1) \la \xb,\nabla_\xb \ra u\\
&\qquad +  \frac{1}2  \left[  \la \xb, \nabla_\xb\ra \Delta_\xb  + \Delta_\xb  \la \xb, \nabla_\xb\ra - \frac 12(1-\|\yb\|^2) \Delta_\xb ^2 
       + (2\a-1) \Delta_\xb  \right] u \\
      &\quad  =   -n(n+2\a+d_2 -1) u= -n (n+ 2\mu + d_1+d_2-1) u,
\end{align*}
after simplification via $\|\yb\|^2 + (1-\|\yb\|^2) =1$. This is \eqref{eq:RB}, and its left-hand side can be 
simplified using the expression of $\fD_\a^{\BB, (\yb)}$ in \eqref{eq:Diff-ball}, as well as $\a = \mu + \frac{d_1}2$ and 
$\Delta_\xb \la \xb, \nabla_\xb\ra = \la \xb, \nabla_\xb\ra \Delta_\xb + 2 \Delta_\xb$, 
to the expression $\fD_\mu^{\RR \rtimes \BB}$ in \eqref{eq:RB}. 
\end{proof} 
  
This is new even in the case of $d_1 = 1$ and $d_2 =1$, which correspond to orthogonal polynomials of 
two variables on the domain $\RR \times [-1,1]$ and the weight function 
$$
   \Wb_\mu^{\RR\times[-1,1]}(x,y) = (1-y^2)^{\mu-\f12} \e^{ - \frac{x^2}{1-y^2}}.
$$
The orthogonal basis \eqref{eq:OP_RB} for $\CV_n\left(\RR\times [-1,1], \Wb_\mu^{\RR\times[-1,1]}\right)$ is given by
$$ 
  \Qb_m^n(x,y) = C_{n-m}^{\mu + m+ \f12} (y) (1-y^2)^{\frac{m}{2}} H_m \bigg( \frac{x}{\sqrt{1-y^2}} \bigg), \quad 0 \le m \le n.
$$
The spectral operator \eqref{eq:RB} is simplified to the following form
$$
  \fD_{\mu}^{\RR \rtimes [-1,1]} = - \frac14 (1-y^2) \partial_x^4 + x \partial_x^3 - (x\partial_x + y \partial_y)^2
     + \bigg(\mu + \frac32\bigg) \partial_x^2 +\partial_y^2 - (2\mu+1) \left(x \partial_x+y \partial_y \right)
$$
and the spectral equation becomes
$$
  \fD_{\mu}^{\RR \rtimes [-1,1]} u = - n (n+2\mu+1) u, \qquad \forall u \in \CV_n\left(\RR\times [-1,1], \Wb_\mu^{\RR\times[-1,1]}\right). 
$$

\subsection{Hermite-Simplex polynomials $\RR^{d_1}  \rtimes \triangle^{d_2}$}
We consider the wrapped product  $\RR^{d_1}  \rtimes \triangle^{d_2}$ with $\rho(y) = 1-|\yb|$, so that 
\begin{align*}
\RR^{d_1}  \rtimes \triangle^{d_2}\, & = \left\{(\xb,\yb): \frac{\xb}{1-|\yb|} \in  \RR^{d_1},  \, \yb \in \triangle^{d_2}\right\} 
 = \RR^{d_1} \times \triangle^{d_2}
\end{align*}
and the weight function $\Wb_{\kb}^{\RR \rtimes \triangle}$ for $\bk \in \RR^{d_2+1}$ with $\k_i > -1$ for $1 \le i \le d_2+1$, 
defined on $\RR \times \triangle^{d_2}$ by  
\begin{align} \label{eq:W-RT}
\Wb_{\bk}^{\RR \rtimes \triangle}(\xb,\yb) \, & =  \prod_{i=1}^d y_i^{\k_i} (1-|\yb|)^{\k_{d_2+1}} \e^{ - \frac{\|x\|^2}{(1-|\yb|)^2}}
   =\Wb^H \bigg(\frac{x}{1-|\yb|} \bigg) \Wb_{\bk}^\triangle(\yb).
\end{align}
In this setting, the wrapped orthogonal polynomials \eqref{eq:wrapOP} are given in terms of classical orthogonal 
polynomials on the simplex and product Hermite polynomials, 
 \begin{equation} \label{eq:OP_RT}
 \bm Q_{\bm j, \bm k, m}^n(\bm x,\bm y) = \Tb_{\bm j, n-m}^{\bk+(2m+d_1)\bep}(\bm y) (1-|\bm y|)^m   
      \bm H_{\bm k, m}  \bigg(\frac{\bm x}{1-|\bm y|} \bigg),
\end{equation}
where $\bep =(0,\ldots, 0, 1) \in \RR^{d_2+1}$, $|\bm k| = m$ for $\bm k \in \NN_0^{d_1}$, $|\bm j| = n-m$ for $\bm j \in \NN_0^{d_2}$, 
and $0 \le m  \le n$. 

\begin{thm}\label{thm:RT}
For $\bk \in \RR^{d_2+1}$ with $\k_i > -1$, define the fourth-order differential operator 
\begin{align} \label{eq:RT}
  \fD_{\bk}^{\RR \rtimes \triangle} u: = \,& -\frac14 (1-|\yb|)^3 \Delta_{\xb}^2 + (1-|\yb|) \la \xb,\nabla_\xb\ra \Delta_\xb 
   - (\la \xb,\nabla_\xb\ra + \la \yb,\nabla_\yb\ra) ^2  \\
   & +  \left(1+ \frac{1}{2} (\k_{d_2+1} + d_1)\right) (1-|\yb|)\Delta_\xb + 
     \sum_{i=1}^{d_2} \left[ y_i \partial_{y_i}^2 + (\k_i+1) \partial_{y_i} \right] \notag \\
    & - (|\kb|+d_1+d_2) \big( \la \xb,\nabla_\xb\ra+ \la \yb,\nabla_\yb\ra\big). \notag
\end{align}
Then $\CV_n(\RR^{d_1}  \rtimes \triangle^{d_2}, \bm W_{\bk}^{\RR \rtimes \triangle})$ is an eigenspace 
of $ \fD_{\bk}^{\RR \rtimes \triangle}$ for each $n \in \NN_0$. More precisely, 
\begin{align} \label{eq:RT-eigen}
  \fD_{\bk}^{\RR \rtimes \triangle} u =  - n (n+ |\bk|+  d_1 + d_2) u, \qquad 
         \forall u \in \CV_n\big(\RR^{d_1}  \rtimes \triangle^{d_2}, \bm W_{\mu}^{\RR \rtimes \triangle}\big).
\end{align}
\end{thm}

\begin{proof}
Let $\b = \k_{d_2  +1} + d_1$. We write $u =  \bm Q_{\bm j, \bm k, m}^n(\bm x,\bm y)$ as $u = g H$ with 
$$
  g(\yb) =  \Tb_{\bm j, n-m}^{\bk+(2m+d_1)\bep}(\bm y)  \quad \hbox{and}\quad  
     H(\xb,\yb) = (1-|\bm y|)^m \bm H_{\bm k, m}  \bigg(\frac{\bm x}{1-|\bm y|} \bigg).
$$
Following the proof in Theorem \ref{thm:B-Tri} and setting $\bk = (\bk', \k_{d_2 +1})$ with $\bk' \in \RR^{d_2}$, 
the identity \eqref{eq:diff-gH} gives
\begin{align} \label{eq:diff-gH-RT}
\fD_{(\bk',\b)}^{\triangle, (\yb)} u & -  2 \la \bm x, \nabla_{\bm x} \ra  \la \bm y, \nabla_{\bm y} \ra u 
- (|\bk'|+ d_2)\la \bm x, \nabla_{\bm x} \ra u \\
& =  - (n-m) (n+ m + |\bk'|+ \b +d_2) u - m (|\bk'|+ d_2)u  - g \fE(H)   \notag \\
& = - n(n + |\bk'| + \b +d_2) u + m (m+ \b) u - g \fE(H)  \notag \notag
\end{align}
where $\fE(H)$ is given by 
\begin{align*}
  \fE(H) =  (\la \bm x, \nabla_{\bm x} \ra +m + \b) \la \bm y, \nabla_{\bm y} \ra H.
\end{align*}
By \eqref{eq:diffH}, we obtain
$$
   (1-|\yb|) \fE(H) = |\yb| \left(\la \xb,\nabla)\xb\ra^2 + \b \la \xb,\nabla)\xb\ra - m(m+\b) \right)H. 
$$
We now follow the proof of Theorem \ref{thm:RB} based on an analog of \eqref{eq:diffH-RB}, which holds
with $\sqrt{1-\|\yb\|^2}$ replaced by $1-|\yb|$ for our $H$. In particular, the analogue of \eqref{eq:mm-H} 
with $\b = 2 \a -1$ becomes 
\begin{align*} 
   & m(m+ \b) H  = \la \xb, \nabla_\xb \ra^2 + \b \la \xb, \nabla_\xb \ra \\
    &\,  - \frac12 (1-|\yb|)^2 \left[  \la \xb, \nabla_\xb \ra \Delta_\xb + \Delta_\xb \la \xb, \nabla_\xb\ra 
      - \frac 12(1-|\yb|)^2 \Delta_\xb^2 + \b \Delta_\xb \right], \notag
\end{align*} 
which implies an analog of \eqref{eq:mm-H} given by 
$$
\fE(H) = - \frac12 |\yb|(1-|\yb|) \left[  \la \xb, \nabla_\xb\ra \Delta_\xb + \Delta_\xb  \la \xb, \nabla_\xb\ra
   - \frac 12(1-|\yb|)^2 \Delta_\xb ^2 + \b \Delta_\xb  \right] H.
$$
Hence, similarly to the proof of Theorem \ref{thm:RB}, the above two identities imply that
\begin{align*}
\fD_{(\bk',\b)}^{\triangle, (\yb)} u & -  2 \la \bm x, \nabla_{\bm x} \ra  \la \bm y, \nabla_{\bm y} \ra u 
- (|\bk'|+ d_2)\la \bm x, \nabla_{\bm x} \ra u - \la \xb, \nabla_\xb \ra^2 u - \b \la \xb, \nabla_\xb \ra  u\\
& \quad+ \frac12 (1-|\yb|) \left[  \la \xb, \nabla_\xb\ra \Delta_\xb + \Delta_\xb  \la \xb, \nabla_\xb\ra
   - \frac 12(1-|\yb|)^2 \Delta_\xb ^2 + \b \Delta_\xb  \right] u\\
   & = - n(n + |\bk'| + \b +d_2) u = -n (n+|\bk| + d_1+d_2) u.
\end{align*}
Using the formula of $\fD_{(\bk',\b)}^{\triangle, (\yb)}$ given in \eqref{eq:Dsimplex2}, the
left-hand side of the formula simplifies to $\fD_\mu^{\RR \rtimes \triangle}$ in \eqref{eq:RT}. This
completes the proof.  
\end{proof}

\subsection{Hermite-Laguerre polynomials $\RR^{d_1}  \rtimes \RR_\Sigma^{d_2}$}
We consider the wrapped product  $\RR^{d_1}  \rtimes \RR_\Sigma^{d_2}$ with $\rho(y) = |\yb|$, so that 
\begin{align*}
\RR^{d_1}  \rtimes \RR_\Sigma^{d_2}\, & = \left\{(\xb,\yb): \frac{\xb}{|\yb|} \in  \RR^{d_1},  \, \yb \in \RR_\Sigma^{d_2}\right\} 
 = \RR^{d_1} \times \RR_\Sigma^{d_2}
\end{align*}
and the weight function $\Wb_{\kb}^{\RR \rtimes \Sigma}$ for $\bk \in \RR^{d_2}$ with $\k_i > -1$ for $1 \le i \le d$, 
defined on $\RR^{d_1} \times \RR_\Sigma^{d_2}$ by  
\begin{align} \label{eq:W-RSig}
\Wb_{\bk}^{\RR \rtimes \Sigma}(\xb,\yb) \, & =  \prod_{i=1}^{d_2-1} y_i^{\k_i} |\yb|^{\k_{d_2}} \e^{ -|\yb'|-|\yb|}
   \e^{ - \frac{\|x\|^2}{|\yb|^2}}
   = \Wb^H \bigg( \frac{x}{|\yb|} \bigg) \Wb_{\bk}^\Sigma(\yb).
\end{align}
The wrapped orthogonal polynomials \eqref{eq:wrapOP} are given in terms of alternative Laguerre polynomials
on $\RR^{d_2}_\Sigma$ and product Hermite polynomials, 
 \begin{equation} \label{eq:OP_RSig}
 \bm Q_{\bm j, \bm k, m}^n(\bm x,\bm y) = \hat \Lb_{\bm j, n-m}^{\bk+(2m+d_1) \bep}(\bm y) |\bm y|^m   
      \bm H_{\bm k, m}  \bigg(\frac{\bm x}{|\bm y|} \bigg),
\end{equation}
where $\bep =(0,\ldots, 0, 1) \in \RR^{d_2}$, $|\bm k| = m$ for $\bm k \in \NN_0^{d_1}$, $|\bm j| = n-m$ for $\bm j \in \NN_0^{d_2}$, 
and $0 \le m  \le n$. 

\begin{thm}\label{thm:RSig}
For $\k \in \RR^{d_2}$ with $k_i > -1$, $1 \le i \le d_2$, define the fourth-order differential operator 
\begin{align} \label{eq:RSig}
  \fD_{\bk}^{\RR \rtimes \Sigma} u: = \,& -\frac14 |\yb|^3 \Delta_{\xb}^2 + |\yb| \la \xb,\nabla_\xb\ra \Delta_\xb 
   + \left(1+ \frac{1}{2} (\k_{d_2} + d_1)\right) |\yb|  \Delta_\xb  \\
     &+  2  \la \xb,\nabla_\xb\ra \partial_{y_{d_2}}
   - 2  \la \yb,\nabla_\yb\ra \partial_{y_{d_2}} + \sum_{i=1}^{d_2} \left[ y_i \partial_{y_i}^2 + (\k_i+1-y_i) \partial_{y_i} \right]
     \notag \\
    & + 2 |\yb| \partial_{y_d}^2 - \la \xb, \nabla_\xb \ra - (|\bk'|- d_1+d_2-1) \partial_{y_{d_2}}. \notag
\end{align}
Then $\CV_n(\RR^{d_1}  \rtimes \RR_\Sigma^{d_2}, \bm W_{\bk}^{\RR \rtimes \Sigma})$ is an eigenspace 
of $ \fD_{\bk}^{\RR \rtimes \Sigma}$ for each $n \in \NN_0$. More precisely, 
\begin{align} \label{eq:RSig-eigen}
  \fD_{\bk}^{\RR \rtimes \Sigma} u =  - n u, \qquad 
         \forall u \in \CV_n\big(\RR^{d_1}  \rtimes \RR_\Sigma^{d_2}, \bm W_{\bk}^{\RR \rtimes \Sigma}\big).
\end{align}
\end{thm}

\begin{proof}
We only need to consider the orthogonal basis \eqref{eq:OP_RSig}, so we assume $u  = g H$ with
$$
  g(\yb) = \hat \Lb_{\bm j, n-m}^{\bk+(2m+d_1) \bep}(\bm y) \quad \hbox{and} \quad H(\xb,\yb) = |\bm y|^m   
      \bm H_{\bm k, m} \bigg(\frac{\bm x}{|\bm y|} \bigg).
$$
Following the proof of Theorem \ref{thm:B-Sig}, see \eqref{eq:D+EH} and the paragraph below, and
setting $\bk+d_1 \bep = (\bk', \a)$ with $\a = \k_{d_2} + d_1$, we have
\begin{align}\label{eq:RSig_1}
  \fD_{\bk+ d_1\bep}^{\Sigma, (\yb)} u  \, & + 2 \la \xb,\nabla_\xb\ra \partial_{y_{d_2}} u = 
   -n u + \la \xb,\nabla_{\xb} \ra u + g \cdot \fE(H), 
\end{align}
where the multiple $|\yb| \fE(H)$ is given by 
\begin{align*}
|\yb| \fE(h) =  - \la \xb,\nabla_\xb\ra^2 H -  \a \la \xb,\nabla_\xb\ra H + m (m+ \a) H.
\end{align*}
By the definition of $H$ and the spectral equation \eqref{eq:diff-eqnH}, we obtain
\begin{equation*} 
    |\yb|^2 \Delta_\xb H - 2 \la \xb, \nabla_\xb \ra H = - 2 m H,
\end{equation*}
which leads to, as an analog of \eqref{eq:mm-H}, 
\begin{align*}
    m(m+ \a) H  = \,&  \la \xb, \nabla_\xb \ra^2 + \a \la \xb, \nabla_\xb \ra 
    - \frac12 |\yb|^2 \left[ 2 \la \xb, \nabla_\xb \ra \Delta_\xb +2 \Delta_\xb - \frac 12 |\yb|^2 \Delta_\xb^2 + \a \Delta_\xb \right], \notag
\end{align*} 
and, as a result, an expression for $\fE(H)$, 
$$
   \fE(H) =  \frac{|\yb|^3}{4} \Delta_\xb^2 H -| \yb|  \la \xb, \nabla_\xb \ra \Delta_\xb H - \left(1+ \frac{\a}{2}\right) |\yb| \Delta_\xb H. 
$$
Substituting this into \eqref{eq:RSig_1} and rearranging terms gives \eqref{eq:RSig-eigen}, whereas the identity
\eqref{eq:RSig} is derived from \eqref{eq:fDSig}. 
\end{proof}

\subsection{Laguerre-Simplex polynomials $\RR_+^{d_1}  \rtimes \triangle^{d_2}$}
We consider the wrapped product  $\RR_+^{d_1}  \rtimes \triangle^{d_2}$ with $\rho(y) = 1-|\yb|$, 
\begin{align*}
\RR_+^{d_1} \rtimes \triangle^{d_2} \, & = \left\{(\xb,\yb): \frac{\xb}{1-|\yb|} \in  \RR_+^{d_1},  \, \yb \in  \triangle^{d_2}\right\} 
=\RR_+^{d_1} \times \triangle^{d_2}
\end{align*}
and the weight function $\Wb_{\bg, \bk}^{L \rtimes \triangle}$ for $\bg \in \RR^{d_1}$ with $\g_i > -1$, $1 \le i \le d_1$, and 
$\bk \in \RR^{d_2+1}$ with $\k_i > -1$, $1\le i \le d_2+1$, defined on $\RR_+^{d_1}  \rtimes \triangle^{d_2}$ by  
\begin{align} \label{eq:W-TSig}
\Wb_{\bg, \bk}^{L \rtimes \triangle}(\xb,\yb) \, & =\prod_{i=1}^{d_2} y_i^{\k_i}
     (1 - |\bm y | )^{\k_{d_2}+1}  \prod_{i=1}^{d_1} x_i^{\g_i}   \e^{- \frac{|\xb|}{1-|\yb|}}\\
  & =  \Wb_{\bg}^L \left( \frac{x}{1- |\bm y|} \right)\Wb_{\bk + |\bg| \bep}^{\triangle}(\yb), \notag
\end{align}
where $\bep = (0,\ldots,0,1) \in \RR^{d_2 +1}$. The wrapped orthogonal polynomials \eqref{eq:wrapOP} are given
in terms of the Laguerre polynomials and the Jacobi polynomials on the simplex, 
 \begin{equation} \label{eq:OP_TSig}
 \bm Q_{\bm j, \bm k, m}^n(\bm x,\bm y) = \Tb_{\bm j, n-m}^{\bk+(\b+2m)\bep}(\bm y) (1-|\bm y|)^m 
      \bm L_{\bm k, m}^{\bg} \left(\frac{\bm x}{1- |\bm y|}\right),
\end{equation}
where $\b = |\bg|+ d_1$, $|\bm k| = m$ for $\bm k \in \NN_0^{d_1}$, $|\bm j| = n-m$ for $\bm j \in \NN_0^{d_2}$, 
and $0 \le m  \le n$. 

It turns out, however, that the differential operator we found as the spectral operator is not of second order but of
fourth order. This is somewhat surprising and has not been seen before in the literature. To simplify the notation,
we introduce the differential operator 
$$
  \fL_\bk^{d, (\xb)} =  \sum_{i=1}^d \left[ x_i \partial_{x_i}^2 + (\k_i+1) \partial_{x_i} \right]
$$

\begin{thm}\label{thm:LSimplex}
For $\bg \in \RR^{d_1}$ with $\g_i  > -1$, $1\le i \le d_1$ and $\bk \in \RR^{d_2+1}$ with $\k_i > -1$, $1 \le i \le d_2+1$, 
define
\begin{align} \label{eq:LSimplex}
\fD_{\bg,\bk}^{L \rtimes \triangle}  : =\,& -(1-|\yb|) \left(\fL_\bg^{d_1, (\xb)}\right)^2 + \fL_\bg^{d_1, (\xb)} \la \xb, \nabla_\xb\ra+\la \xb, \nabla_\xb\ra \fL_\bg^{d_1, (\xb)} \\
 & +(\k_{d_2+1} + |\bg| +d_1) \fL_\bg^{d_1, (\xb)} +  \fL_\bk^{d_2, (\yb)} - (\la \xb, \nabla_{\xb}\ra +\la \bm y, \nabla_{\bm y}\ra)^2 
 \notag \\
 &     - (|\bk|+|\bg|+d_1 + d_2) ( \la \bm x, \nabla_{\bm x} \ra +\la \bm y, \nabla_{\bm y} \ra)  \notag 
 \end{align}
Then $\CV_n(\RR_+^{d_1} \rtimes \triangle^{d_2}, \bm W_{\bg, \bk}^{L \rtimes \triangle})$ is an eigenspace 
of $ \fD_{\bg,\bk}^{L \rtimes \triangle}$ for each $n \in \NN_0$. More precisely, 
\begin{align} \label{eq:LSimplex-eigen}
  \fD_{\mu,\bk}^{L \rtimes \triangle} u =  - n (n + |\bk|+ |\bg|+d_1+d_2) u, \qquad \forall 
      u \in \CV_n\big(\RR_+^{d_1}\rtimes \triangle^{d_2}, \bm W_{\bg, \bk}^{L \rtimes \triangle}\big).
\end{align}
\end{thm}

\begin{proof}
Like in the previous proofs, we set $u  = g H$ with
$$
  g(\yb) = \Tb_{\bm j, n-m}^{\bk+(\b+2m)\bep}(\bm y) \quad \hbox{and} \quad H(\xb,\yb) =(1- |\bm y|)^m 
      \bm L_{\bm k, m}^{\bg} \left(\frac{\bm x}{1- |\bm y|}\right).
$$
By \eqref{eq:diff-gH} with $\bk+\b \bep = (\bk', \a)$ for $\a = \k_{d_2+1} + \b$ and $\bk = (\bk', \k_{d_2+1})$ we obtain 
\begin{align} \label{eq:diff-gH_LT}
& \fD_{\bk+\b \bep}^{\triangle, (\yb)} u  -  2 \la \bm x, \nabla_{\bm x} \ra  \la \bm y, \nabla_{\bm y} \ra u 
 =  - (n-m) (n+ m + |\bk|+ \b +d_2) u \\
& \qquad \qquad\qquad\qquad\qquad\qquad\quad + (|\bk'|+ d_2)\left (\la \bm x, \nabla_{\bm x} \ra -m\right )u  - g  \fE(H), \notag \\
& =  - n (n + |\bk|+ \b +d_2) u + (|\bk'|+ d_2) \la \bm x, \nabla_{\bm x} \ra u + m(m+\a)u  - g  \fE(H), \notag 
\end{align}
where $\fE(H) = (\la \bm x, \nabla_{\bm x} \ra +m + \a) \la \bm y, \nabla_{\bm y} \ra H$, so that by \eqref{eq:diffH},
$$
  (1-|\yb|)\fE(H) = |\yb| (\la \bm x, \nabla_{\bm x} \ra +m + \a) (\la \xb, \nabla_{\bm x} \ra-m)  H.
$$
Since $H$ is a multiple of Laguerre polynomials, by \eqref{eq:diff_prod_L}, $H$ satisfies 
\begin{equation}\label{eq:eigen_wrapLT}
   (1-|\yb|) \fL_\bg^{d_1, (\xb)} H - \la \bm x, \nabla_{\bm x} \ra H = - m H,
\end{equation}
which implies immediately that $\fE(H)$ satisifies 
$$
   \fE(H) = |\yb| \big (\la \bm x, \nabla_{\bm x} \ra +m + \a\big)\fL_\bg^{d_1, (\xb)} H. 
$$
Substituting into \eqref{eq:diff-gH_LT} and collecting terms, we further deduce 
\begin{align*}
\fD_{\bk+\b \bep}^{\triangle, (\yb)} u  & -  2 \la \bm x, \nabla_{\bm x} \ra   \la \bm y, \nabla_{\bm y} \ra u 
 - (|\bk'|+ d_2) \la \bm x, \nabla_{\bm x} \ra u \\
  & =  - n (n + |\bk|+ \b +d_2) u  + m(m+\a)u  - |\yb| \big (\la \bm x, \nabla_{\bm x} \ra +m + \a\big)\fL_\bg^{d_1, (\xb)} u. \notag
\end{align*}
Furthermore, using \eqref{eq:eigen_wrapLT}, which implies $m - |\yb| \fL_\bg^{d_1, (\xb)} =
   \la \xb, \nabla_\xb\ra - \fL_\bg^{d_1, (\xb)}$ when applied on $u$, so that we deduce by using \eqref{eq:eigen_wrapLT}
one more time
\begin{align*}
  & m(m+\a)u  - |\yb| \big (\la \bm x, \nabla_{\bm x} \ra +m + \a\big)\fL_\bg^{d_1, (\xb)} u \\
 & \qquad =  - |\yb| \la \bm x, \nabla_{\bm x} \ra \fL_\bg^{d_1, (\xb)} u + \a \left(\la \xb, \nabla_\xb\ra - \fL_\bg^{d_1, (\xb)}\right) u
  + m\left(\la \xb, \nabla_\xb\ra - \fL_\bg^{d_1, (\xb)}\right) u \\
  & \qquad 
   =(1-|\yb|) \left(\fL_\bg^{d_1, (\xb)}\right)^2u -\fL_\bg^{d_1, (\xb)} \la \xb, \nabla_\xb\ra  u 
     - \la \xb, \nabla_\xb\ra \fL_\bg^{d_1, (\xb)} u   - \alpha \fL_\bg^{d_1, (\xb)} u \\
    & \qquad\quad + \la \xb, \nabla_{\xb} \ra^2 u + \a  \la \xb, \nabla_{\xb} \ra u. 
\end{align*}
Together, the last two identities show that the spectral operator $\fD_{\bg,\kb}^{L\rtimes \triangle}$ is given by
\begin{align*}
 \fD_{\bg,\bk}^{L\rtimes \triangle} & = \fD_{\bk+\b \bep}^{\triangle, (\yb)}   
     -  2 \la \bm x, \nabla_{\bm x} \ra\la \bm y, \nabla_{\bm y} \ra    - (|\bk|+\b+ d_2) \la \bm x, \nabla_{\bm x} \ra  \\
 & -  (1-|\yb|) \left(\fL_\bg^{d_1, (\xb)}\right)^2 + \fL_\bg^{d_1, (\xb)} \la \xb, \nabla_\xb\ra+\la \xb, \nabla_\xb\ra \fL_\bg^{d_1, (\xb)} 
  +\a \fL_\bg^{d_1, (\xb)} - \la \xb, \nabla_{\xb} \ra^2.   
\end{align*}
Using the notation $\fL_\bk^{d, (\yb)}$, the operator $\fD_{\bk+\b \bep}^{\triangle, (\yb)}$ can be written as, by 
\eqref{eq:Dsimplex2}, 
$$
  \fD_{\bk+\b \bep}^{\triangle, (\yb)} =  \fL_\bg^{d_2, (\yb)} - (|\bk| +\b+d_2)\la \yb, \nabla_{\yb}\ra - \la \yb, \nabla_{\yb} \ra^2.
$$ 
Consequently, combining terms and simplifying, we obtain the expression in \eqref{eq:LSimplex} and 
the identity \eqref{eq:LSimplex-eigen}. 
\end{proof}

\section{Wrapped product orthogonal polynomials on quadratic surfaces}
\setcounter{equation}{0}

In this section we consider orthogonal polynomials defined on a quadratic surface, such as the unit
sphere $\sph$ of $\RR^d$. The first subsection contains definitions, including the wrapped product
surfaces and weights. Two new families of wrapped product orthogonal polynomials
that possess a spectral operator are discussed in the second and the third subsections. 

Throughout this section, we adopt the convention of using letters in Sans-serif font, such as $\sW$ 
and $\sb$, for functions and related constants on surfaces, and retain those in bold font for solid
domains. 

\subsection{Orthogonal polynomials on quadratic surfaces}

A domain $\Omega_0$ is called a quadratic surface if $\Omega_0 = \{\xb \in \RR^d: \phi(\xb) = 0\}$, 
where $\phi$ is a quadratic polynomial. Let $\sW$ be a nonnegative weight function defined on 
$\Omega_0$, so that the bilinear form
$$
  \la f,g \ra_\sW = \int_{\Omega_0} f(\xb) g(\xb) \sW(\xb) \d \tau(\xb),
$$
where $\d \tau$ denotes the Euclidean surface measure on $\Omega_0$, is a well-defined inner product 
in $L^2(\Omega_0, \sW)$. To consider orthogonal polynomials with respect to this inner product,
we work with polynomial functions restricted to $\Omega_0$, or equivalently with the appropriate 
quotient of the polynomial ring. Let $\Pi_n(\Omega_0)$ be the subspace of polynomials of degree 
at most $n$. It has the dimension 
$$
   \dim \Pi_n\left(\Omega_0 \right) = \binom{n + d -1}{n} + \binom{n + d -2}{n-1}. 
$$
Let $\CV_n^d(\Omega_0, \sW)$ denote the subspace of orthogonal polynomials of degree $n$ 
with respect to the inner product $\la \cdot, \cdot\ra_W$. Since $\phi$ is quadratic, 
\begin{equation}\label{eq:dimV0n}
    \dim \CV_n^d(\Omega_0,W) =  \binom{n+d-2}{n}+ \binom{n+d-3}{n-1}. 
\end{equation}

The most well-studied orthogonal polynomials on quadratic surfaces are spherical harmonics
on the unit sphere $\sph$ of $\RR^d$, which is the surface of the unit ball $\BB^d$ and corresponds 
to $\phi(x) = 1-\|\xb\|^2$. The studies of orthogonal structure on other quadratic surfaces are modeled
after spherical harmonics; cf. \cite{OX, X20, X21}. 

\subsubsection{Spherical harmonics} 

A spherical harmonic is a homogeneous polynomial that satisfies $\Delta Y =0$. Let $\CH_n(\sph)$ 
denote the space of spherical harmonics of degree $n$ in $d$-variables. Its dimension is given 
by \eqref{eq:dimV0n}. For $n \in \NN_0$, let $\{Y_\ell^n: 1 \le \ell \le \dim \CH_n(\sph)\}$ be an 
orthonormal basis of $\CH_n(\sph)$; then 
$$
   \frac{1}{\o_d} \int_\sph Y_\ell^n (\xi) Y_{\ell'}^m (\xi)\d\s(\xi) = \delta_{\ell,\ell'} \delta_{m,n},
$$
where $\o_d = {2 \pi^{\f{d}{2}}}/{\Gamma(\f{d}{2})}$ denotes the surface area of $\sph$. 
An explicit basis of $\CH_n^d$ can be given in spherical coordinates in terms of the 
Jacobi polynomials (see, for example, \cite[p. 116]{DX}). 

In terms of the spherical polar coordinates $\bm x = r \xi$, $r \ge 0$ and $\xi \in \sph$, the Laplace operator satisfies
$$
  \Delta = \frac{\d^2}{\d r^2} + \frac{d-1}{r} \frac{\d}{\d r} + \frac{1}{r^2} \Delta_0,
$$
where $\Delta_0$ is the Laplace-Beltrami operator $\Delta_0$ on the sphere, which is the spectral operator of harmonics; more
precisely (cf. \cite[(1.4.9)]{DaiX}), 
\begin{equation} \label{eq:sph-harmonics}
     \Delta_0 Y = -n(n+d-2) Y, \qquad Y \in \CH_n^d. 
\end{equation}
We refer to \cite[Section 1.4]{DaiX} for an explicit expression of $\Delta_0$. The operator can be decomposed in terms of 
angular derivatives $D_{i,j}$ defined in \eqref{eq:Dij}; that is, 
\begin{equation} \label{eq:Delta0=}
   \Delta_0 = \sum_{1 \le i < j \le d} D_{i,j}^2. 
\end{equation}
The operators $D_{i,j}$ are skew-adjoint, hence $\Delta_0$ is self-adjoint on $\sph$. In particular, it follows 
\begin{equation} \label{eq:Dij-integral}
  \int_{\sph} \Delta_0 f(\xi)\cdot g(\xi) \d \sigma(\xi) = - \int_{\sph} \sum_{1 \le i<j \le d} D_{i,j} f(\xi) D_{i,j} g(\xi) \d \sigma(\xi). 
\end{equation}

\subsubsection{Wrapped product orthogonal polynomials}
We consider the surface of the wrapped domain $\Omega^{d_1} \rtimes \Omega^{d_2}$ and restrict our 
consideration to $\Omega_1^{d_1} = \BB^{d_1}$, which has the boundary $[\Omega_1^{d_1}]_0 = \SS^{d_1-1}$
being the unit sphere. 
Let $\Omega_2^{d_2}$ be a domain in $\RR^2$. We consider the wrapped quadratic surface  
$$
  \SS^{d_1-1} \rtimes \Omega_2^{d_2} := \left\{(\xb,\yb): 
      \frac{\xb}{\rho(\yb)}  \in \SS^{d_1-1}, \quad \yb \in \Omega_2^{d_2}\right\},
$$
where $\rho$ is a polynomial of degree at most two as in Definition \ref{defn:wrap}. Let $\d\s_\SS $ be the 
surface measure on $\SS^{d_1 -1}$ and let $\Wb$ be defined on the domain $\Omega_2$. Parametrizing
the domain as  
$$
 \SS^{d_1-1} \rtimes \Omega_2^{d_2} =\left \{(\rho(\yb) \xi, \yb): \xi \in \sph, \, \yb \in \Omega_2^{d_2} \right\},
$$
the integral over this domain is equal to 
\begin{align}\label{eq:int_surface}
   \int_{\SS^{d_1-1} \rtimes \Omega_2^{d_2}}&  f(\xb,\yb) \Wb(\yb)\d \tau (\xb,\yb) \\
    &  = \int_{\SS^{d_1-1}} \int_{\Omega_2^{d_2}} f(\rho(\yb) \xi, \yb)  \Wb(\yb) \rho(\yb)^{d_1-1} \sqrt{1+|\nabla(\rho(\yb)|^2}
        \d \yb \d \s(\xi). \notag
\end{align} 
Indeed, for the parametrization $F(\xi,\yb) = (\rho(\yb) \xi, \yb)$, a spherical tangent vector $\vb \perp \xi$ maps
to $\rho(\yb) v, 0)$, while $\partial_{y_j} F = (\partial_j \rho(y) \xi, e_j)$. As these two groups of tangent vectors
are orthogonal, the metric determinant is 
$$
   \rho^{2(d_1-1)} \det \left(I + \nabla \rho \nabla \rho^T\right) = \rho^{2(d_1-1)} \big(1+|\nabla \rho|^2\big). 
$$

Following the notation of the previous sections, we consider orthogonal polynomials with respect to the 
inner product 
\begin{equation}\label{eq:ipd_surface}
  \la f, g\ra_\sW = \int_{\SS^{d_1-1} \rtimes \Omega_2^{d_2}} f(\xb,\yb) g(\xb,\yb) 
      \sW^{\SS \rtimes \Omega_2}(\xb,\yb) \d \tau (\xb,\yb),
\end{equation}
where the weight function is defined by 
\begin{equation}\label{eq:weight_surface}
     \sW^{\SS \rtimes \Omega}(\yb) =\frac{ \Wb(\yb)}{ \sqrt{1+|\nabla(\rho(\yb)|^2}}, \qquad \yb \in \Omega_2^{d_2}.  
\end{equation}
By \eqref{eq:int_surface}, the inner product $ \la \cdot, \cdot\ra_\sW$ is well defined for all polynomials 
restricted to $ \SS^{d_1-1} \rtimes \Omega_2^{d_2}$. Let $\CV_n \big(\SS^{d_1-1} \rtimes \Omega_2^{d_2}, \sW^{\SS \rtimes \Omega_2} \big)$ be the space of 
orthogonal polynomials of degree $n$ with respect to this inner product. A basis of orthogonal polynomials 
for this space can be derived from spherical harmonics and orthogonal polynomials on $\Omega_2^d$. 


Let $\big\{\Rb_{\jb, n}^{(m)}: |\jb|=n\big\}$ be an orthogonal basis for $\CV_n (\rho^{2m+d_1-1} \Wb, \Omega^{d_2})$ 
and with the squared norm of $\Rb_{\jb, n}^{(m)}$ denoted by $\hb_{\jb, n-m}\big(\rho^{2m+d_1-1}\Wb)$, and let 
$\{\sY_\ell^m: 1 \le \ell \le \dim \CH_m^{d_1}\}$ be an orthonormal basis for $\CH_n^{d_1}$. 

\begin{prop}\label{prop:OP_Wrap0}
Let polynomials $\sQ_{\jb, \kb, m}^n$ of $d_1+d_2$ variables be defined by
\begin{equation*}
   \sQ_{\jb, \ell, m}^n(\xb, \yb) = \Rb_{\jb,n-m}^{(m)}(\yb) [\rho(\yb)]^m \sY_{\ell}^m \left(\frac{\xb}{\rho(\yb)}\right),  \quad
   (\xb,\yb) \in \SS^{d_1-1} \rtimes \Omega_2^{d_2},
\end{equation*}
where $0 \le m \le n$, $|\jb| = n-m$ with $\jb \in \NN_0^{d_2}$, $1\le \ell \le \dim \CH_m^{d_1}$, and $(\xb,\yb)\in
\SS^{d_1-1} \rtimes \Omega_2^{d_2}$. 
Then the set $\left\{\sQ_{\jb, \ell, m}^n: 1 \le \ell \le \dim \CH_m^{d_1}, \,  |\jb| = n-m, \, \jb \in \NN_0^{d_2}, 0 \le m \le n\right\}$ 
is an orthogonal basis of $\CV_n\big(\SS^{d_1-1} \rtimes \Omega^{d_2},\sW^{\SS \rtimes \Omega_2}\big)$.
Moreover, the squared norm defined by 
$\sh_{\jb, \ell, m}^{n, \SS \rtimes \Omega_2} 
= \la  \sQ_{\jb, \ell, m}^n, \sQ_{\jb, \ell, m}^n\ra_\sW$, is equal to 
$
\hb_{n-m}\big(\rho^{2m+d_1}\sW^{\SS\rtimes \Omega_2} \big).
$
\end{prop}

\begin{proof}
Making a change of variables $\xb \mapsto \rho(\yb) \xi$, $\xi \in \SS^{d_1-1}$, and apply \eqref{eq:int_surface},
we obtain
\begin{align}\label{eq:int_surface2}
   \la f, g \ra_{\sW} = 
     \int_{\SS^{d_1-1}} \int_{\Omega_2^{d_2}} f(\rho(\yb) \xi, \yb) g(\rho(\yb) \xi, \yb)\Wb(\yb) \rho(\yb)^{d_1-1}  
       \d \yb \d \s(\xi). 
\end{align} 
The proof of orthogonality follows straightforwardly from this identity and the orthogonality of 
$\Rb_{\jb,n-m}^{(m)}$ as well as $\Yb_\ell^m$. 
\end{proof}

We are interested in the surfaces of wrapped product on which orthogonal polynomials possess a
spectral operator. Since $\SS^{d_1 -1}$ is the surface of $\BB^{d_1}$, we consider the cases when
the orthogonal polynomials on the solid domain $\BB^{d_1} \rtimes \Omega_2$ possess a spectral
operator, which corresponds to the two cases in \eqref{eq:class1}. These are surfaces given by
$$
\SS^{d_1-1}\!\! \rtimes \triangle^{d_2} \qquad \hbox{and}\qquad \SS^{d_1-1}\!\! \rtimes \RR_\Sigma^{d_2}. 
$$
In both cases, $\frac{\xb}{\rho(\yb)} \in \Omega_1$ becomes $\|\xb\| = \rho(\yb)$, so that
$$
 \SS^{d_1-1} \rtimes \Omega_2 =\{(\xb,\yb) \in \RR^{d_1}\times \RR^{d_2}:  \|\xb\| = \rho(\yb)\}.
$$


\subsection{Sphere-Simplex orthogonal polynomials}
Let $d_1$ and $d_2$ be two positive integers and $d_1 \ge 2$. We consider the domain 
defined by
$$
\SS^{d_1-1} \!\rtimes \triangle^{d_2} = \left\{(\bm x,\bm y) \in \RR^{d_1} \times \RR_+^{d_2}: \|\bm x\| = 1- |\bm y| \le 1\right\},
$$
is a quadratic surface in $\RR^{d_1+d_2}$. In this case, $\rho(\yb) = 1-|\yb|$, so that
$\sqrt{1+|\nabla \rho|^2} = \sqrt{1+d_2}$ in \eqref{eq:int_surface}. For $\bm\kappa \in \RR^{d_2}$ such 
that $\k_i > -1$, $1 \le i \le d$ and $\g > -1$, we define  
$$
  {\sf W}_{\g,\bk}^{\SS \rtimes \triangle}(\bm y) =  \prod_{i=1}^{d_2} |y_i|^{\k_i} (1-|\bm y|)^{\g}, \qquad \yb \in \triangle^{d_2},
$$
and consider orthogonal polynomials with respect to the inner product defined by 
\begin{equation} \label{eq:ipd_ST0}
  \la f, g\ra_{\g, \bk}^{\SS \rtimes \triangle} = \sb_{\g,\bk}^{\SS \rtimes \triangle} 
   \int_{\SS^{d_1-1} \!\rtimes \triangle^{d_2}} f(\bm x,\bm y) g(\bm x,\bm y) 
       \sW_{\g,\bk}^{\SS \rtimes \triangle}(\bm y) \d \s(\bm x, \bm y), 
\end{equation}
where $\sb_{\g,\bk}^{\SS \rtimes \triangle} $ is the  normalization constant so that 
$ \la 1, 1\ra_{\g, \bk} =1$ and $\d \s(\bm x, \bm y)$ is the
surface measure of $\SS^{d_1-1} \!\rtimes \triangle^{d_2}$. The inner product is well defined on the space 
$\Pi(\SS^{d_1-1} \!\rtimes \triangle^{d_2})$ of polynomials in $d_1+d_2$ variables that are restricted on the surface 
$\SS^{d_1-1} \!\rtimes \triangle^{d_2}$.

Let $\CV_n(\SS^{d_1-1}\! \rtimes \triangle^{d_2}, \sW_{\g,\bk}^{\SS \rtimes \triangle})$ be the space of orthogonal 
polynomials of degree $n$ with respect to the inner product $\la \cdot, \cdot\ra_{\g, \bk}^{\SS\rtimes \triangle}$, 
which has the dimension \eqref{eq:dimV0n}. If $d_2 =1$, then the domain with $(d_1,d_2) = (d,1)$ is the rotaional 
conic surface,  
$$
  \sph \!\rtimes [0,1] = \left\{(\bm x,y) \in \RR^d \times \RR_+: \|\bm x\| = 1- y, \quad 0 \le y \le 1\right\}, 
$$
on which the orthogonal polynomials for $\sW_{\g,\k}$ have been studied in \cite{X20, X21}, where the 
variables were chosen as $(\bm x,t)$, which corresponds to $t=1-y$ for $0 \le t \le 1$. 

\begin{rem}
We assume $d_1 > 1$ since if $d_1=1$, then the surface 
is defined by $|x| = 1-y_1-\cdots - y_d$, where we write $d = d_2$, which consists of two mirroring simplices in
$\RR^{d+1}$ joined at the face $x=0$ and $|\bm y|=1$. For $d =2$, it is two triangular planes in $\RR^3$, defined by
$\pm x = 1-y_1-y_2$, $(y_1, y_2) \in \triangle^2$, that are joined on a line, and it has vertices at 
$(1,0,0), (-1,0,0), (0,1,0), (0,0,1)$ in $(x,y_1,y_2)$ coordinates. Our wrapped product construction of an 
orthogonal basis does not apply to this case. Moreover, for $d =1$, the domain becomes two line segments 
intersecting at one point, for which a family of orthogonal polynomials is studied in \cite{OX1}, and they are not 
eigenfunctions of a spectral operator. 
\end{rem}

Parametrising the domain $\SS^{d_1-1} \!\rtimes \triangle^{d_2}$ by setting $\bm x = (1- | \bm y|) \xi$ 
with $\xi \in \sph$, it follows readily that 
\begin{equation*} 
  \int_{\SS^{d_1-1} \!\rtimes \triangle^{d_2}} f(\bm x, \bm y) \d\s(\bm x, \bm y) =  \int_{\triangle^{d_2}} (1-|\bm y|)^{d_1-1} \int_{\SS^{d_1-1}}
     f\big( (1-|\bm y|) \xi, \bm y\big) \d \s(\xi)\, \d \bm y;
\end{equation*}
moreover, under the same parametrization, 
$$   
  \sW_{\g,\bm \kappa}^{\SS \rtimes \triangle}(\bm y) \d \s(\bm x, \bm y)
    = (1-|\bm y|)^{\g+ d_1-1} \sW_{(\bk,0)}^\triangle (\bm y) \d\s(\xi) \d \bm y
    =  \sW_{(\bk,\g+d_1-1)}^\triangle (\bm y) \d\s(\xi) \d \bm y.
$$
From the above two identities, the normalization constant $\sb_{\g,\bk}$ in \eqref{eq:ipd_ST0} satisifies
\begin{equation} \label{eq:bV0}
  \sb_{\g,\bk} =  \frac{\bm b^\triangle_{(\bk, \g+d_1 -1)}}{\o_{d_1} \sqrt{1+d_2}},
\end{equation}
where $\o_d$ is the surface area of $\sph$.  

We again let $\{\bm T_{\bm j, m}^{\bg}: |\bm j| = m, \, \bm j\in \NN_0^{d_2}\}$ be the orthogonal basis for 
$\CV_{m}\left(\triangle^{d_2}, \Wb_{\bg}^\triangle\right)$ on the simplex, where $\bg \in \RR^{d+1}$, 
such as the one given in \eqref{eq:OP_TT}, and denote the squared norm of $\bm T_{\bm j, m}^{\bg}$ 
by $\bm h_{\bm j,m}^{\bg,\triangle}$. 
Let $\{\sY_\ell^m: 1 \le \ell \le  \dim \CH_m^{d_1}\}$ be an orthonormal basis of spherical harmonics in 
$\CH_m^{d_1}$. Then Proposition \ref{prop:OP_Wrap0} becomes the following. 

\begin{prop}
For $\g > - d_1$, let $\a = \g + d_1 -1$. Define 
\begin{equation} \label{eq:ST0_OP}
 \sQ_{\bm j, \ell, m}^n(\bm x,\bm y) = \bm T_{\bm j, n-m}^{(\bk, \alpha+2m)}(\bm y) (1-|\bm y|)^m 
      \sY_{\ell}^m \left(\frac{\bm x}{1-|\bm y|}\right). 
\end{equation}
Then $\{\sQ_{\bm j, \ell, m}^n: |\bm j| = n-m, \, 1 \le \ell \le \dim \CH_m^{d_1}, \, 0 \le m  \le n, \bm j \in \NN_0^{d_2}\}$ 
is an orthogonal basis of $\CV_n(\SS^{d_1-1}\!\! \rtimes \triangle^{d_2}, \sW_{\g, \bk}^{\SS\rtimes \triangle})$. Moreover, 
the squared norm of $\sQ_{\bm j, \ell, m}^n$ 
is given by 
\begin{equation} \label{eq:ST0_Norm}
   \sh_{m,n}^{\g,\bk}  
       =   \frac{\bm b^\triangle_{(\bk,\alpha)}}{\bm b^\triangle_{(\bk, \alpha + 2m)}} \bm h_{\bm j,n-m}^{(\bk, \a+ 2m), \triangle}
       = \frac{(\g +d_1)_{2m}} {(|\bk|+ \g + d_1 + d_2)_{2m}} \bm h_{\bm j,n-m}^{(\bk, \a+ 2m), \triangle}.  
\end{equation}
\end{prop}
 
The most important family of orthogonal polynomials on this wrapped domain turns out to be the one for the weight function
$\sW_{-1,\k}^{\SS\rtimes \triangle}$. Indeed, the spectral operator for $\sW_{\g,\k}^{\SS\rtimes \triangle}$ exists only when $\g = -1$, 
which is derived with the help of the explicit basis in \eqref{eq:ST0_OP}. 

\begin{thm}
Let  $\k_1,\ldots, \k_{d_2} > -1$. For $(\bm x,\bm y) \in \SS^{d_1-1} \!\rtimes \triangle^{d_2}$, 
let $\fD_{-1,\bk}^{\SS\rtimes \triangle}$ be the differential operator defined by 
\begin{align}\label{eq:diff-eqnST0}
  \fD_{-1,\bk}^{\SS \rtimes \triangle} : = &  
    \sum_{i=1}^{d_2} \left( y_i \partial_{y_i}^2 + (\kappa_i +1) \partial_{y_i} \right) - \la \bm y, \nabla_{\bm y} \ra^2 \\
         & - (|\bk| +d_1+d_2-2) \la \bm y, \nabla_{\bm y} \ra 
           + \frac{1}{1-|\bm y|} \Delta_0^{(\bm x)}, \notag
\end{align}
where $\Delta_0^{(\bm x)}$ is the Laplace-Beltrami operator on the unit sphere in $\bm x$ variable. 
Then $u \in \CV_n\big(\SS^{d_1-1} \!\rtimes \triangle^{d_2}, \sW_{-1, \bk}^{\SS \rtimes \triangle}\big)$ satisfies the 
differential equation 
\begin{align}\label{eq:eigenST0}
             \fD_{-1,\bk}^{\SS \rtimes \triangle} u = - n \big( n+ | \bk| + d_1 +d_2 -2 \big)u.
\end{align}
\end{thm}

\begin{proof}
The proof follows that of Theorem \ref{thm:B-Tri}  but is simpler. We consider $u =\sQ_{\bm j, \ell, m}^n$. Since 
$\sY_\ell^m$ is homogeneous, setting $\bm x = (1-|\bm y|) \xi$ for $\xi \in \sph$, we write 
$u(\bm x, \bm y) = g(\bm y) H(\xi, \bm y)$ with 
$$
   g(\bm y) = \bm T_{\bm j, n-m}^{(\bk, \alpha +2m)} (\bm y), \quad \hbox{and}\quad
      H(\xi, \bm y) =   (1-|\bm y|)^m \sY_\ell(\xi)
$$
where $\a = d_1-2$. For $\xb = r \xi$, we have $\la \xb, \nabla_\xb \ra \sY_\ell (\xi) = r \frac{\partial}{\partial r}  \sY_\ell (\xi)  =0$,
which implies readily that $\la \xb, \nabla_\xb\ra u =  (1-|\yb|) g \la \xb, \nabla_\xb\ra  \sY_\ell(\xi) = 0$. 
Using the notation \eqref{eq:fD+} and following the proof of Theorem \ref{thm:B-Tri} 
almost verbatim, we obtain 
\begin{align*}
\fD_{(\bk, \a)}^\triangle u  &\, = \fD_{(\bk, \a+2 m )}^\triangle g \cdot H - (\a + m) g \la \bm y, \nabla_{\bm y} \ra H  
     - m (|\bk| +d_2) g H \\
 & =  - (n-m) (n+ m + |\bk|+ \a +d_2) u - m  (|\bk|+ d_2) u + m(\a+m) \frac{|\bm y|}{1- |\bm y|} u.
\end{align*}
Now, $- m(\a+m) = - m(m+d_1-2)$ is the eigenvalue of the operator $\Delta_0^{(\bm x)}$, it follows that 
$ - m(\a+m) u = \Delta_0^{(\bm x)} u$, which leads to, when substituting it into the previous identity and rearranging
the resulting identity,  
\begin{align} \label{eq:diffST0}
\fD_{(\bk, \a)}^\triangle u + \frac1{1-|\bm y|} \Delta_0^{(\bm x)} u =  - n (n + |\bk | + \a +d_2) u,
\end{align}
which is \eqref{eq:eigenST0}, and the formula \eqref{eq:diff-eqnST0} follows from the \eqref{eq:fD+} and the last 
identity in the proof of Theorem  \ref{thm:B-Tri}.  
\end{proof}

If we do not assume $\g = -1$, then the proof can still be carried out halfway with $\a = \g + d_1 -1$ but $m$
will not be completely absobled in the eigenvalues of $\Delta_0^{(\xb)}$. As a result, the eigenvalues will 
depend on both $n$ and $m$. Consequently, the space 
$\CV_n\big(\SS^{d_1-1} \!\rtimes \triangle^{d_2}, \sW_{\g, \bk}^{\SS \rtimes \triangle}\big)$ is a eigenspace,
with an eigenvalue depending only on $n$, holds only if $\g = -1$. 

When $d_2 =1$ and $t=1-y$, the operator \eqref{eq:diff-eqnST0} coincides with the spectral operator established for
the Jacobi polynomials on the rotational cone in \cite{X20}. As an analog of Theorem \ref{cor:diff-eqnV}, we can 
also state a more structural form given below.

\begin{cor} \label{cor:ST0}
Let $\k_1,\ldots, k_{d_2} > -1$. The spectral operator $\fD_{-1, \bk}^{\SS \rtimes \triangle}$ can be written as 
\begin{align} \label{eq:diff-eqnST02}
  \fD_{-1, \bk}^{\SS \rtimes \triangle} & = \frac{1}{1-|\bm y|} \Delta_0^{(\bm x)} 
      + \frac{1}{\sW_{-1,\bk}^{\SS \rtimes \triangle}(\bm y)} 
        \sum_{1\le i< j \le d_2} (\partial_{y_i} - \partial_{y_j})  y_i y_j \sW_{-1,\bk}^{\SS \rtimes \triangle}(\bm y)
   (\partial_{y_i} - \partial_{y_j}) \notag \\
   & +  \frac{1}{(1-|\bm y|)^{d_1-1}} \frac{1}{\sW_{-1,\bk}^{\SS \rtimes \triangle}(\bm y)}  \sum_{i=1}^{d_2}  
\partial_{y_i} (1-|\bm y|)^{d_1} y_i \sW_{-1,\bk}^{\SS \rtimes \triangle}(\bm y) \partial_{y_i}. 
 \end{align}
In particular, it follows that 
\begin{align} \label{eq:self-adjV0} 
   & - \int_{\SS^{d_1-1} \!\rtimes \triangle^{d_2}}\fD_{-1, \bk}^{\SS \rtimes \triangle}f (\bm x,\bm y) \cdot g(\bm x, \bm y)
     \sW_{-1, \bk}^{\SS \rtimes \triangle}(\bm x, \bm y) \d \s(\bm x, \bm y)  \\ 
  &   =  \int_{\SS^{d_1-1} \!\rtimes \triangle^{d_2}}  \frac{1}{1-|\bm y|} 
      \sum_{1 \le i< j \le d_1}  D_{i,j}^{(\bm x)} f(\bm x, \bm y)  D_{i,j}^{(\bm x)} g(\bm x, \bm y) 
  \sW_{-1,\bk}^{\SS \rtimes \triangle}(\bm x,\bm y)  \d \s(\bm x, \bm y)    \notag \\ 
  & + \int_{\SS^{d_1-1} \!\rtimes \triangle^{d_2}}   \sum_{1\le i< j \le d_2}  y_i y_j (\partial_{y_i} - \partial_{y_j}) f(\bm x,\bm y) \cdot 
    (\partial_{y_i} - \partial_{y_j})g(\bm x,\bm y)
  \sW_{-1, \bk}^{\SS \rtimes \triangle}(\bm x,\bm y)  \d \s(\bm x, \bm y)    \notag \\
    &  +  \int_{\SS^{d_1-1} \!\rtimes \triangle^{d_2}} \sum_{i=1}^{d_2} y_i (1-|\bm y|)  \partial_{y_i} f(\bm x,\bm y) \partial_{y_i} g(\bm x,\bm y) 
       \sW_{-1,\bk}^{\SS \rtimes \triangle}(\bm x,\bm y)   \d \s(\bm x, \bm y).  \notag
 \end{align}
 \end{cor}

\begin{proof}
By \eqref{eq:diff-eqnST0}, the two sums in the right-hand side of \eqref{eq:diff-eqnST0} are deduced
from \eqref{eq:DkSimplex2} for $\fD_{\{\kb,\a\}}^\triangle$ and, since $\a = d_1-2$,
$$
   \Wb_{\kb,\a}^\triangle(\yb) = \prod_{i=1}^{d_2} y_i^{\k_i} (1-|\yb|)^{d_1-2} = 
      \sW_{-1, \bk}^{\SS \rtimes \triangle}(\yb) (1-|\yb|)^{d_1-1}.
$$
The integral identity follows from integration by parts, see \eqref{eq:DkSimplex3}, whereas the first term 
on the right-hand side follows from the decomposition of $\Delta_0$ at \eqref{eq:Delta0=} and the 
self-adjointness of the operator $D_{i,j}$ on $L^2(\sph)$. 
\end{proof}

As in the case of the solid domain, the identity \eqref{eq:self-adjV0} leads to the Bernstein inequality on 
$\SS^{d_1-1} \!\rtimes \triangle^{d_2}$. Let $\|\cdot \|_\bk$ denote the norm of $L^2\big(\SS^{d_1-1} \!\rtimes \triangle^{d_2}, \sW_{-1,\bk}\big)$. 

\begin{thm} 
Let $\k_i > -\f12$, $1 \le i \le d_2$. Then for $f$ being a polynomial of degree at most $n$ on $\SS^{d_1-1} \!\rtimes \triangle^{d_2}$, 
\begin{align}\label{eq:B2_0}
 \sum_{1 \le i< j \le d_1} & \left \| \frac{1}{\sqrt{1-|\bm y|}}  D_{i,j}^{(\bm x)} f \right \|_{\bk}^2   + 
  \sum_{1\le i< j \le d_2} \left\| \sqrt{y_iy_j}(\partial_{y_i} - \partial_{y_j}) f \right\|_{\bk} ^2 \\
    & + 
   \sum_{i=1}^{d_2} \left \| \sqrt{y_i  (1-|\bm y|)}  \partial_{y_i} f\right\|_{\bk} ^2 
   \le   n(n+ |\bk| + d_1 + d_2-2)\|f\|_{\bk}^2. \notag
\end{align}
Moreover, the inequality is sharp in the sense that the equality is attainable by some polynomial of degree $n$, and 
so are the next two inequalities,
\begin{align}\label{eq:B2_0A}
 \sum_{1 \le i< j \le d_1}  \left \| \frac{1}{\sqrt{1-|\bm y|}}  D_{i,j}^{(\bm x)} f \right \|_{\bk}^2   
    & +   \sum_{i=1}^{d_2} \left \| \sqrt{y_i  (1-|\bm y|)}  \partial_{y_i} f\right\|_{\bk} ^2 \\
    & \le   n(n+ |\bk| + d_1 + d_2-2)\|f\|_{\bk}^2, \notag
\end{align}
\begin{align}\label{eq:B2_0B}
  \sum_{1\le i< j \le d_2} \left\|  \sqrt{y_iy_j}(\partial_{y_i} - \partial_{y_j}) f \right\|_{\bk} ^2 & + 
   \sum_{i=1}^{d_2} \left \| \sqrt{y_i  (1-|\bm y|)}  \partial_{y_i} f\right\|_{\bk} ^2 \\
  & \le   n(n+ |\bk| + d_1 + d_2-2)\|f\|_{\bk}^2. \notag
\end{align}
\end{thm}

\begin{proof}
Setting $g = f$ in \eqref{eq:self-adjV0}, the proof of the inequality \eqref{eq:B2_0} follows as in the proof 
of Theorem \ref{thm:Bernstein}, from which the two other inequalities \eqref{eq:B2_0A} and \eqref{eq:B2_0B}
follow readily. The inequality \eqref{eq:B2_0} is evidently sharp for all polynomials in 
$\CV_n(\SS^{d_1-1} \!\rtimes \triangle^{d_2}, \sW_\bk)$, 
whereas \eqref{eq:B2_0A} is sharp for $f(\bm x, \bm y) =  (1-|\bm y|)^n Y_\ell^n(\frac{x}{1-|\bm y|})$, which 
is $\sQ_{\bm 0, \ell, n}^n$ given in \eqref{eq:ST0_OP}, as $(\partial_{y_i} - \partial_{y_j}) f =0$, and \eqref{eq:B2_0B} 
is sharp for $f(\bm x, \bm y) = \bm T_{\bm j, n}^{(\bk, \alpha)}(\bm y)$, which is $\sQ_{\jb, 0, 0}^n$ in \eqref{eq:ST0_OP}.
\end{proof}

\subsection{Sphere-Laguerre polynomials on the quadratic surface}
Let $d_1$ and $d_2$ be two positive integers and $d_1 \ge 2$. We consider the domain defined by
$$
\SS^{d_1-1} \!\rtimes \RR_\Sigma^{d_2} = \left\{(\bm x,\bm y) \in \RR^{d_1} \times \RR_\Sigma^{d_2}: \|\bm x\| = |\bm y|\right\},
$$
which is a quadratic surface in $\RR^{d_1+d_2}$. In this case $\rho(\yb) = |\yb|$. It follows that 
$\sqrt{1+|\nabla \rho|^2} = \sqrt{1+d_2}$ in \eqref{eq:int_surface}. For $\bm\kappa \in \RR^{d_2}$ such that $\k_i > -1$, 
$1 \le i \le d_2$, we define  
$$
  {\sf W}_{\bk}^{\SS \rtimes \Sigma} (\bm y) =  \prod_{i=1}^{d_2-1} |y_i|^{\k_i} |\bm y|^{\k_{d_2}} \e^{-|\yb'|-|\yb|}, 
  \qquad \yb \in \RR_\Sigma^{d_2},
$$
and consider orthogonal polynomials with respect to the inner product defined by 
\begin{equation} \label{eq:ipd_SL0}
  \la f, g\ra_{\bk}^{\SS\rtimes \Sigma}   = \sb_{\bk}^{\SS \rtimes \Sigma}  \int_{\SS^{d_1-1} \!\rtimes \triangle^{d_2}} f(\bm x,\bm y) g(\bm x,\bm y) 
       \sW_{\bk}^{\SS \rtimes \Sigma}  (\bm y) \d \s(\bm x, \bm y), 
\end{equation}
where $\sb_{\bk}^{\SS \rtimes \Sigma}$ is the  normalization constant so that $ \la 1, 1\ra_{\bk} =1$ and $\d \s(\bm x, \bm y)$ is the
surface measure of $\SS^{d_1-1} \!\rtimes \RR_\Sigma^{d_2}$. Parametrizing the integral by setting
$\xb = |\yb| \xi$, we obtain, 
$$
   \int_{\SS^{d_1-1} \!\rtimes \RR_\Sigma^{d_2}} f(\xb,\yb) \d \s (\xb,\yb) = 
   \int_{\RR_\Sigma^{d_2}} |\yb|^{d_1-1} \int_{\SS^{d_1-1}} f(|\yb| \xi, \yb) \d\s(\xi) \d \yb,
$$
which implies readily, where $\o_d$ denotes the surface area of $\sph$, that
$$
   \sb_{\bk}^{\SS \rtimes \Sigma}  = \frac{1}{\omega_{d_1} \sqrt{1+d_2}
       \prod_{i=1}^{d_2-1} \Gamma(\k_i +1) \Gamma(\k_{d_2} + d_1)}.
$$

Let $\CV_n(\SS^{d_1-1}\! \rtimes \RR_\Sigma^{d_2}, \sW^L_{\bk})$ be the space of orthogonal polynomials 
of degree $n$ with respect to the inner product $\la \cdot, \cdot\ra_{\bk}^{\SS\rtimes \Sigma}$, which has the
dimension \eqref{eq:dimV0n}. If $d_2 =1$, then the domain with $(d_1,d_2) = (d,1)$ is the rotaional conic surface,  
$$
  \sph \!\rtimes \RR_+ = \left\{(\bm x,y) \in \RR^d \times \RR_+: \|\bm x\| = y\right\}, 
$$
on which the orthogonal polynomials for $\sW_{\k}(t) = |t|^\k e^{-t}$ have been studied in \cite{X20}. 

Let $\{\hat \Lb_{\bm j, m}^{\bk}: |\bm j| = m, \, \bm j\in \NN_0^{d_2}\}$ be the orthogonal basis for 
$\CV_{m}\big(\RR_\Sigma^{d_2}, \Wb_{\bk}^\Sigma \big)$ and 
where $\bk \in \RR^{d}$, such as the one given in \eqref{eq:basis_S} and denote the squared norm of
$\hat \Lb_{\bm j, m}^{\bk}$ by $h_{\bm j,m}^{\bk}$. Let $\{\sY_\ell^m: 1 \le \ell \le  \dim \CH_m^{d_1}\}$
be an orthonormal basis of $\CH_m^{d_1}$. Then Proposition \ref{prop:OP_Wrap0} becomes the following.

\begin{prop}
For $\bk \in \RR^{d_2}$, $\k_i  >  -1$ for $1 \le i \le d_2$, let $\a = d_1-1$ and 
$\bep = (0,\ldots,0,1) \in \RR^{d_2}$. Define 
\begin{equation} \label{eq:SL0_OP}
 \sQ_{\bm j, \ell, m}^n(\bm x,\bm y) = \hat \Lb_{\bm j, n-m}^{\bk+ (\a+2m)\bep}(\bm y) |\bm y|^m 
      \sY_{\ell}^m \left(\frac{\bm x}{|\bm y|}\right). 
\end{equation}
Then $\{\sQ_{\bm j, \ell, m}^n: |\bm j| = n-m, \, 1 \le \ell \le \dim \CH_m^{d_1}, \, 0 \le m  \le n, \bm j \in \NN_0^{d_2}\}$ 
is an orthogonal basis of $\CV_n(\SS^{d_1-1}\!\! \rtimes \RR_\Sigma^{d_2}, \sW_{\bk}^L)$. Moreover, 
the squared norm of $\sQ_{\bm j, \ell, m}^n$ is given by 
\begin{equation} \label{eq:SL0_Norm}
   \sh_{m,n}^{\bk}  
       =   \frac{\sb^L_{\bk+ \alpha\bep}}{\sb^L_{\bk+(\alpha + 2m) \bep}} \bm h_{\bm j,n-m}^{\bk+(\a+ 2m)\bep, L}
       =  (\k_{d_2} +d_1)_{2m}  \bm h_{\bm j,n-m}^{\bk+(\a+ 2m)\bep, L}.  
\end{equation}
\end{prop}
 
As with the Sphere-Simplex orthogonal polynomials, the most important family among the Sphere-Laguerre wrapped products
is $\sW_{\k}^{\SS\rtimes \Sigma}$ with $\k_{d_2} = -1$. 

\begin{thm}
Let  $\k_1,\ldots, k_{d_2-1} > -1$ and $\k_{d_2} = -1$. For $(\bm x,\bm y) \in \SS^{d_1-1} \!\rtimes \RR_\Sigma^{d_2}$, 
let $\fD_{\bk}^{\SS \rtimes \Sigma}$ be the differential operator defined by 
\begin{align}\label{eq:diff-SL0}
  \fD_{\bk}^{\SS \rtimes \Sigma}:= &  \sum_{i=1}^{d_2} \left[ y_i \partial_{y_i}^2 + (\kappa_i +1-y_i) \partial_{y_i} \right]
      - 2 \la \bm y, \nabla_{\bm y} \ra \partial_{y_{d_2}}  + 2 |\yb|  \partial_{y_{d_2}}^2   \\
       & +(d_1 -d_2 - |\bk| - 1) \partial_{y_{d_2}} + |\bm y|^{-1} \Delta_0^{(\bm x)}, \notag
\end{align}
where $\Delta_0^{(\bm x)}$ is the Laplace-Beltrami operator on the unit sphere in $\bm x$ variable. 
Then 
\begin{align}\label{eq:eigenSL0}
  \fD_{\bk}^{\SS \rtimes \Sigma} u = - n u, \qquad
       \forall u \in \CV_n\big(\SS^{d_1-1} \!\rtimes \RR_\Sigma^{d_2}, \sW_{\bk}^{\SS \rtimes \Sigma} \big).
\end{align}
\end{thm}

\begin{proof} 
As before, we consider $u = \sQ_{\bm j, \ell, m}^n(\bm x,\bm y)$ and write it as $u = g H$ with
$$
  g(\yb) = \hat \Lb_{\bm j, n-m}^{\bk+ (\a+2m)\bep}(\bm y) \quad \hbox{and}\quad H(\xb,\yb) = |\bm y|^m 
      \sY_{\ell}^m \left(\frac{\bm x}{|\bm y|}\right). 
$$
Following the proof of Theorem \ref{thm:B-Sig}, \eqref{eq:D+EH} in particular, and using 
$\la \xb,\nabla_\xb\ra u =  g \la \xb,\nabla_\xb\ra H = 0$, we obtain
\begin{align} \label{eq:D+EH-SSL}
  \fD_{\bk+ \alpha \bep}^{\Sigma, (\yb)} u  =  -n u  + g \cdot \fE(H), 
\end{align}
where, since $\k_{d_2} = -1$, $\bk + \alpha \bep = (\bk', d_1 -2)$, $|\yb| \fE(H)$ satisfies 
$$
    |\yb| \fE(H) = - \la \xb,\nabla_\xb\ra H - (d_1-2) \la \xb,\nabla_\xb\ra H + m (m+d_1-2)H = m (m+d_1-2)H
$$
 Now, since $m(m+d_1-2)$ is the eigenvalue of the spherical harmonics
$\CH_m^{d_1}$, it follows from \eqref{eq:sph-harmonics} that $|\yb| \fE(H) = - \Delta_0^{(\xb)} H$, from which 
we obtain
\begin{equation}\label{eq:diff-SL02}
  \fD_{\bk+ (d_1-1) \bep}^{\Sigma, (\yb)} u  + |\yb|^{-1} \Delta_0^{(\xb)} u = -n u, \qquad \k_{d_2} =-1.
\end{equation}
The left-hand side is $\fD_\k^{\SS \rtimes \Sigma}$. By \eqref{eq:fDSig}, $\fD_{\bk+ \alpha \bep}^{\Sigma, (\yb)} =  \fD_{\bk}^{\Sigma, (\yb)} + \a \partial_{\yb_d}$, from which we deduce the formula for $\fD_\k^{\SS \rtimes \Sigma}$ 
in \eqref{eq:diff-SL0}. 
\end{proof}

We can also state a more structural expression for the operator $\fD_\bk^{\SS \rtimes \Sigma}$. 
\begin{cor} \label{cor:SSL0}
Let $\k_1,\ldots, \k_{d_2-1} > -1$ and $\k_{d_2} = -1$. The spectral operator $\fD_{\bk}^{\SS \rtimes \Sigma} $ can be 
written as 
\begin{align} \label{eq:diff-eqnSSL0}
  \fD_{\bk}^{\SS \rtimes \Sigma} & = 
   \frac{1}{\sW_{\bk}^{\SS \rtimes \Sigma} (\bm x, \bm y)}  
       \sum_{i=1}^{d_2-1} (\partial_{y_i} - \partial_{y_{d_2}})\left( y_i \sW_{\bk}^{\SS \rtimes \Sigma} (\bm x, \bm y) 
         (\partial_{y_i} - \partial_{y_d})\right) \\
   & +  \frac{1}{|\bm y|^{d_1-1}} \frac{1}{\sW_{\bk}^{\SS \rtimes \Sigma} (\bm x, \bm y)}     
        \partial_{y_{d_2}}  \big( |\bm y|^{d_1} \sW_{\bk}^{\SS \rtimes \Sigma} (\bm x, \bm y) \partial_{y_{d_2}} \big)
        + \frac{1}{|\bm y|} \Delta_0^{(\bm x)}. \notag 
 \end{align}
In particular, it follows that 
\begin{align} \label{eq:self-adjSSL0} 
    & - \int_{\SS^{d_1-1}\! \rtimes \RR_\Sigma^{d_2}}\fD_{\bk}^{\SS \rtimes \Sigma} 
    f (\bm x,\bm y) \cdot g(\bm x, \bm y) \sW_{\bk}^{\SS \rtimes \Sigma} (\bm y) \d \s(\bm x, \bm y)  \\ 
  &   =  \int_{\SS^{d_1-1}\! \rtimes \RR_\Sigma^{d_2}}  \frac{1}{|\bm y|}  \sum_{1 \le i< j \le d_1}  D_{i,j}^{(\bm x)} f(\bm x, \bm y)  D_{i,j}^{(\bm x)} g(\bm x, \bm y) 
  \sW_{\bk}^{\SS \rtimes \Sigma}(\bm y)  \d \s(\xb, \yb)    \notag \\ 
  & + \int_{\SS^{d_1-1}\! \rtimes \RR_\Sigma^{d_2}}   \sum_{i=1}^{d_2-1}  y_i (\partial_{y_i} - \partial_{y_{d_2}}) f(\bm x,\bm y) \cdot 
    (\partial_{y_i} - \partial_{y_{d_2}})g(\bm x,\bm y)
  \sW_{\bk}^{\SS \rtimes \Sigma}(\bm y)  \d \s(\bm x, \bm y)    \notag \\
    &  +  \int_{\SS^{d_1-1}\! \rtimes \RR_\Sigma^{d_2}}   \partial_{y_d} f(\bm x,\bm y) \partial_{y_{d_2}} g(\bm x,\bm y) 
       |\yb| \sW_{\bk}^{\SS \rtimes \Sigma}(\bm y)   \d \s(\bm x,\bm y).  \notag
 \end{align}
 \end{cor}

The identity \eqref{eq:diff-eqnSSL0} is derived from \eqref{eq:diff-SL02} and \eqref{eq:fD-L2} and the relation
$$
  \Wb_{\bk+\a \bep}^\Sigma(\yb) = \prod_{i=1}^{d_2 -1} y_i^{\k_i} |\yb|^{d_1-2} \e^{-|\yb|-|\yb'|} 
     = |\yb|^{d_1-1} \Wb_\bk^{\SS\rtimes \Sigma}(\yb),
$$
as $\a = d_1-1$, just as in the case of $\SS \rtimes \triangle$. The integral identity follows from a straightforward 
integration by parts. 

The identity \eqref{eq:self-adjSSL0} can also be used to derive sharp Bernstein inequalities
on the surface $\SS^{d_1-1}\! \rtimes \RR_\Sigma^{d_2}$. We leave the details to interested readers.

\appendix

\section{Spectral operators for $d =3$}

We give a list of second-order linear differential operators, in three variables, that have orthogonal polynomials 
as eigenfunctions. In each case, the list provides the weight function, the domain, and the spectral operator. The 
first four are product Hermite/Laguerre polynomials, and we write H = Hermite and L = Laguerre. 
 
\begin{enumerate}[   ]
\item[] HHH: $\Wb^\mathrm{H}(\xb) = \e^{- \|\xb\|^2}$, on $\RR^3$,
$$
      \sum_{i=1}^3 \big(\partial_{x_i}^2 - 2 x_i \partial_{x_i} \big) u = - 2n u.
$$
\item[] HHL: $\Wb_\a^\mathrm{HHL}(\xb) = \e^{- x_1^2 -x_2^2} x_3^\a \e^{-x_3}$, $\a > -1$, on $\RR^2 \times \RR_+$,
$$
    \partial_{x_1}^2 u + \partial_{x_2}^2 u + 2 x_3 \partial_{x_3}^2 u - 2 x_1 \partial_{x_1} u
        - 2 x_2 \partial_{x_2} u  + 2 (\a+1 - x_3) \partial_{x_3}  u = - 2 n u. 
$$
\item[] HLL: $\Wb_{\a,\b}^\mathrm{HLL}(\xb) = x_2^\a x_3^\b \e^{- x_1^2 -x_2-x_3}$, $\a,\b> -1$, on $\RR \times \RR_+^2$,
$$
    \partial_{x_1}^2 u + 2 x_2 \partial_{x_2}^2 u + 2 x_3 \partial_{x_3}^2 u - 2 x_1 \partial_{x_1} u+
          2 (\a+1 - x_2)\partial_{x_2} u + 2(\b+1 - x_3) \partial_{x_3} u  = - 2 n u.
$$
\item[]LLL:  $\Wb_\kb^\mathrm{LLL}(\xb) = x_1^{\k_1} x_2^{\k_2} x_3^{\k_3} \e^{-x_1-x_2 -x_3}$, $\k_1, \k_2, \k_3  > -1$, on $\RR_+^3$,
$$
    \sum_{i=1}^3 \big ( x_i \partial_{x_i}^2 + (\k_i+1 - x_i)  \partial_{x_i} \big ) u = - n u.
$$
\item Unit ball:  $\Wb_\mu^\BB(\xb) = (1-x_1^2-x_2^2-x_3^2)^{\mu-\f12}$, $\mu > -\f12$ on $\BB^3$,
$$
   \sum_{i=1}^3 \left[ (1-x_i^2) \partial_{x_i}^2 u - (2\mu+3) x_i \partial_{x_i} \right]u -2 \sum_{1\le i<j \le 3} x_i x_j \partial_{x_i x_j}^2 u   
 = - n(n+2\mu+2) u
$$
\item Simplex:  $\Wb_\bk^\triangle(\xb) = x_1^{\k_1} x_2^{\k_2} x_3^{\k_3} (1-x_1-x_2-x_3)^{\k_4}$, $\k_i > -1$, on $\triangle^3$,
\begin{align*}
   \sum_{i=1}^3  \left[ x_i (1-x_i) \partial_{x_i}^2 +\big(\k_i +1 - (|\bk| +4) x_i\big) \partial_{x_i} \right]u 
     &  -2 \sum_{1\le i<j \le 3} x_i x_j \partial_{x_i x_j}^2 u   \\
   &  = - n(n+|\bk| +3) u.
\end{align*}
\item Finite cone:  $\Wb_{\mu,\k}^{\BB\rtimes[0,1]}(\xb) = \left[(1- x_3)^2- x_1^2-x_2^2\right]^{\mu-\f12} x_3^{\k}$ on $\BB^2 \rtimes [0,1]$,
\begin{align*}
 & \big[ (1-x_1^2- x_3) \partial_{x_1}^2  +  (1-x_2^2- x_3)\partial_{x_2}^2  + x_3(1-x_3)\partial_{x_3}^2  \\
 &\quad  - 2 x_1 x_2 \partial_{x_1 x_2}    - 2 x_1 x_3 \partial_{x_1 x_3} 
   - 2 x_2 x_3 \partial_{x_3 x_2} 
  - (\k+2\mu +3) \left(x_1 \partial_{x_1} + x_2 \partial_{x_2} \right)  \\
&  \quad  + \big( \k+1 -  (\k+2\mu +3) x_3 \big) \partial_{x_3} \big] u = - n (n+\k + 2 \mu + 2) u. 
\end{align*}
\item Infinite cone: $\Wb_\bk^{\BB\rtimes \RR_+}(\xb) = \left[ x_3^2- x_1^2-x_2^2\right]^{\mu-\f12} \e^{-x_3}$
   on $\BB^2 \rtimes \RR_+$, 
\begin{align*}
   \big[ x_3 \big(\partial_{x_1}^2+\partial_{x_2}^2 + \partial_{x_3}\big) + 2 (x_1 \partial_{x_1}+x_2 \partial_{x_2}) \partial_{x_3}  
    &   -(x_1 \partial_{x_1}+x_2  \partial_{x_2}) u  \\
   & +(2 \mu+2-x_3) \partial_{x_3} \big] u = -n u. 
\end{align*} 
\end{enumerate}

We note that the last two cases are equations \eqref{eq:diff_eqnBT} and \eqref{eq:B-Sig} with $d_1 =2$ 
and $d_2 = 1$. The two equations also contain the cases $d_1 = 1$ and $d_2 =2$, which, however, are affine 
equivalent to the simplex  $\triangle^3$ and the Laguerre $\mathrm{LLL}$ as shown in Remarks \ref{rem:B-T}
and \ref{rem:B-Sig}. Furthermore, no new cases will appear from triple product/wrapped products, as we 
observed in Remark \ref{rem:d=3}. We formulate the result as a proposition. 

\begin{prop} 
For $d = 3$, the eight families listed above are the primary families, up to affine change of variables, 
among wrapped product orthogonal polynomials that are eigenfunctions of a second-order linear differential 
operator. 
\end{prop}

We conjecture that, up to affine transformation, the list contains the only second-order linear differential 
operators that have orthogonal polynomials as eigenfunctions when $d =3$.


\medskip
\noindent 
{\bf AI Acknowledgement}. The author used AI for proofreading with access through the ChatGPT for Academic 
Researchers program.

\end{document}